\pdfoutput=1
\documentclass[12pt]{iopart}

\usepackage{lipsum}
\usepackage{iopams}
\usepackage{graphicx}
\usepackage{epstopdf}
\usepackage{algorithmic}
\usepackage{enumitem,mathrsfs}
\usepackage{stackrel,amsthm}
\eqnobysec

\DeclareSymbolFont{symbolsC}{U}{pxsyc}{m}{n}
\DeclareMathSymbol{\coloneqq}{\mathrel}{symbolsC}{"42}

\newtheorem{thm}{Theorem}
\newtheorem{prop}{Proposition}

\newtheorem{rem}[thm]{Remark}

\newtheorem{applem}{Lemma}
\newtheorem{apprem}{Remark}

\def\bbb[#1]{\boldsymbol{#1}}
\def\mmm[#1]{\mathcal{#1}}
\def\sss[#1]{\mathscr{#1}}

\def\R{\mathbb{R}} 
\def\Z{\mathbb{Z}}
\def\N{\mathbb{N}} 

\def\dd{ \,\rmd}
\def\p{\partial}

\def\lap{\Delta}

\def\e{\varepsilon}

\def\dD{\p_D \Omega}
\def\dN{\p_N \Omega}

\usepackage{verbatim}
 
\begin{document}

\title[{Binary recovery via phase field regularization}]{Binary recovery via phase field regularization for first traveltime tomography}
\author{Oliver R. A. Dunbar $^1$, Charles M. Elliott$^2$} 
\address{Mathematics Institute, University of Warwick, UK, CV4 7AL} 
\ead{$^1$ o.dunbar.1@warwick.ac.uk, $^2$ c.m.elliott@warwick.ac.uk}

\begin{abstract}
We propose a double obstacle phase field methodology for binary recovery of the slowness function of an Eikonal equation found in first traveltime tomography. We treat the inverse problem as an optimization problem with quadratic misfit functional added to a phase field relaxation of the perimeter penalization functional. Our approach yields solutions as we account for well posedness of the forward problem by choosing regular priors. We obtain a convergent finite difference and mixed finite element based discretization and a well defined descent scheme by accounting for the non-differentiability of the forward problem. We validate the phase field technique with a $\Gamma$ -- convergence result and numerically by conducting parameter studies for the scheme, and by applying it to a variety of test problems with different geometries, boundary conditions, and source -- receiver locations.

\end{abstract}

\vspace{2pc}


\section{Introduction}\label{sec:intro}

We consider recovery of subsurface structure formed by a disjoint composition of two materials with distinct impedances to ground waves. First traveltime tomography (FTT) (see \cite{RawSam03}) is a subfield of seismic tomography where the observational data for recovery is the first hitting time of a wave at known locations. Other information (e.g. wave profile, amplitude) is unknown, so a full seismic wave description is inefficient and so we use an Eikonal ray approximation (\cite[Appendix C]{book:She09} or \cite[Chapter 4]{book:AkiRic02}). 

The link between the Eikonal equation and  first traveltime problems, has been widely studied  \cite{Son86,Lio85,CapLio90,CraLio83}, where it is shown that if the wave {\em impedance} or {\em slowness function} is  continuous (or with specific forms of jump discontinuities \cite{DecEll04}) then the description admits unique Lipschitz continuous viscosity solutions. The discontinuities arise at boundaries of material regions of different wave impedances \cite{HarZhu15}.  Numerical methods have been investigated to obtain unique viscosity solutions, typically using monotone finite difference schemes for stability. We use the Fast Marching Method (FMM) \cite{Set96,Set99,SetVla00,AdaSet95} due to its robustness \cite{HysTur05}  though fast sweeping methods \cite{Zha05,LanZha13,LuoQiaBur14,JeoWhi08} are also popular.

The inverse problem of recovering the slowness function given observed traveltimes is an underdetermined problem. We work in an optimization framework, and overcome determination issues by restricting optimization over a {\em prior} or {\em model subspace} of functions and adding to the misfit functional a regularization functional.  For piecewise constant slowness function we seek the interfaces between values defined by finite parameters \cite{ChaCotGauHei16,LinOrt17}, or with a finite basis \cite{DecEllSty11} or infinite dimensional subspace \cite{FarHurZhe13,LeuLi13}. One may  use a probabilistic frame \cite{Stu10,IglLuStu16,GhaMarStaTan13}, where the prior is a space of probability distributions, but this is not within our scope.  For binary recovery the solution of the inverse problem is a binary function. Unfortunately, the forward problem is not well posed for general binary minimizers due to the discontinuity. A natural regularization in the optimization framework is to penalize the size of the interfaces between constant values \cite{FatOshRud92,FarHurZhe13,LeuLi13} or to use other so-called `blocky' or `edge preserving' regularizations \cite{JiaZha16,JiaZha18,HarZhu15}. This perimeter regularization is difficult for binary priors because of non-differentiabilty unless one tracks the interfaces.

In this paper we address these two issues and  present a method based on phase field techniques \cite{BloEll93} for binary recovery.  The slowness is regularized using a phase field function and the perimeter regularization is approximated by a phase field gradient energy. This method has several attractive features. Primarily it is mathematically grounded, by producing continuous priors for the slowness function so the forward problem is well posed - this leads to an existence theory for solutions. It it is fit for purpose, as the regularization serves as an approximation to the perimeter regularization functional and so is naturally suited to binary recovery. We are able to construct a convergent discretization scheme based on finite difference method (forward problem) and a mixed finite element method (inverse problem).  Furthermore, the solution of the Eikonal equation is not differentiable with respect to the slowness. In practice, derivatives improve efficiency of schemes, as one may use an adjoint equation or write optimality conditions \cite{LeuQia06,ChaCalNobTai09}. It is shown in \cite{DecEllSty11} that for a particular finite difference scheme, a derivative of the discrete forward problem exists, and is computed efficiently by a variant of the FMM on the discrete adjoint equation. We implement the scheme, using an adjoint equation for the efficient calculation of (a well defined) discrete derivative, and showcase it with a parameter study and apply it to several varied test configurations. 

The power of our method comes from a notion of convergence of smooth minimizers to binary functions \cite{Mod87,BelBraRie05,BloEll91}, a well developed  numerical analysis  \cite{BloEll92,BlaGarSarSty11} and suitability for implementation. We take an obstacle phase field approach \cite{BloEll91,BloEll93} with an  $H^2$ regular functional \cite{HilPelSch02}, (for other choices \cite{CheDelFonLeo11,FonMan00}). The techniques have been applied to piecewise constant recovery for other inverse problems \cite{pre:DunEllHoaStu18,BreDedEll14,DecEllSty16,pre:BerRatVer17}.

\subsubsection{Outline}
We set out the forward problem in \Sref{sec:fp} and inverse problem \Sref{sec:ip}. We present the phase field regularization in \Sref{sec:PFform}. We discretize the problems in  \Sref{sec:discfp} and \Sref{sec:discip}, with attention given to the discrete derivative in \Sref{sec:discderiv}. We present a numerical scheme and investigate choice of parameters, and the scheme's effectiveness through computations in \Sref{sec:res}.

\subsection{Binary recovery}\label{sec:binrec}
\begin{figure}[ht]
        \centering
        \includegraphics[width=0.8\textwidth]{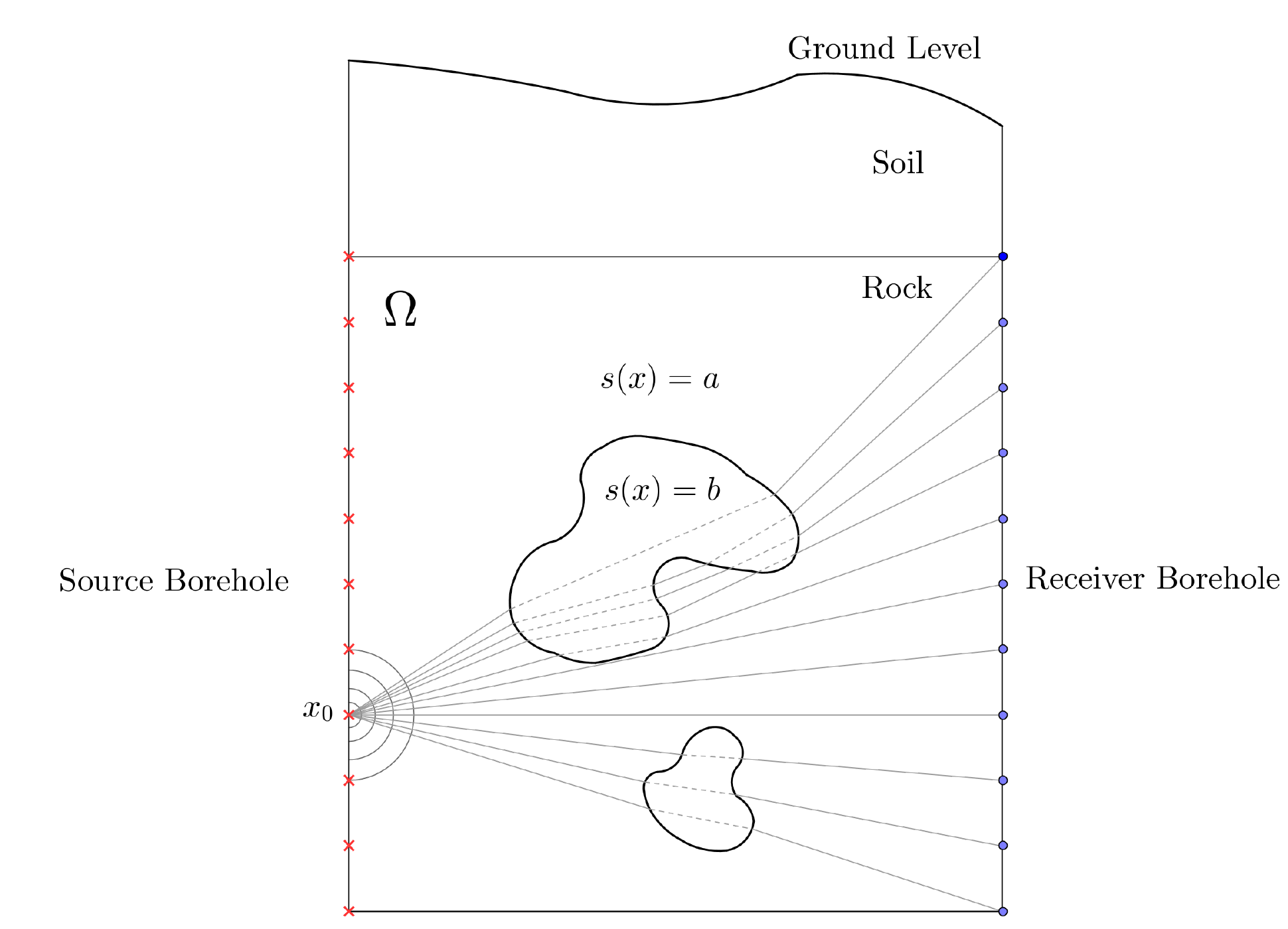}
       \caption{\label{fig:crosswell}Schematic showing crosswell tomography for binary recovery. The boundaries are drilled boreholes, the left populated with wave sources, the right with receivers. A source at $x_0$ is producing a wave that moves through a background medium with slowness $s(x)=a$, and through inclusions where $s(x)=b$. Drawing solution ray paths, one sees Snell's law at inclusion boundaries. The Soner boundary condition requires that rays leaving the domain $\Omega$ do not re-enter.}
       
\end{figure}
We motivate the choices we make in subsequent sections using the following model. Let $\Omega=[0,L_x]\times[0,L_z]\subset\R^2$. The domain is covered by two media, with known distinct slowness $0<a<b$. we then represent the slowness by a binary function, $s:\Omega \to \{a,b\}$. This representation arises in seismic tomography \cite{Gui11,HarZhu15}. The forward problem is the Eikonal equation \eref{eq:eik} -- \eref{eq:sonerbc} with appropriate boundary conditions. We know this yields a unique solution if jump discontinuities form  Lipschitz regular interfaces. 

We know location of sources and receivers, and one common configuration is {\em crosswell} or {\em cross hole} tomography, where two boreholes are drilled at $x=0$ (the source borehole), and $x=L_x$ (the receiver borehole). Sources are set of at regular intervals $[0,(i/M)L_z]$, and recorded at regular intervals $[L_x,(j/N)L_z]$ for $i=1,\dots,M$ and $j=1,\dots,N$. This scenario is depicted in \Fref{fig:crosswell}. The goal of the inverse problem is: given experimental data of hitting times of seismic waves at receivers, can one reconstruct the slowness $s:\Omega \to \{a,b\}$? 

\section{The abstract forward and inverse problems}\label{sec:ap}
\subsection{The forward problem}\label{sec:fp}
Let $\Omega\subset \R^d$ ($d=2$ or $3$), be an open bounded domain with Lipschitz boundary $\p \Omega$ and take $x_0\in\Omega$ fixed. We investigate the first arrival time at $x\in \bar\Omega$ of a ray originating from a source at $x_0$. Denote the space of possible ray paths by,
\[
\Xi_{x_0}(x) \coloneqq \{\,\xi \in W^{1,\infty}([0,1],\bar\Omega) \quad | \quad \xi(0)=x_0, \,\xi(1) = x\,\}. 
\]
We denote the impedance of the ray in the medium (representing subsurface heterogeneity) by defining a continuous and positive {\em slowness function} $s:\bar\Omega \to \R_+$. The {\em first traveltime} $T(x)$ over this set of paths is defined as
\begin{equation}\label{eq:optcont}
 T(x)\coloneqq\inf_{\xi \in \Xi_{x_0}(x)} \int_0^1 s(\xi(r))\, |\xi'(r)|\dd r.
\end{equation}
This extremal value is viewed as the shortest arrival time of a ray that obeys Fermat's principle and travels from $x_0$ to $x$, with speed $c(x)=s(x)^{-1}$.  It has been shown in  \cite{book:Lio82,Son86}, that $T(x)$ formally satisfies a stationary Hamilton-Jacobi equation, namely the following Eikonal equation: 
\begin{eqnarray}
 |\nabla T (x)| = s(x),& \qquad&\forall x \in \Omega \setminus \{x_0\},\label{eq:eik}\\
 T(x_0)=0,& &\label{eq:ptconstr}\\
 \nabla T(y) \cdot n(y) \geq 0,& & \forall y \in \p\Omega,\label{eq:sonerbc}
\end{eqnarray}
where $n$ is the outward pointing unit normal. We say $T$ is a viscosity solution of the Eikonal equation \eref{eq:eik} -- \eref{eq:sonerbc}. In the context of \eref{eq:optcont}, $\nabla T(x)$ is the direction of the optimal ray and $s(x)^{-1}$ is the speed of the ray at $x$. The point condition \eref{eq:ptconstr}  ensures the source at $x_0$ has zero travel time. The Soner boundary condition \eref{eq:sonerbc} ensures information propagates out of the domain, that is, all ray paths terminate at $\p\Omega$ \cite{Son86}. 
\subsection{Well posedness}\label{sec:eikwp}

The well posedness of \eref{eq:eik} -- \eref{eq:sonerbc} is of importance when considering the inverse problem associated to it. The goal of the theory is to clarify the requirements for this Eikonal equations to produce regular (Lipschitz) solutions.
The first proof of well posedness is from \cite{book:Lio82}, where the problem \eref{eq:optcont} was posed with Dirichlet boundary data (replacing \eref{eq:sonerbc}): 
\begin{equation}\label{eq:dirbc}
s = \varphi \textrm{ on }\p\Omega, \textrm{ where } |\varphi(x) -\varphi(y)| \leq \bar T(x,y),
\end{equation}
and $\bar T (x,y)$ is the traveltime between a point $x$ and $y$ in $\bar\Omega$.

The next development is found in \cite{Son86}, where it is shown that if one constrains paths to lie within $\bar\Omega$, a solution of \eref{eq:optcont} is a viscosity solution of \eref{eq:eik} -- \eref{eq:sonerbc}. By viscosity solution here, we mean that $T(x)$ is a subsolution of \eref{eq:eik} for $x \in\Omega\setminus\{x_0\}$ and a supersolution of \eref{eq:eik} for $x\in\bar \Omega\setminus\{x_0\}$, also used in \cite{DecEll04,DecEllSty11}. In \cite{Son86}, it was shown that if $s\in C^0(\Omega)$ is continuous, positive and bounded, there is a unique optimal value $T$, that is Lipschitz continuous on $\bar\Omega$ with constant bounded by $\|s\|_\infty$. The result holds under regularity conditions satisfied by bounded $C^1$ domains, or piecewise smooth boundaries containing isolated corners.

The results have been extended in \cite{DecEll04} to deal with discontinuous slowness function while still achieving Lipschitz continuous solutions. First note, the regularity requirement on $s$ in \cite{Son86} can be written: $\forall x,y \in \Omega$,
\begin{equation}\label{eq:contSon}
 |s(x)-s(y)| \leq w_u(\|x-y\|),
\end{equation}
 where $w_u$ is nondecreasing, continuous and $w_u(0)=0$, differences in $s$ are bounded by a continuous function. In \cite{DecEll04}, this is extended as follows:  $\forall x\in \Omega$, $\exists \e>0,\ n\in S^{n-1}$ such that $\forall y\in\Omega,\ r>0,\ d \in S^{n-1}$, with $|d-n|<\e$ and $y+rd \in \Omega$, such that 
\begin{equation}\label{eq:contDecEll}\fl 
\indent s(y+rd)-s(y) \leq w_s(\|x-y\| + \|y+rd - y\|) = w_s(\|x-y\|+r) = w_s^y (r),
\end{equation}
for the $(n-1)$ -- dimensional unit sphere $S^{n-1}$. The property \eref{eq:contDecEll} requires that at every $x\in\Omega$ one can choose a cone in direction $n$ such that within this cone, the slowness function is bounded by a nondecreasing continuous function $w_u(r)$ (as in \eref{eq:contSon}). This condition holds if $s$ is continuous, and allows for jump discontinuities along an interface where the left and right limits of the jump satisfy strict inequality along the length of the interface. In two dimensions, this can be extended to allow junctions of curves of discontinuity \cite{DecEll04}. 
\begin{rem}\label{rem:wp}
The extension is technical. In particular, functions in $W^{1,\alpha}(\Omega)$ or $BV(\Omega,\{a,b\})$ for $0<a<b$, will not automatically satisfy these cone conditions.
\end{rem}

\subsection{Nondifferentiability of the solution operator}

We note that the solution of \eref{eq:eik} is not differentiable in $s$. Differentiability of the forward problem is often an important feature to solve an inverse problem, as derivative informed methods are often far more efficient for formulating and searching for optimizers. We shall address this through the course of the derivation by making use of a discrete construction found in \cite{DecEllSty11}.

\subsection{The abstract inverse problem}\label{sec:ip}

Let $y$ be observations of the solution of a forward problem, from input slowness $s$. Assume observations are perturbed by (additive) noise $\eta$. The abstract inverse problem is stated as: Given the observed data $y$ of the problem $\mmm[G]$, 
find $s$ satisfying 
\begin{equation}\label{eq:absinv}
 y = \mmm[G](s) + \eta.
\end{equation}
We choose the forward problem to be \eref{eq:eik} -- \eref{eq:sonerbc}, acting on a slowness function $s$ and followed by an operator $K: T(s)\mapsto \mmm[G](s)\in\mmm[O]$ an observation space. The data $y=T_{obs}\in \mmm[O]$  is a set of observed first hitting times at fixed known receiver locations in $\bar \Omega$ or on densely in a region or curve in $\bar \Omega$. Altogether, $ \eta = y - \mmm[G](s) = T_{obs}-K(T(s))$.

We consider a {\em mismatch functional} $\mmm[I](v)$ to measure the {\em data misfit} $y-\mmm[G](s)$ (chosen in \Sref{sec:misfit}). Our goal is to find a minimizer of $\mmm[I]$ within a space $\mmm[A]$, known as the {\em prior} or {\em model space}. We write,
\begin{equation}\label{eq:minfunc}
\textrm{Find } s \coloneqq \min_{v\in\mmm[A]} \mmm[I](v).
\end{equation}
The choices of $\mmm[O]$ and $\mmm[A]$, affect the complexity of the inverse problem. In particular, \eref{eq:eik} -- \eref{eq:sonerbc} and $\mmm[I]$, must be well posed for any $s\in \mmm[A]$, as discussed in \Sref{sec:eikwp}.

\subsection{Misfit functional}\label{sec:misfit}
 For the misfit  we choose a quadratic functional and take the observations  $\mmm[O]$ to be  a subspace of a Hilbert space: 
\begin{equation}\label{eq:mismatch}
 \mathcal{I}(s) = \|y - \mmm[G](s)\|^2_{\mmm[O]} =\frac{1}{2}\|K(T(s)) - T_{obs}\|^2_\mmm[O]. 
\end{equation}
Often the measurements are assumed to be densely defined on union of curves $\Gamma \subset \bar\Omega$.  In this case, take $T_{obs}:\Gamma \to \R_{>0}$ for $T_{obs} \in \mmm[O] =L^2(\Gamma)$. The natural candidate for a misfit functional is
\begin{equation}\label{eq:mismatchbdry}
 \mathcal{I}(s) = \frac{1}{2}\int_{\Gamma}|T(s)(x) - T_{obs}(x)|^2 \dd s_x.
\end{equation}
Commonly $\Gamma\subset \p\Omega$, as observations are taken at surface seismic stations, or at detectors in a drilled well (see \cite{HarZhu15,FarHurLel12,LeuLiQia14} and \Fref{fig:crosswell}). 

We may also consider observations at a finite number of distinct points  $x_1,\dots,x_M \in \bar\Omega$ (see \cite{FarHurLel12}). In this case the natural functional is,
\begin{equation}\label{eq:mismatchpoints}
 \mathcal{I}(s) = \frac{1}{2}\sum_{i=1}^M|T(s)(x_i) - T_{obs}(x_i)|^2.
\end{equation}
In either case, the positions of the observations relative to a source will affect the recovery. We generalize to consider multiple experiments using travel times $T^1(s),\dots T^S(s)$ produced from sources $x_0^1,\dots x_0^S$ by summing over the functionals:
\[
 \mathcal{I}(s) = \frac{1}{2}\sum_{j=1}^S \|K(T^j(s)) - T^j_{obs}\|^2_\mmm[O],
\]
where $T_{obs}^j$ is a traveltime data set corresponding to the source $x_0^j$. This assumes independence of experiments $j=1,\dots S$. Hereafter, we assume a single source and observation space $\mmm[O] =L^2(\Gamma)$ for a boundary segment $\Gamma\subset \p\Omega$. 

\subsection{Noisy observations}
\label{sec:noise}
We often consider the observational data $y$ to be perturbed by a random variable, interpreted as observation imprecision. In the case of $M$ finite observations at discrete points, we introduce this randomness through a normal random variable $\eta\sim \mu = N(0,\Sigma)$ on $\R^M$, where $\Sigma \in \R^{M\times M}$ as a positive-definite covariance. 

The misfit \eref{eq:mismatchpoints} is naturally weighted with the confidence, for example choose $\Sigma=\nu^2 \mathbb{I}$, where $\mathbb{I}$ is the identity on $\R^{M\times M}$ and a constant $\nu$:
\[\fl
\|\Sigma^{-\frac{1}{2}}(y - \mmm[G](s))\|^2_{\mmm[O]} = \frac{1}{2}\sum_{i=1}^M|\Sigma^{-\frac{1}{2}}(T(s)(x_i) - T_{obs}(x_i))|^2 =\frac{1}{2\nu^2}\sum_{i=1}^M|T(s)(x_i) - T_{obs}(x_i)|^2.
\]

\section{Regularization}\label{sec:pfreg}

\subsection{Introduction}\label{eq:misfwp}

To motivate the phase field method, consider {\em Tikhonov} regularizations.
\begin{equation}\label{eq:tik}
 \mathcal{I}^{\textrm{\em Tik},\gamma}(s)= \frac{1}{2}\|y-\mmm[G](s)\|^2_\mathcal{O} + \delta \int_\Omega  |\nabla s |^\alpha \dd x,
\end{equation}
where $\delta>0$ and $\alpha = 1,2$. The simplest choice $\alpha=2$ is problematic because  $s\in \mmm[A]\subset H^1(\Omega)$ does not give a well posed forward problem. A natural regularization for a binary recovery problem is to penalize the interfaces between constant values \cite{FatOshRud92}, (also \cite{FarHurZhe13,HarZhu15}). If $s\in BV(\Omega,\{a,b\})$, we can interpret \eref{eq:tik} with $\alpha=1$ as the total variation of $s$, that is, the perimeter of the set $\{s = a\}$. This is difficult to approximate numerically. Commonly, the problem is relaxed by taking $s\in BV(\Omega)$ and losing the binary nature of the prior. None of these guarantee continuous solutions necessary for the well-posedness of the forward problem (see Remark \ref{rem:wp}). 

\subsection{Phase field formulation}
\label{sec:PFform}
{\em Phase field regularization} provides regular approximation (via gamma convergence \cite{Mod87,BloEll91}) to \eref{eq:tik} with $\alpha=1$ and binary priors.
We separate the notation for the slowness function and the {\em phase field variable} in the regularization. Denote the phase field variable by $u:\Omega \to [-1,1]$, and we solve \eref{eq:eik} -- \eref{eq:sonerbc} with
the  slowness function $s(u(x))$, linear in $u$,  and  $s:[-1,1] \to [s_{\min},s_{\max}]$
\begin{equation}
s(u)=u(s_{\max} - s_{\min})/2 +(s_{\max} + s_{\min})/2
\end{equation}
bounded by $0<s_{\min}<s_{\max}$.
As $u$ approaches $-1$ (or $1$), $s$ approaches $s_{\min}$ (or $s_{\max}$). 
 We take $u$ to be our new variable, and we use a Ginzburg-Landau type functional  requiring $u\in H^2(\Omega)\subset C^0(\Omega)$ so that the forward problem is well defined. Define the phase field functional: for $\sigma,\e,\gamma >0$,
\begin{equation}\label{eq:PFmismatchabs}\fl 
\indent
  \mathcal{I}^\gamma_\e(u) \coloneqq\frac{1}{2} \|y-\mmm[G](s(u))\|^2_\mmm[O]  + \underbrace{\sigma\int_\Omega \Big(
 \gamma\frac{\e^3}{2}(\Delta u)^2 +\frac{\e}{2}|\nabla u|^2 + \frac{1}{\e}\Psi(u) \Big)\dd x}_{\coloneqq \sigma \mathcal{J}_\e^\gamma(u)},
\end{equation}
where $\Psi(u)$ is given by the double obstacle potential:
\begin{equation}\label{eq:doubobs}
\Psi(u)\coloneqq \cases{\case{1}{2} (1-u^2), & if  $u\in [-1,1],$\\
\infty, & otherwise. \\}
\end{equation}
The behaviour of a minimizer of $\mmm[J]_\e^\gamma$ (defined in \eref{eq:PFmismatchabs}), for small $\epsilon$, is a function which favours taking the exact values $-1$ or $1$ in large regions of the domain due to the double obstacle potential \cite{BloEll93}. 
The minimizer changes between values continuously due to the gradient and Laplacian penalizations, across a thin interface characterized by the width parameter $\e>0$. 
The $(\Delta u)^2$ term  of $\mmm[J]_\e^\gamma$ in \eref{eq:PFmismatchabs}, weighted by $\gamma>0$ ensures $s(u) \in H^2(\Omega)\subset C^0(\Omega)$ see \Sref{sec:eikwp}). The use of such penalization is found in \cite{pre:DunEllHoaStu18,BurEspZep15}. 
Possible boundary conditions that may be imposed are as follows
\begin{enumerate}[label=(\roman*)]
 \item $u=1\textrm{ or } -1 \textrm{ on }\p\Omega,\ \frac{\p u}{\p n} =0 \textrm{ on }\p \Omega,$
 \item $ u=1\textrm{ or } -1  \textrm{ or }\dD,\ \frac{\p u}{\p n} =0 \textrm{ on } \dN,  \textrm{ for }\p\Omega=\dD \cup \dN,$
 \item no conditions imposed. 
\end{enumerate}
We interpret a Dirichlet condition on the phase field variable, as having knowledge of the value of the slowness there. A zero Neumann condition on the boundary imposes (in the absence of other information) that the interface must touch the boundary orthogonally there.
We choose to take (ii) with $\dD \cap \dN=\emptyset$. Our prior space is written:
\begin{equation}\label{eq:contAdef}
  \mmm[A] \coloneqq \Big\{ u\in H^2(\Omega)\ \Big|\ u=1\textrm{ or } -1   \textrm{ on } \dD, ~\frac{\p u}{\p n}=0  \textrm{ on } \dN           \Big\}. \\
\end{equation}
\begin{rem}\label{rem:findimprior}
Another choice is a finite dimensional prior as in \cite{DecEllSty11}. The problem is reduced to minimizing over a set of  coefficients of some given functions: Let $\{ \psi_k\} $ satisfy $0\leq \psi_k \in W^{1,\infty}(\Omega)$ $\forall i=1,\dots,K$ and $\sum_{i=k}^K \psi_k = 1$ and their support covers $\bar\Omega$. Define
 \begin{equation}\label{eq:Adef}\fl 
\indent
 \mmm[A] \coloneqq \{\ s\colon \bar \Omega \to \R \ | \ s(x) = \sum_{k=1}^K s_k \psi_k (x), \ s_k\geq 0 \textrm{ bounded, and } \phi_k\geq 0 \textrm{ continuous}\ \},
\end{equation}
The authors showed that when discretized appropriately, the discrete forward problem (based on \eref{eq:eik} -- \eref{eq:sonerbc}) can attain a derivative with respect to the state variables $s_i$. We exploit this setting later to obtain a discrete derivative.
\end{rem} 
Define $\mmm[I]^\gamma_\e$ by \eref{eq:PFmismatchabs}, then our problem is:
\begin{equation}\label{eq:invprob}
 \textrm{Find } u \coloneqq \stackrel[v \in \mathcal{A} ]{}{\arg \min}\ \mmm[I]^\gamma_\e(v),
\end{equation}

\begin{thm}\label{thm:minexist}
 Let $\Omega\subset \R^d$ be bounded with $\p\Omega$ Lipschitz and consider $\mmm[I]^\gamma_\e$ defined in \eref{eq:PFmismatchabs} and $\mmm[A]$ defined in \eref{eq:contAdef}. Then there exists a solution to the minimization problem \eref{eq:invprob}.
 \end{thm}
 
 \begin{proof}
  We wish to use  \cite[Theorem 9.3.1]{book:Cia13} to prove minimizers exist. We require $\mmm[A]$ to be weakly sequentially closed in $H^2(\Omega)$, and that the functional \eref{eq:PFmismatchabs}:
  \[
 \mmm[I]^\gamma_\e (u) = \frac{1}{2}\|y - \mmm[G](s(u))\|^2_\mmm[O] +\sigma \mmm[J]^\gamma_\e(u),
\]
is coercive and lower semicontinuous over $\mmm[A]$. Firstly, $\mmm[A]$ is an unbounded, weakly closed subset of $H^2$. Due to \eref{eq:doubobs}, a minimizer will not satisfy  $u(\Omega)\not\subset[-1,1]$  so we work over $\tilde\mmm[A]=\mmm[A]\cap H^2(\Omega\, ;[-1,1])$. Over $\tilde\mmm[A]$, we see 
\[\fl 
\indent
 \sigma \mmm[J]^\gamma_\e(u) = \sigma\int_\Omega \Big(
 \gamma\frac{\e^3}{2}(\Delta u)^2 +\frac{\e}{2}|\nabla u|^2 + \frac{1}{2\e}(1-u^2) \dd x \\
 \geq \sigma \gamma \frac{\e^3}{2}||\lap u||_{L^2(\Omega)}^2 +\frac{\e}{2}||\nabla u||_ {L^2(\Omega)}^2
\]
and $\| u\|^2_{L^2(\Omega)} \leq C_u\leq |\Omega|$. From the Poincar\'{e} inequality we obtain  $\mmm[J]^\gamma_\e$ is coercive on $\tilde\mmm[A]$. Weak convergence in $H^2(\Omega)$ implies strong convergence of $u,\ \nabla u$ in $L^2(\Omega)$ and weak convergence of $\lap u$ in $L^2(\Omega)$. Quadratic functionals are weak lower semicontinuous, and so $\mmm[J]^\gamma_\e$ is too. 

The data misfit is nonnegative and so immediately $\mmm[I]^\gamma_\e$ is coercive. For lower semicontinuity of the misfit we consider a weakly convergent sequence $(u_k)$ in $H^2(\Omega)$ to $u \in \mmm[A]$. As $u_k\mapsto \mmm[G](s(u_k))$ is well defined and Lipschitz continuous, then the trace onto $\p\Omega$ is (at least) in $\mmm[O]=L^2(\p\Omega)$. Furthermore, $u_k$ converges strongly in $H^1(\Omega)$, and so by the trace theorem for Lipschitz domains \cite[Theorem A8.6]{book:Alt16}, $u_k$ converges strongly on $\p\Omega$ in $L^2(\Omega)$. The continuity of $\mmm[G](s(u_k))$ preserves the limit, thus $\mmm[G](s(u_k)$ converges to $\mmm[G](s(u))$ in $\mmm[O]$ and the misfit (and so $\mmm[I]^\gamma_\e$) is weakly sequentially lower semicontinuous. 
We then apply \cite[Theorem 9.3.1]{book:Cia13} to complete the proof.
\end{proof}

\subsection{Gamma convergence}\label{sec:gamma}

The strength of the phase field technique lies with the convergence of the phase field functional  to the perimeter functional in the sense of $\Gamma$ -- convergence. 
as the interfacial parameter $\e \to 0$. The first result  \cite{Mod87} was for the Ginzburg-Landau functional (with $H^1$ minimizers) with a quartic double well potential $\Psi(u) = \frac{1}{2}(1-u^2)^2$. More recently this has been extended to a functional with $H^2$ minimizers \cite{HilPelSch02} (similar results in \cite{CheDelFonLeo11,FonMan00}). The analysis of \cite{Mod87} was extended to the double obstacle potential \eref{eq:doubobs} in \cite{BloEll91} by investigation of profile which appear across interfaces of minimizers. With this technique, we extend the analysis of \cite{HilPelSch02} to $\mathcal{J}^\gamma_\e$ in \eref{eq:PFmismatchabs} with a double obstacle \eref{eq:doubobs}. We denote $\int_\Omega |\nabla u|\dd x$ for $u\in BV(\Omega)$ to be the total variation of $u$.

\begin{thm}\label{thm:gammaconv}
 Let $\Omega \subset \R^d$ with Lipschitz boundary. Define the following functionals:
\begin{equation}\label{eq:gconvint}
 \mathcal{J}^\gamma_\e (u) \coloneqq\int_\Omega \gamma\frac{\e^3}{2}(\Delta u)^2 +\frac{\e}{2}|\nabla u|^2 + \frac{1}{\e}\Psi(u)\dd x,
\end{equation}
and
\begin{equation}\label{eq:gconvint1d}
 j^\gamma (z) \coloneqq \int_{\R}\frac{1}{2}\gamma (z'')^2 + \frac{1}{2}(z')^2+\Psi(z) \dd x,
\end{equation}
where $\Psi$ is the double obstacle potential \eref{eq:doubobs}. The $\Gamma$ -- limit:  $\mathcal{J}^\gamma_0 \coloneqq \stackrel[\e\to 0]{}{\Gamma - \lim } \mathcal{J}^\gamma_\e$,
is proportional to the perimeter functional:
\[
 \mathcal{J}^\gamma_0(u) =   \cases{
   \frac{1}{2}P^\gamma \int_{\Omega} |\nabla u | \dd x & if  $u \in BV(\Omega,\{-1,1\}),$\\
   \infty &  if  $u \in L^1(\Omega) \setminus BV(\Omega,\{-1,1\})$.
   }
\]
for 
\[
BV(\Omega,\{a,b\})\coloneqq \{\ w \in BV(\Omega)\; \colon \; w(\Omega) \subset \{a,b\}\ \}.
\]
The constant $P^\gamma \coloneqq \stackrel[v \in {\mmm[V]}]{}{\inf} j^\gamma(v)$ is the double obstacle transition energy, where
\begin{eqnarray*}\fl 
\indent
 \mmm[V] = \Big\{ v\in C^2(\R ; [-1,1] ) \ |& \  \exists \delta>0, \forall x \in \R  \ v(x)= -v(-x), \\
 &\ v'(x)\geq 0, \textrm{ and }  v(x>\delta)=1,\ v(x<-\delta)=-1\Big\}.
\end{eqnarray*}

\begin{proof}
 We consign the proof of this theorem to \ref{sec:app}
\end{proof}

\end{thm}
\begin{rem}\label{rem:gamma1}
 The relationship between \eref{eq:gconvint} and \eref{eq:gconvint1d} is through an ansatz on the phase field variable $u$:
 \[
 u(x) = z^\gamma\Big(\frac{d(x)}{\e}\Big),\qquad z(0)=0,
 \]
where $d(x)$ is the signed distance function to the limiting interface at $\{u(x)=0\}$, and $z^\gamma:\R \to [-1,1]$ is a smooth function found in \ref{sec:app}. Inserting the ansatz into \eref{eq:gconvint} results in \eref{eq:gconvint1d}.
\end{rem}

\begin{rem}
For $\mmm[I]^\gamma_\e$ in \eref{eq:PFmismatchabs}, an equivalent result to Theorem \ref{thm:gammaconv} is unknown, due to technicalities in assuring that misfit functional is well defined in the $\e\to 0$ limit.
\end{rem}

\subsection{Mixed formulation}\label{sec:mix}
In practice, approximating the high order derivatives as found in the problem \eref{eq:invprob} in the discrete setting is complex when using  conforming finite elements. We therefore create a mixed formulation. We begin by introducing a new variable $w$ such that $w=-\lap u$ weakly on $\Omega$, and setting: 
 \begin{equation}\label{eq:mixPFmismatchabs}
 \mmm[I]^\gamma_{\e}(u,w) =\frac{1}{2} \|y-\mmm[G](s(u))\|_\mmm[O]^2 + \underbrace{ \sigma\int_\Omega  \gamma\frac{\e^3}{2}w^2 +\frac{\e}{2}|\nabla u|^2 + \frac{1}{\e}\Psi(u)  \dd x}_{ \sigma\mmm[J]^\gamma_\e(u,w)}.
\end{equation}
 The natural optimization problem is then 
\begin{equation}\label{eq:optcontmix}
 \textrm{Find } (u,w)\coloneqq \stackrel[(\tilde u,\tilde w) \in {\mmm[A]_{\Delta}}]{} {\arg\min} \ \mmm[I]^\gamma_{\e}(\tilde u,\tilde w). 
\end{equation}
We find a suitable set $\mmm[A]_\Delta$ with appropriate boundary conditions and treat the cases $\dD \neq \emptyset$ and $\dD =\emptyset$ separately. 
We define a bilinear form $B:H^1(\Omega)\times H^1(\Omega) \to \R$, 
\[
 B(z,v)= \int_\Omega \nabla z \cdot \nabla v\dd x,
\]
and denote the $L^2$ inner product by $\langle\cdot,\cdot\rangle$. Define the spaces
\begin{eqnarray}
H^1_D &= \{ u \in H^1(\Omega) \ | \ u=g \textrm{ on } \dD\},\label{eq:h1d}\\
H^1_{0,D} &= \{ u \in H^1(\Omega) \ | \ u=0 \textrm{ on } \dD\}.\label{eq:h10d}
\end{eqnarray}
then we may define 
\begin{equation}\label{eq:mixAdef}
 \mmm[A]_\Delta \coloneqq \cases{ 
                              \mmm[A]_{\Delta,D}, & if $\dD \neq \emptyset,$ \\
                              \mmm[A]_{\Delta,N}, &if  $\dD = \emptyset,$
                          }
\end{equation}
where, 
\begin{eqnarray}\label{eq:mixADdef}\fl 
\indent
 \mmm[A]_{\Delta,D} \coloneqq \Big\{ (u,w) \in H^1_D(\Omega) \times L^2(\Omega)\ \Big|\  B(u,\zeta) = \langle w,\zeta \rangle,\  \forall \zeta \in H^1_{0,D}(\Omega) \Big\}, \\\fl 
\indent
 \mmm[A]_{\Delta,N} \coloneqq \Big\{  (u,w)\in H^1(\Omega) \times L^2(\Omega)\ \Big|\  \int_\Omega w \dd x= 0,\nonumber \\
 \phantom{\coloneqq \Big\{  (u,w)\in H^1(\Omega) \times L^2(\Omega)\ \Big|\ \ \ } B(u,\zeta) = \langle w,\zeta \rangle,\  \forall \zeta \in H^1(\Omega) \Big\}.\label{eq:mixANdef} 
\end{eqnarray}
\begin{rem}
The  compatibility condition $\int_\Omega w \dd x =0$ in  \eref{eq:mixANdef} is necessary to make sense of the relation between $u$ and $w$ in the case of the  pure Neumann boundary condition, $\p u/\p n=0$ on $\p\Omega$. 
\end{rem}
We now show existence of solutions for the mixed formulation.
\begin{thm}\label{thm:mixminexist}
 Let $\Omega\subset \R^d$ be bounded with  $\p\Omega$  either being    $C^2$ or  $\p\Omega$  Lipschitz and $\Omega$ convex. Consider $\mmm[I]^\gamma_\e$ defined in \eref{eq:mixPFmismatchabs} and $\mmm[A]_\Delta$ as in \eref{eq:mixAdef}. Then there exists a solution to the minimization problem \eref{eq:optcontmix}
\end{thm}
\begin{proof}
The assumptions on $\Omega$ imply that if $(u,w)\in  \mmm[A]_\Delta$ then standard results for $-\Delta u=w, w\in L^2(\Omega),u\in H^1(\Omega)$ in elliptic regularity, (see \cite[Chapter 2]{book:Gri11} for convex Lipschitz domain),  imply that $u\in H^2(\Omega)$ and is continuous so that the forward problem has a solution and the 
misfit functional is well defined. Consider $\dD \neq \emptyset$. The set $\mmm[A]_{\Delta,D}$ is an unbounded, weakly closed subset of $H^1(\Omega) \times L^2(\Omega)$. We seek to apply the theorem of \cite[Theorem 9.3.1]{book:Cia13}, and so must show the functional \eref{eq:mixPFmismatchabs} is coercive and lower semicontinuous over $\mmm[A]_{\Delta,D}$. For coercivity, by definition of \eref{eq:doubobs}, we are done if $u(\Omega)\not\subset[-1,1]$, and so work over $\tilde \mmm[A]_{\Delta,D}=\mmm[A]_{\Delta,D} \cap \big( H^1(\Omega,[-1,1]) \times L^2(\Omega)\big)$ with norm $\|(u,w)\|_{\Delta} = (\|u\|^2_{H^1(\Omega)}+\|w\|^2_{L^2(\Omega)})^{\frac{1}{2}}$.

We see that over $\tilde\mmm[A]_{\Delta,D}$ using \eref{eq:doubobs}, we have 
\begin{eqnarray*}
\sigma \mmm[J]^\gamma_\e(u,w) &= \sigma\int_\Omega  \gamma\frac{\e^3}{2}w^2 +\frac{\e}{2}|\nabla u|^2 + \frac{1}{2\e}(1-u^2) \dd x \\
&\geq \sigma \int_\Omega\gamma\frac{\e^3}{2}w^2 +\frac{\e}{2}\Big(|\nabla u|^2+u^2\Big)\dd x - (\frac{1}{2\e}+\frac{\e}{2}) \int_\Omega u^2 \dd x\\
&\geq \min(\sigma\frac{\e^3}{2},\frac{\e}{2})\|(u,w)\|^2_{\Delta} -(\frac{1}{2\e}+\frac{\e}{2})C_u,
\end{eqnarray*}
where $\|u\|^2_{L^2(\Omega)}\leq C_u \leq |\Omega|$. Therefore  $\mmm[J]^\gamma_\e$ is coercive on $\tilde \mmm[A]_{\Delta,D}$. Moreover, it is sequentially lower semicontinuous on $\mmm[A]_{\Delta,D}$ as weak convergence in $H^1(\Omega) \times L^2(\Omega)$ implies $\nabla u$ and $w$ weakly, and $u$ strongly converge, and quadratic functionals are weakly lower semicontinuous.

The data misfit is nonnegative and so immediately $\mmm[I]^\gamma_\e$ is coercive. 
For lower semicontinuity of the misfit we consider a weakly convergent sequence $(u_k,w_k)$ in $\mmm[A]_{\Delta,D}$. Thus $u_k\mapsto \mmm[G](s(u_k))$ is well defined and Lipschitz continuous, then the trace onto $\p\Omega$ is (at least) in $\mmm[O]=L^2(\p\Omega)$.

The relationship of \eref{eq:mixADdef}  holds for any $\zeta \in H^1_{0,D}(\Omega)$. One can choose $g\in H^1(\Omega)$ s.t $u_k-g=v_k\in H^1_{0,D}(\Omega)$ and choose $\zeta = v_{k}$. Then:
\begin{eqnarray*}
 \int_{\Omega} |\nabla u_k|^2\dd x =  \int_{\Omega} \nabla u_k \cdot \nabla v_k + \nabla u_k \cdot \nabla g \dd x  = \int_{\Omega} w_{k}v_k + \nabla u_k\cdot \nabla  g\dd x \\
 \to \int_{\Omega} w v +\nabla u\cdot \nabla g \dd x = \int_{\Omega} |\nabla u|^2 \dd x ,\qquad \textrm{ as } k\to \infty
\end{eqnarray*}
and so $u_k$ converges strongly in $H^1$. By the trace theorem for Lipschitz domains \cite[Theorem A8.6]{book:Alt16} we have that $u_k$ converges strongly on $\p\Omega$ in $L^2$. The continuity of $\mmm[G](s(u_k))$ preserves the limit, thus $\|\mmm[G](s(u_k)\|_\mmm[O]\to\|\mmm[G](s(u))\|_\mmm[O]$ and the misfit (and so $\mmm[I]^\gamma_\e$) is weakly sequentially lower semicontinuous. 

The proof for the case $\dD =\emptyset$ is similar.
\end{proof}

We now assert that the mixed formulation yields solutions for the original problem:
\begin{thm}\label{thm:mixequiv}
 Let $\Omega\in\R^d$, bounded with $\p\Omega$ either being $C^2$ or with $\p\Omega$ Lipschitz and $\Omega$ convex. Let $(u,w)\in \mmm[A]_\Delta$ defined by \eref{eq:mixAdef}, be solutions of the minimization problem \eref{eq:optcontmix} with $\mmm[I]_\e^\gamma(u,w)$ defined by  \eref{eq:mixPFmismatchabs} . Then, $u$ is also a solution of the minimization problem \eref{eq:optcont} with $\mmm[I]^\gamma_\e(u)$ defined by \eref{eq:PFmismatchabs} and $\mmm[A]$ defined by \eref{eq:contAdef}.
\end{thm}
\begin{proof}
 Let $(u,w)\in \mmm[A]_\Delta$. Standard regularity theorems for elliptic problems, gives that $u$ is not only in $H^1_D(\Omega)$ (resp. $H^1(\Omega)$), but is actually $u\in H^2_{loc}(\Omega)$. Furthermore, the boundary regularity is sufficient to extend this result to the boundary (see \cite[Chapter 2]{book:Gri11} for convex Lipschitz domain). In this case $w=-\lap u$ is well defined in $L^2(\Omega)$, so we show the boundary conditions are preserved. If $\dD \neq\emptyset$, we obtain through the relation that for any $\xi \in H^1_{0,D}(\Omega)$:
\[
  \int_\Omega w \xi \dd x =\int_\Omega \nabla u \cdot \nabla \xi \dd x = -\int_\Omega \lap u  \xi \dd x + \int_{\p\Omega_N} \xi\nabla u \cdot n \dd y,
\]
 and so as $w=-\lap u$ in $\Omega$, $\xi$ is arbitrary on $\p\Omega_N$ we obtain that $\p u/\p n=0$ on $\p\Omega_N$.If $\dD=\emptyset$, and $\int_\Omega w \dd x=0$, we know from the elliptic relation that $\forall \xi \in H^1(\Omega)$,
 \[
  \int_\Omega w \xi \dd x = \int_\Omega \nabla u \cdot \nabla \xi \dd x = -\int_\Omega \lap u \xi \dd x + \int_\Omega \xi \nabla u \cdot n \dd y, 
 \]
and setting $\xi=1$, and $w=-\lap u$ in $\Omega$ we obtain $\p u /\p n = 0$ on $\p\Omega$, So $u \in\mmm[A]$.  We have shown that for both sets of boundary conditions, we have $u\in \mmm[A]$ and so,
 \[
  \inf_{(\bar u,\bar w)\in\mmm[A]_\Delta} \mmm[I]^\gamma_\e(\bar u,\bar w)\leq \inf_{\bar u \in \mmm[A]} \mmm[I]^\gamma_\e(\bar u,-\lap \bar u) =\inf_{\bar u \in \mmm[A]} \mmm[I]^\gamma_\e(\bar u).
 \]
Thus a solution of the mixed problem is also a solution of the original problem. 
\end{proof}

\begin{rem}\label{rem:nondiff}
In the continuous setting we are unable to calculate derivatives (which would lead to critical points, optimality conditions of $\mmm[I]^\gamma_\e$ etc.), due to the lack of differentiability of the misfit functional with respect to $s(u)$ (briefly mentioned in Remark \ref{rem:findimprior}). Hence discussion of derivatives is delayed until \Sref{sec:discderiv}.
\end{rem}

\section{Discretization}
\subsection{Discrete forward problem} \label{sec:discfp}

For the discretization of \eref{eq:eik} -- \eref{eq:sonerbc} we follow \cite{DecEllSty11}. Here we recall the framework and some key results sufficient for our purpose. 
Assume that $\Omega$ is a convex polygonal domain. We take a regular quadrilateral grid of uniform grid size $h$, and so $\Omega_h = \Omega \cap \Z^2_h$ where $\Z^2_h = \{(h\alpha_1,h\alpha_2)\ |\ (\alpha_1,\alpha_2) \in \Z^2 \}$. We assume the point constraint \eref{eq:ptconstr} in the forward problem lies on a grid point, $x_0 \in \Omega_h$. We index with $\alpha\in\Z^2$, therefore we may say $x_0 = x_{\alpha_0}$ for some $\alpha_0\in \Z^2$. 

Define the discrete boundary $\p\Omega_h\coloneqq \big\{y \not\in \Omega_h \ | \ y=x+(-1)^ie_j, \ x\in\Omega_h,\ i,j \in \{1,2\} \big\}$ for standard basis vectors $e_j$ and set  $\bar \Omega_h = \Omega_h \cup \p\Omega_h$. Finally, we take the set of neighbours $\mmm[N]_\alpha$ about an interior point $x_\alpha$ to be the set of neighbours in $\bar\Omega_h$, but for boundary $x_\alpha$ we take the set of neighbours in $\Omega_h$ (i.e interior neighbours only).  

In this way the \eref{eq:eik} --  \eref{eq:sonerbc} can be discretized using the monotone finite difference scheme,
\begin{eqnarray}
 \sum_{x_\beta \in \mmm[N]_\alpha} \Big[\Big(\frac{\mmm[T]_\alpha - \mmm[T]_\beta}{|x_\alpha-x_\beta|}\Big)^+\Big]^2  = s(u_h(x_\alpha))^2,\qquad & \textrm{ if } x_\alpha \in \bar\Omega_h \setminus \{x_{\alpha_0}\}, \label{eq:FDeik} \\
 \mmm[T]_{\alpha_0} = 0\label{eq:FDptconstr},&
\end{eqnarray}
where $\mmm[T]_\delta = \mmm[T](x_\delta)$, $u_h:\Omega_h \to \R$ continuous interpolation of $u$ onto $\Omega_h$ and $y^+ = \max(y,0)$, and 
\[\fl \indent
 s:[-1,1] \to [s_{\min},s_{\max} ],\qquad s(u_h (x_\alpha))=\frac{s_{\max} - s_{\min}}{2}u_h(x_\alpha) +\frac{s_{\max} + s_{\min}}{2}.
\]

It is known \cite{DecEllSty11}, for $0<u\in C^0(\bar\Omega)$, equations \eref{eq:FDeik} -- \eref{eq:FDptconstr} gives a unique nonnegative (uniform) Lipschitz continuous solution $\mmm[T]:\bar\Omega_h \to \R_{\geq 0}$ which converges as the mesh size decreases, 
 \begin{equation}\label{eq:approxeik}
 \max_{x_\alpha \in \Omega_h} |\mmm[T]_\alpha - T(x_\alpha)|\leq C\sqrt{h}.
\end{equation}
This equation is solved using the Fast Marching Method \cite{Set96,Set99,SetVla00,DecEllSty11}, an efficient and robust technique terminating in a finite number of steps.  We label the discrete FMM solution map taking $u_h$ and obtaining the discrete solution $\mmm[T]$, as $\mmm[G]_h(s(u_h))$.

\subsection{The misfit functional and prior space}

We now consider a discrete  analogue of problem \eref{eq:invprob}. 
We discretize the misfit functional boundary integral \eref{eq:mismatchbdry} to give $\mmm[I]_h$. More precisely, we take $\mathcal{O}_h = L^2_h(\Gamma_h)$, $\Gamma_h\subset\p\Omega_h$, a finite difference approximation of $L^2(\Gamma_h)$. Let $\mmm[T]:\bar\Omega_h \to \R_{\geq 0}$, be the solution of \eref{eq:FDeik} -- \eref{eq:FDptconstr} with slowness function $s(u_h)$. Then,
\begin{equation}\label{eq:mismatchbdrydisc}
 \mmm[I]_h(u_h)\coloneqq \frac{1}{2}\sum_{i=1}^N h_{\alpha_i}|\mmm[T](x_{\alpha_i}) - T_{obs}(x_{\alpha_i})|^2,
\end{equation}
where $x_{\alpha_i}$ are ordered boundary grid points in $\Gamma_h$ and $h_{\alpha_i}=\frac{1}{2}(|x_{\alpha_i}-x_{\alpha_{i-1}}|+|x_{\alpha_i}-x_{\alpha_{i+1}}|)$. The observations are assumed to be at grid points, and henceforth denoted by $y_h=\{T_{obs}(x_{\alpha_i})\}$.

\subsection{Discrete inverse problem} \label{sec:discip}

We work with a finite element formulation for approximating $u$ and the regularization functional. To keep the distinction between the discretization of forward and inverse problems, we use $\hbar$ as the numerical approximation parameter, such that the scheme converges as $\hbar \to 0$. 
For $\hbar\geq h>0$, we decompose the polygonal domain $\Omega$ into a union of triangles, $\mathscr{T}_\hbar$, whose  diameters are  bounded below by $\hbar$. 
Working directly by discretizing \eref{eq:invprob},  would require an $H^2$ conforming finite element space to approximate the function $u$, (such as Lagrangian $\mathbb{P}^3$ for $C^0$ elements, or Bell triangular for $C^1$ elements), unfortunately these are computationally costly, see \cite{book:BreFor12}. Therefore, we present a formulation based on $H^1$ conforming elements and work with a discretization based on the mixed formulation \eref{eq:optcontmix}. We choose to use $P^1$ Lagrangian finite elements, given by piecewise linear nodal basis functions $\{\phi_l\}$ for every $x_l \in \Omega$ and satisfying $\phi_l(x_k) = \delta_{kl}$. Define this finite element space,  
\[
 S_\hbar(\Omega) \coloneqq \big\{ v \in C^0(\Omega) \ | \  v_h |_{\mmm[T]} \in \mathbb{P}^1, \mmm[T] \in \mmm[T]_\hbar\big\},
\]
where the space of polynomials can be written as 
\[
 \mathbb{P}^1 = \big\{ v \ | \ v(x) = \sum^L_{l=1} v_l \phi_l(x) \textrm{ where } v_l = v(x_l), \ \  \forall x_l \in \Omega \big\}  .
\]

\begin{rem}
 We use a finer numerical mesh parameter for the forward problem than the inverse problem i.e $\hbar \le h$.  This is because our problem is to approximate $u$ and the solution of the forward problem is required for resolution. 
In proofs we often identify the nodes of both finite difference and finite element grids, so that we do not have to include the forward and backward interpolation operators.
\end{rem}

\label{sec:discmix}
We define the discrete functional, 
\begin{equation}\label{eq:mixdiscIPFabs}\fl \indent
 \mmm[I]^\gamma_{\e,h,\hbar}(u_\hbar,w_\hbar) \coloneqq \frac{1}{2}\|y_h- \mmm[G]_h(s(u_\hbar))\|_{\mmm[O]_h}^2 +  \sigma \int_{\Omega}  \gamma\frac{\e^3}{2}w_\hbar^2 +\frac{\e}{2}|\nabla u_\hbar|^2 + \frac{1}{\e}\Psi(u_\hbar)  \dd x,
\end{equation}
and the corresponding optimization problem 
\begin{equation}\label{eq:mixdiscoptcont}
 \textrm{Find } (u_\hbar,w_\hbar)\coloneqq \stackrel[(\tilde u_\hbar,\tilde w_\hbar) \in {\mmm[A]_{\Delta,\hbar}}]{} {\arg\min} \ \mmm[I]^\gamma_{\e,h,\hbar}(\tilde u_\hbar,\tilde w_\hbar). 
\end{equation}
We treat $\dD=\emptyset$  or $\dD \neq \emptyset$ separately. Define a bilinear form on $S_\hbar (\Omega)\times S_\hbar(\Omega)$ and discrete $L^2$ inner product $\langle \cdot, \cdot \rangle_\hbar$:
\[
 B_\hbar(z_\hbar,v_\hbar)\coloneqq \int_{\Omega} \nabla z_\hbar \cdot \nabla v_\hbar\dd x,  \qquad \langle z_\hbar, v_\hbar \rangle_\hbar = \int_\Omega z_\hbar v_\hbar\dd x 
\]
and denote the space 
\begin{eqnarray}
S_{D\hbar} &\coloneqq \{ u_\hbar \in S_\hbar(\Omega) \ | \ u_\hbar=0 \textrm{ on } \dD\}.\label{eq:h1dh}
\end{eqnarray}
Then we may define 
\begin{equation}\label{eq:mixdiscAdefabs}
 \mmm[A]_{\Delta, \hbar} \coloneqq \cases{ 
                              \mmm[A]_{\Delta,D\hbar}, & if  $\dD \neq \emptyset,$ \\
                              \mmm[A]_{\Delta,N\hbar}, & if  $\dD = \emptyset,$
                          }
\end{equation}
given by
\begin{eqnarray}\label{eq:mixdiscADdefabs}\fl
 \mmm[A]_{\Delta,D\hbar} \coloneqq \Big\{ (u_\hbar,w_\hbar) \in S_{D\hbar}(\Omega) \times S_\hbar(\Omega)\ \Big| \ 
 B_\hbar(u_\hbar,\zeta_\hbar) = \langle w_\hbar,\zeta_\hbar \rangle_\hbar,\  \forall \zeta_\hbar \in S_{D\hbar}(\Omega) \Big\} ,\\ \fl
 \mmm[A]_{\Delta,N\hbar} \coloneqq \Big\{ (u_\hbar,w_\hbar) \in S_{\hbar}(\Omega) \times S_\hbar(\Omega)\ \Big| \ \int_{\Omega} w_\hbar \dd x= 0,  \nonumber\\ \fl \indent \qquad
 \phantom{\coloneqq \Big\{(u_\hbar,w\hbar) \in S_{\hbar}(\Omega) \times L^2_\hbar(\Omega)\ \Big| \  \ } B_\hbar(u_\hbar,\zeta_\hbar) = \langle w_\hbar,\zeta_\hbar \rangle_\hbar,\  \forall \zeta_\hbar \in S_{\hbar}(\Omega) \Big\}.\label{eq:mixdiscANdefabs} 
\end{eqnarray}

\begin{thm}\label{thm:fulldiscmixminexist}
 Let $\Omega \subset \R^d$ be a convex polygonal domain. Define $\mmm[A]_{\Delta,\hbar}$ as in \eref{eq:mixdiscAdefabs}, and for any $\gamma, \e>0$ define $\mmm[I]^\gamma_{\e,h,\hbar}$ as in \eref{eq:mixdiscIPFabs}, for a convergent discretization $\mmm[G]_h$ of the forward problem $\mmm[G]$ \eref{eq:eik} -- \eref{eq:sonerbc} and data $y\in \mmm[O]$. Then there exists $(u_\hbar,w_\hbar)\in\mmm[A]_{\Delta,\hbar}$ such that $\mmm[I]^\gamma_{\e,h,\hbar}(u_\hbar,w_\hbar) = \min_{(\tilde u_\hbar,\tilde w_\hbar) \in \mmm[A]_{\Delta,\hbar}}\mmm[I]^\gamma_{\e,h,\hbar}(\tilde u_\hbar, \tilde w_\hbar)$.
 Moreover, every sequence $(u_{\hbar_k},w_{\hbar_k})_k$ with $\hbar_k\searrow 0$ has a subsequence  such that $u_{\hbar_k} 
$ converges  strongly in $H^1(\Omega)$ and $w_{\hbar_k}$ converges weakly in $L^2(\Omega)$ to a minimizer of $I^\gamma_{\e}(\cdot,\cdot)$, given by \eref{eq:mixPFmismatchabs}.
\end{thm}
\begin{proof}

As we work over a finite dimensional space, from the proof of Theorem \ref{thm:mixminexist} it is straightforward to deduce the coercivity and lower semi continuity of $\mmm[I]^\gamma_{\e,h,\hbar}$. Thus there exists a minimum of $\mmm[I]^\gamma_{\e,h,\hbar}$.  

The cases for $\p_D\Omega \neq \emptyset$ and $\p_D\Omega =\emptyset$ follow identically, and so we make no reference to the boundary conditions for $u$. We also assume for brevity that the nodes for the finite difference and finite element grids are identified (one may rewrite the proof with interpolation operators between them). 

Consider a sequence of minimizers $(u_{\hbar_{k}},w_{\hbar_{k}}) \in \mmm[A]_{\Delta,\hbar_k}$ of $\mmm[I]^\gamma_{\e,h_k,\hbar_k}$, for $h_k= \hbar_k$ with $\hbar_k \to 0$ as $k\to \infty$. Then by construction,  $(u_{\hbar_k},w_{\hbar_k})$ is bounded in $H^1(\Omega)\times L^2(\Omega)$.  There exists a subsequence relabeled as $(u_{\hbar_{k}},w_{\hbar_k})$ and some $(u^*,w^*)\in \mmm[A]_\Delta$ (as defined in \eref{eq:mixADdef}) such that, as $k\to \infty$:
\begin{eqnarray*}
 u_{\hbar_k} &\to u^* \textrm{ weakly in } H^1(\Omega)\label{eq:mixstrh1h} ,\\
 u_{\hbar_k} &\to u^* \textrm{ strongly in } L^2(\Omega) \label{eq:mixstrl2h}\\
 w_{\hbar_k} &\to w^* \textrm{ weakly in } L^2(\Omega)\label{eq:mixstrl2hw}
\end{eqnarray*}
For convenience we now drop the subscript $k$ in `$\hbar_k$' but note that in the following we mean  $(u_\hbar,w_\hbar)$ to denote a subsequence. The elliptic relation of $B_\hbar(u_\hbar,\zeta_\hbar) = \langle w_\hbar,\zeta_\hbar \rangle_\hbar$,  holds for any $\zeta_{\hbar} \in S_{\hbar}(\Omega)$ (resp $S_{D\hbar}$) so choose $\zeta_{\hbar} = u_{\hbar}$. Then as $\hbar \to 0$:
\begin{equation}\fl \indent
 \int_{\Omega} |\nabla u_{\hbar}|^2= \int_{\Omega} w_{\hbar}u_{\hbar}\dd x  \to \int_{\Omega} w^* u^* \dd x = \int_{\Omega} |\nabla u^*|^2\dd x,
\end{equation}
and therefore $u_{\hbar}\to u^*$ strongly in $H^1(\Omega)$. Furthermore elliptic regularity implies $u^*\in H^2(\Omega)$.    Denote by $\mathcal K$ and $\mathcal K_\hbar$ the solution operators for
 the elliptic relation $-\Delta \mathcal K \eta =\eta$  and its finite element approximation with either the Dirichlet or Neuman conditions  with appropriate data  so that
$$u^*=\mathcal K w^*~\mbox{and}~ u_\hbar=\mathcal K_\hbar w_\hbar.$$
By elliptic regularity and  standard finite element theory we have 
$$||\mathcal K \eta||_{H^2(\Omega)}\le C||\eta||_{L^2(\Omega)}~\mbox{and}~||\mathcal K\eta - \mathcal K_\hbar \eta||_{L^\infty(\Omega)}\le C\hbar^{2-\frac{d}{2}}||\eta||_{L^2(\Omega)}.$$
Decomposing
$$u_\hbar-u^*= (\mathcal K_\hbar w_\hbar-\mathcal Kw_\hbar) +(\mathcal Kw_\hbar-\mathcal Kw^*),$$
we see that the first term on the right converges to zero uniformly because of the uniform $L^2(\Omega)$ bound on $w_\hbar$ and the finite element  $L^\infty(\Omega)$ error  bound. Turning to the second term on the right, we observe that $\mathcal K w_h$ converges weakly in $H^2(\Omega)$  because of elliptic regularity and  the weak limit is  $\mathcal K w^*$. Thus by compact embedding we see that $\mathcal Kw_h-\mathcal K w^*$ converges to zero uniformly. Thus $u_\hbar$ converges uniformly to $u^*$. Stability results for approximations of viscosity solutions imply that
$\mmm[G]_{h}(s(u_{h}))$ converges uniformly to $\mmm[G](s(u^*))$. It follows that the misfit functional converges
$$\mathcal I_h(u_h)\rightarrow \mathcal I(u^*).$$ 

We may now prove the claim that that $(u^*,w^*)$ is a minimum of $\mmm[I]^\gamma_{\e}$, we see this by taking $(v,z) \in \mmm[A]_\Delta$ and a sequence $(v_k,z_k) \to (v,z)$ strongly in $H^1(\Omega)\times L^2(\Omega)$. This may be chosen using suitable interpolations. By definition 
 $\mmm[I]^\gamma_{\e,h_k,{\hbar_k}}(u_{\hbar_k},w_{\hbar_k})\leq \mmm[I]^\gamma_{\e,h_k,{\hbar_k}}(v_{\hbar_k},z_{\hbar_k})$ for all $k$. Convergence of the misfit functional  and the weak lower semi continuity of the functional $\mmm[I]^\gamma_{\e}$ (established in Theorem \ref{thm:mixminexist}) implies
\begin{eqnarray}\fl \indent
 \mmm[I]^\gamma_{\e}(u^*,w^*) \leq \stackrel[k\to\infty]{}{\lim\inf}\ \mmm[I]^\gamma_{\e,h_k,{\hbar_k}}(u_{\hbar_k},w_{\hbar_k})& \leq \stackrel[k\to\infty]{}{\lim\sup}\ \mmm[I]^\gamma_{\e,h_k,{\hbar_k}}(u_{\hbar_k},w_{\hbar_k})\nonumber\\
 &\leq \lim_{k\to \infty} \mmm[I]^\gamma_{\e,h_k,{\hbar_k}}(v_k,z_k) =\mmm[I]^\gamma_{\e}(v,z).\label{eq:mixlscIh}
\end{eqnarray}
Thus $\mmm[I]^\gamma_{\e}(u^*,w^*) = \min_{(v,z) \in \mmm[A]_\Delta}\mmm[I]^\gamma_{\e}(v,z)$.

\end{proof}

\subsection{The discrete derivative}\label{sec:discderiv}
The forward problem \eref{eq:eik} -- \eref{eq:sonerbc} is not differentiable with respect to the state variable. However using an argument in \cite{DecEllSty11} we  recover  differentiability for the discrete misfit functional associated to the discretization \eref{eq:eik} -- \eref{eq:sonerbc}. We identify finite difference and finite element nodes (as one can rewrite with interpolation operators) then for $(u_\hbar,w_\hbar)\in \mmm[A]_{\Delta,\hbar}$, we obtain $u_\hbar=\sum_{i=1}^L u_i\phi_i\in \mmm[A]_{\Delta,\hbar}$ for $\{\phi_i\}$ basis functions of $S_\hbar$ and get the form discussed in  \cite{DecEllSty11} . 

We define an adjoint problem associated with the discrete solution $\mmm[T]$ of \eref{eq:FDeik} -- \eref{eq:FDptconstr}, with discrete slowness function $s(u_\hbar)$. Find $P \colon \bar\Omega_h \setminus \{ x_{\alpha_0}\}$ so that
\begin{eqnarray}\label{eq:discadj1}\fl 
  \sum_{x_\beta \in \mmm[N]_\alpha} \Big(\frac{\mmm[T]_\alpha - \mmm[T]_\beta}{h_{\alpha,\beta}}\Big)^+\frac{P_\alpha}{h_{\alpha,\beta}} - \Big(\frac{\mmm[T]_\beta - \mmm[T]_\alpha}{h_{\alpha,\beta}}\Big)^+\frac{P_\beta}{h_{\alpha,\beta}}=0,  &\hspace{15pt} x_\alpha \in \Omega_h\setminus\{x_{\alpha_0}\}, \\ \fl 
    \sum_{x_\beta \in \mmm[N]_\alpha} \Big(\frac{\mmm[T]_\alpha - \mmm[T]_\beta}{h_{\alpha,\beta}}\Big)^+\frac{P_\alpha}{h_{\alpha,\beta}} - \Big(\frac{\mmm[T]_\beta - \mmm[T]_\alpha}{h_{\alpha,\beta}}\Big)^+\frac{P_\beta}{h_{\alpha,\beta}}= \frac{h_\alpha}{h^2}(T_{obs}-\mmm[T]_\alpha), &\hspace{15pt}x_\alpha \in \p\Omega_h,\label{eq:discadj2}
\end{eqnarray}
where $h_{\alpha,\beta}= |x_\alpha - x_\beta|$ and the discrete misfit functional is given by \eref{eq:mismatchbdrydisc}.
\begin{prop}\label{thm:dermismatch}
Let $u_\hbar=\sum_{i=1}^L u_i\phi_i \in \mmm[A]_\hbar$, and $m\in{1,\dots,L}$. Let $\mmm[I]_h$ be defined by \eref{eq:mismatchbdrydisc}, then,
\begin{equation}\label{eq:discderI}
 \frac{\p \mmm[I]_h}{\p u_m}(u_\hbar) = -h^2 \frac{s_{\max} - s_{\min}}{2}\sum_{x_\alpha \in \bar \Omega_h\setminus \{x_{\alpha_0}\}} P_\alpha s(u_\hbar)(x_\alpha)\phi_m(x_\alpha),
\end{equation}
where $P$ is the solution of \eref{eq:discadj1} -- \eref{eq:discadj2}.
\end{prop}
\begin{proof}

From \cite[Theorem 3.6]{DecEllSty11} we trivially extend to linear functions: let $s_h=c_1\sum_{k=1}^K s_k \psi_k+c_2$, for  $\{\psi_k\}$ basis functions satisfying the properties of Remark  \ref{rem:findimprior} and $c_1,c_2$ constants. If we solve \eref{eq:FDeik} -- \eref{eq:FDptconstr} with right hand side $s_h^2$, define 
\[
  \tilde{\mmm[I]}_h(s_h)\coloneqq \frac{1}{2}\sum_{i=1}^N h_{\alpha_i}|\mmm[T](x_{\alpha_i}) - T_{obs}(x_{\alpha_i})|^2,
\]
and
\[\fl \indent
 \frac{\p \tilde{\mmm[I]}_h}{\p s_m}(s_h) = -h^2 \sum_{x_\alpha \in \bar \Omega_h\setminus \{x_{\alpha_0}\}} P_\alpha s_h(x_\alpha)\frac{\p s_h}{\p s_m}(x_\alpha) = -c_1h^2 \sum_{x_\alpha \in \bar \Omega_h\setminus \{x_{\alpha_0}\}} P_\alpha s_h(x_\alpha)\psi_k(x_\alpha).
\]
Take $\{\psi_i\}$ to be the the finite element basis functions $\{\phi_i\}$, and as we have assumed the finite difference and finite element nodes are identified, we define
\[
s_h \coloneqq s(u_\hbar) = \frac{s_{\max} - s_{\min}}{2}\sum_{i=1}^L  u_i \phi_i +\frac{s_{\max} + s_{\min}}{2},
 \]
and the result follows: $\p \mmm[I]_h/\p u_m(u_\hbar)=\p \tilde{\mmm[I]}_h/\p u_m(s_h)= \eref{eq:discderI}$.
\end{proof}

We use the elliptic relationships of \eref{eq:mixdiscADdefabs} --  \eref{eq:mixdiscANdefabs} to describe derivatives of $w_\hbar$. In particular $\p w_\hbar/\p u_m$ satisfies
\[
  \int_{\Omega} \frac{\p w_\hbar}{\p u_m} \zeta_\hbar \dd x = \int_{\Omega}\nabla \phi_m \cdot \nabla \zeta_\hbar \dd x \qquad \forall \zeta_\hbar \in S_{D\hbar}(\Omega) \ \ (\textrm{resp }S_{\hbar} (\Omega) ).
\]
As $w_\hbar=\sum_{i=1}^K w_l \phi_l$, where $\{\phi_l\}$ are a basis of $S_\hbar(\Omega)$, set $\zeta_\hbar = \phi_l$. The chain rule gives
\begin{equation}\label{eq:discw2dif}\fl 
  \int_{\Omega} \frac{\p(w_\hbar^2)}{\p u_m}=\int_{\Omega} \frac{\p(w_\hbar^2)}{\p w_\hbar} \frac{\p w_\hbar}{\p u_m} = 2\sum_{l=1}^Lw_l\int_{\Omega}\phi_l  \frac{\p w_\hbar}{\p u_m} \dd x = 2\sum_{l=1}^L w_l\int_{\Omega}\nabla \phi_m \cdot \nabla \phi_l \dd x.
\end{equation}

\begin{prop} Let $(u_\hbar=\sum_{l=1}^L u_i\phi_i,w_\hbar=\sum_{l=1}^L w_i\phi_i)\in\mmm[A]_{\Delta,\hbar}$ defined by \eref{eq:mixdiscAdefabs} and $m\in \{1,\dots,L\}$. Let $\mmm[I]_{\e,h,\hbar}$ be defined by \eref{eq:mixdiscIPFabs}. For $P:\Omega_h\setminus\{x_{\alpha_0}\} \to \R$, the solution for the adjoint equations \eref{eq:discadj1} -- \eref{eq:discadj2} for the discrete Eikonal equations \eref{eq:FDeik} -- \eref{eq:FDptconstr} with slowness $s(u_\hbar)$. Then 
\begin{eqnarray}\label{eq:mixdiscderivI} \fl 
 \frac{\p \mmm[I]^\gamma_{\e,h,\hbar}}{\p u_m}(u_\hbar,w_\hbar) =&\ -h^2 \frac{s_{\max} - s_{\min}}{2}\sum_{x_\alpha \in \Omega_h\setminus\{x_{\alpha_0}\}}P_\alpha s(u_\hbar)(x_\alpha)\phi_m(x_\alpha) \nonumber \\
 &\ +\sigma \sum^L_{l=1}  \int_{\Omega}  \gamma \e^3 w_l \nabla \phi_m\cdot\nabla \phi_l + \e u_l\nabla \phi_m \cdot \nabla \phi_l - \frac{1}{\e} u_l\phi_m \phi_l \dd x,
  \end{eqnarray} 
 where, 
 \begin{equation}\label{eq:mixdiscw}
 \sum_{k,l=1}^L w_l \int_\Omega \phi_k \phi_l \dd x =  \sum_{k,l=1}^Lu_l\int_{\Omega}\nabla \phi_k \cdot \nabla \phi_l \dd x.
 \end{equation}
 \end{prop}
 \begin{proof}
  Insert the ansatz for $(u_\hbar,w_\hbar)$ into the regularization integral of \eref{eq:mixdiscIPFabs} and differentiate with respect to $u_m$. In view of Proposition \ref{thm:dermismatch} and \eref{eq:discw2dif} we obtain \eref{eq:mixdiscderivI} -- \eref{eq:mixdiscw}.
 \end{proof}

 \begin{rem}
Let $\bbb[u]= (u_l)_l$ and $\bbb[w]= (w_l)_l$, we can write the above in a matrix formulation
\begin{eqnarray}\label{eq:mixdiscderivImat}
 \frac{\p \mmm[I]^\gamma_{\e,h,\hbar}}{\p u_m}(u_\hbar,w_\hbar) =&\ - h^2 \frac{s_{\max} - s_{\min}}{2}\sum_{x_\alpha\in \Omega_h\setminus\{x_0\}}P_\alpha s(u_\hbar)(x_\alpha)\phi_m(x_\alpha)\nonumber \\
 &\ +\sigma \Big(\mathbb{S}( \gamma \e^3\bbb[w] +\e\bbb[u]) -\frac{1}{\e} \mathbb{M}\bbb[u]\Big),
\end{eqnarray} 
and,  
\begin{equation}\label{eq:mixdiscwmat}
\mathbb{M}\bbb[w] = \mathbb{S}\bbb[u]. \qquad \Big( \textrm{for }\mathbb{M}_{ij} = \int_{\Omega} \phi_i \phi_j \dd x, \quad \mathbb{S}_{ij} = \int_{\Omega} \nabla \phi_i \cdot \nabla \phi_j \dd x \Big).
\end{equation}
We call $\mathbb{M}$ the mass matrix  and $\mathbb{S}$ the stiffness matrix for the discretization.
\end{rem}

\subsection{An explicit descent scheme}\label{sec:desc}

We can write down a simple iterative scheme for updating the phase field variables from the derivative \eref{eq:mixdiscderivImat} -- \eref{eq:mixdiscwmat}. Denote coefficient vectors as $\bbb[a]= (a_l)_l$.  Set a numerical tolerance Tol $>0$, and set $\eta \in (0,1)$, and $\alpha^{{init}} \in \R_+$. Define initial coefficient vector $\bbb[u]^{(0)}$ so that $u^{(0)}_\hbar \in S_\hbar$ and set $\mathbb{M}\bbb[w]^{(0)} = \mathbb{S}\bbb[u]^{(0)}$, ensuring that $(\bbb[u]^{(0)}\cdot\bbb[\phi],\bbb[w]^{(0)}\cdot\bbb[\phi])\in \mmm[A]_{\Delta,\hbar}$.

\vspace{10pt}

\noindent For each $k=0,1,2,\dots$, do the following: 
\begin{enumerate}
\item Calculate the discrete functional derivative
\begin{eqnarray}\label{eq:mixdiscderivImatsch}\fl \indent
 \frac{\p \mmm[I]^\gamma_{\e,h}}{\p u_m}(u^{(k)}_\hbar,w^{(k)}_\hbar) =&\ - h^2\frac{s_{\max} - s_{\min}}{2} \sum_{x_\alpha\in \Omega\setminus\{x_0\}}P_\alpha s(u^{(k)}_\hbar)(x_\alpha)\phi_m(x_\alpha)\nonumber \\
 &\ +\sigma \Big(\mathbb{S}( \gamma \e^3\bbb[w]^{(k)} +\e\bbb[u]^{(k)}) -\frac{1}{\e} \mathbb{M}\bbb[u]^{(k)}\Big).
\end{eqnarray} 
 \item Find the largest step $\alpha \in \{\frac{\alpha^{{init}}}{2^{j-1}}\ | \ j\in\N\}$ so that, if we define 
\[ \cases{u_m^{(k+1)} &$= \Pi\Big(u_m^{(k)}-\alpha\frac{\p \mmm[I]^\gamma_{\e,h}}{\p u_m}(u^{(k)}_\hbar,w^{(k)}_\hbar)\Big), \qquad \forall m=1,\dots, L,$\\
                                \mathbb{M}\bbb[w]^{(k+1)}_m &$= \mathbb{S}\bbb[u]^{(k+1)}.$
                               }
\]
(where $\Pi:\R \to [-1,1]$ is a projection) then, the following inequality is satisfied:
\[
 \mmm[I]^\gamma_{\e,h}(u^{(k+1)}_\hbar,w^{(k+1)}_\hbar) - \mmm[I]^\gamma_{\e,h}(u^{(k)}_\hbar,w^{(k)}_\hbar) < -\frac{\eta}{\alpha^2}\|u^{n+1}_\hbar-u^{n}_\hbar\|^2.
\]
\item If $\alpha^{-2}\|u^{n+1}_\hbar-u^{n}_\hbar\|^2 < \textrm{Tol}$ we are done.
\item Otherwise go back to step 1. with $k \to k+1$
\end{enumerate}

\begin{rem}
 Our scheme uses tolerances based on differences of $\alpha^{-1}(u^{k+1}_\hbar - u^k_\hbar)$ which involve the projection $\Pi$ implicitly, and lead to better numerical properties than directly using gradients.
\end{rem}

\section{Numerical results}\label{sec:res}

We validate our model choice and scheme by presenting numerical simulations in two dimensions. We begin by investigating the choice of important model parameters. We then illustrate some different geometries of the true slowness function with different source -- receiver configurations designed to show the behaviour of recovery, as well as intuition into the reliability and limitations of solutions.

The solver for the forward problem was constructed in C++, and compiled into a MATLAB mex function. The inverse solver was then computed using MATLAB 2017b.

\subsection{Parameter Study}

\subsubsection{Model Parameters}
We demonstrate binary recovery of a simple test case to give intuition into sensible model parameter choices.  For this we shall use data and source receiver locations as described in \Fref{fig:circ}. We refer to the true field as `circular disk', defined on $\Omega=[0,1]\times[0,1]$, with $s_{\min}=2,\ s_{\max}=4$, and is given by 
\[
 s(x)=\cases{
	s_{\max}, &$ \  (x-\frac{1}{2})^2+(y-\frac{1}{2})^2 \leq \big(\frac{1}{4}\big)^2, $     \\
	s_{\min}, &  otherwise.
      }
\]
The source $x_0=(1/2,1/2)$ and data observed on all of $\p\Omega$. We choose the misfit functional to be a boundary integral of $\p\Omega$ as in \eref{eq:mismatchbdry}, and unless specified, in the study we do not have noisy observations. We take our prior space $\mmm[A]_{\Delta,\hbar}$ as in \eref{eq:mixdiscAdefabs}, where $\dD=\p\Omega$.  For the regularization we have $\e,\gamma,\sigma>0$. The recovery of \Fref{fig:circ} (right), had parameters $\gamma =10^{-2}$, $\e$ to produce an interface width of $1/20$, and $\sigma=10^{-3}$ and with $\nu=10^{-2}$ ($1\%$ noise on observations).

We discretize the inverse problem choosing $\hbar=1/160$, and take $\Gamma_h = \p\Omega_h$, for the discrete mismatch functional \eref{eq:mismatchbdrydisc}. We take $h=\hbar$, and avoid commiting an `inverse crime' by generating data from a forward problem solve on a fine mesh with $h_{\textrm{dat}} = h/8$. We solve the problem using the scheme of \Sref{sec:desc}, with tolerance Tol$=10^{-12}$ and $\eta=10^{-5}$, and $\alpha^{{init}}=10^4$.  We take the initial value of $u_\hbar^{0}\equiv -1$ everywhere for all studies.

\begin{figure}[ht]
        \centering
        \includegraphics[width=0.32\textwidth]{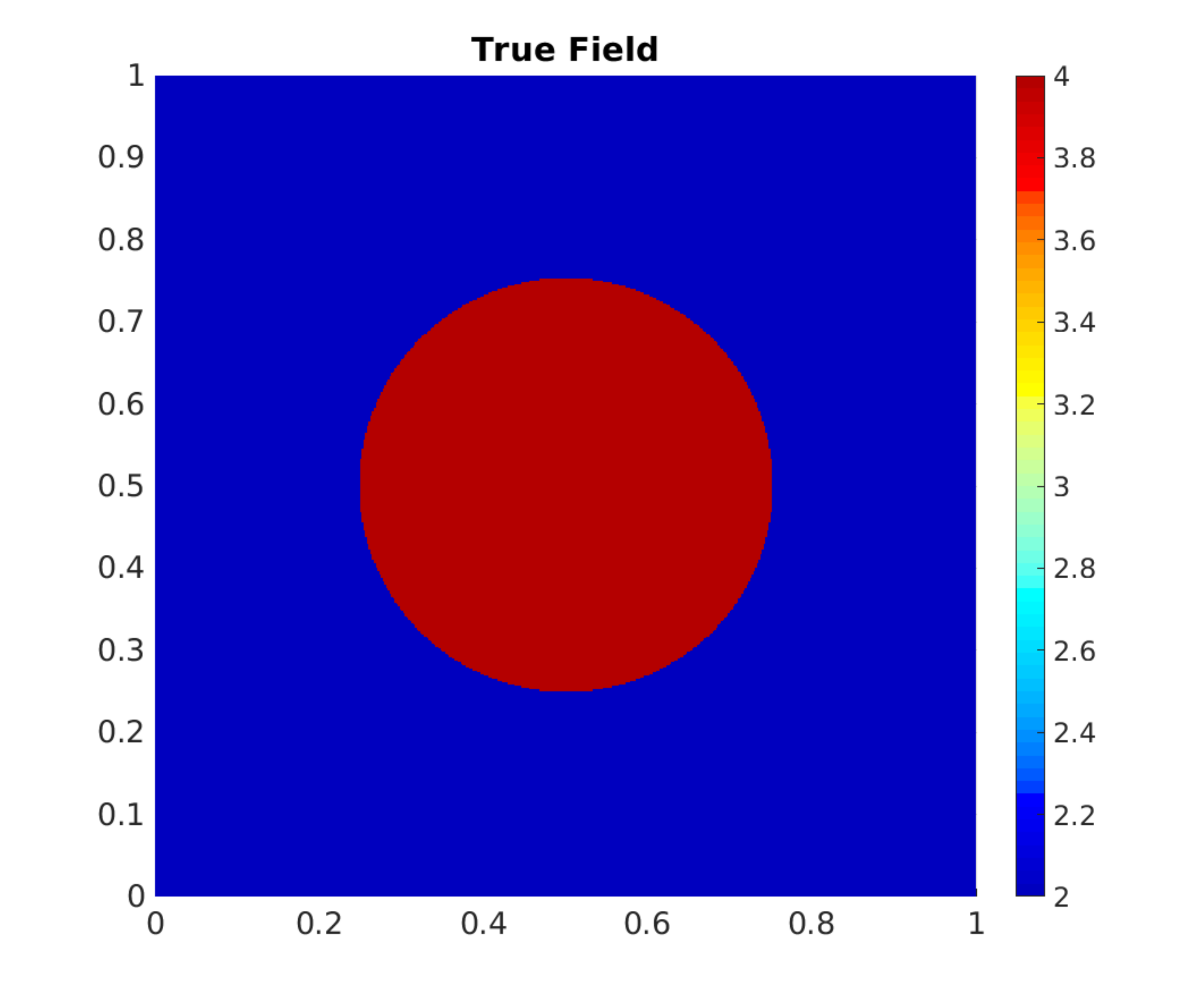}
        \includegraphics[width=0.32\textwidth]{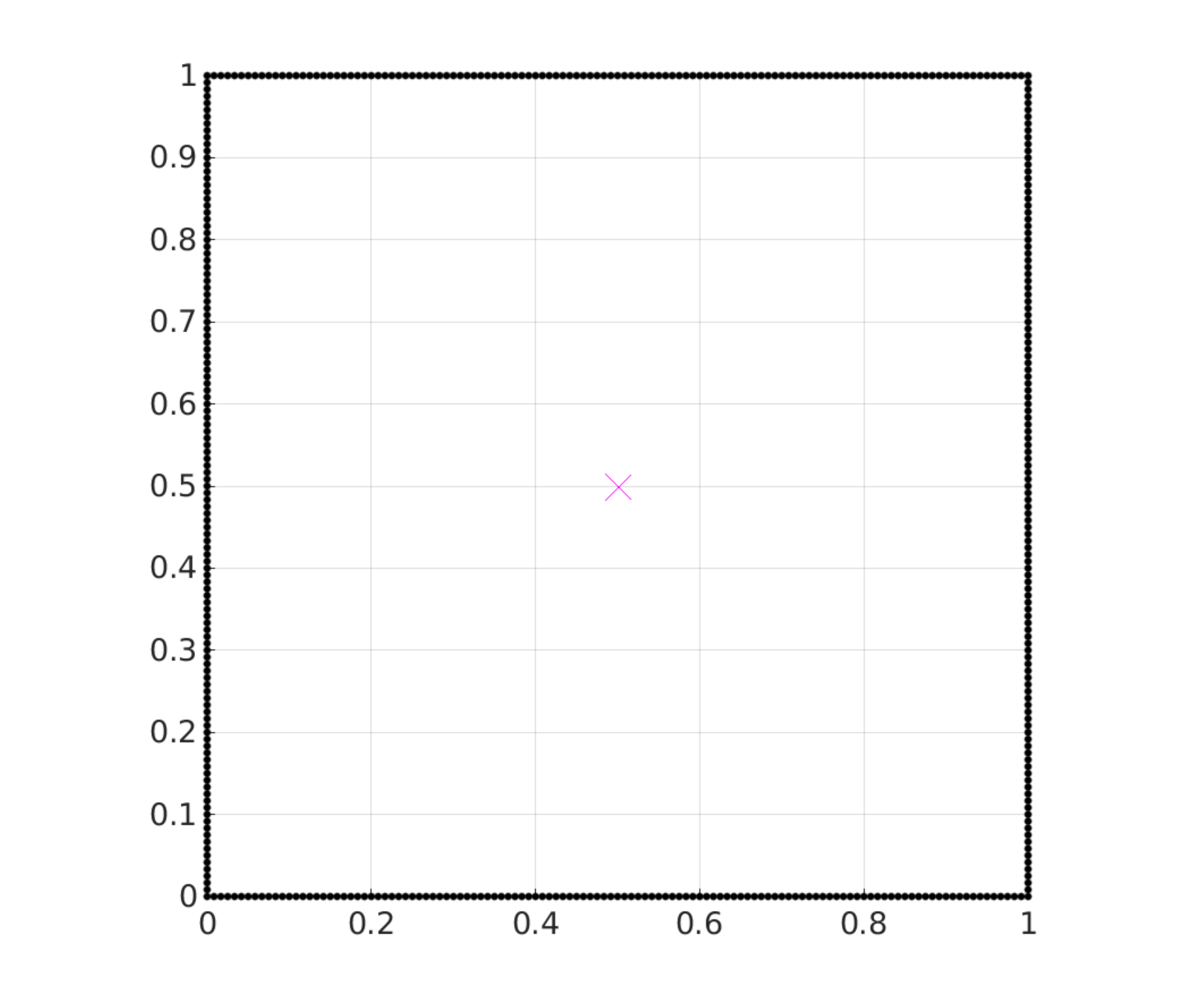}
        \includegraphics[width=0.32\textwidth]{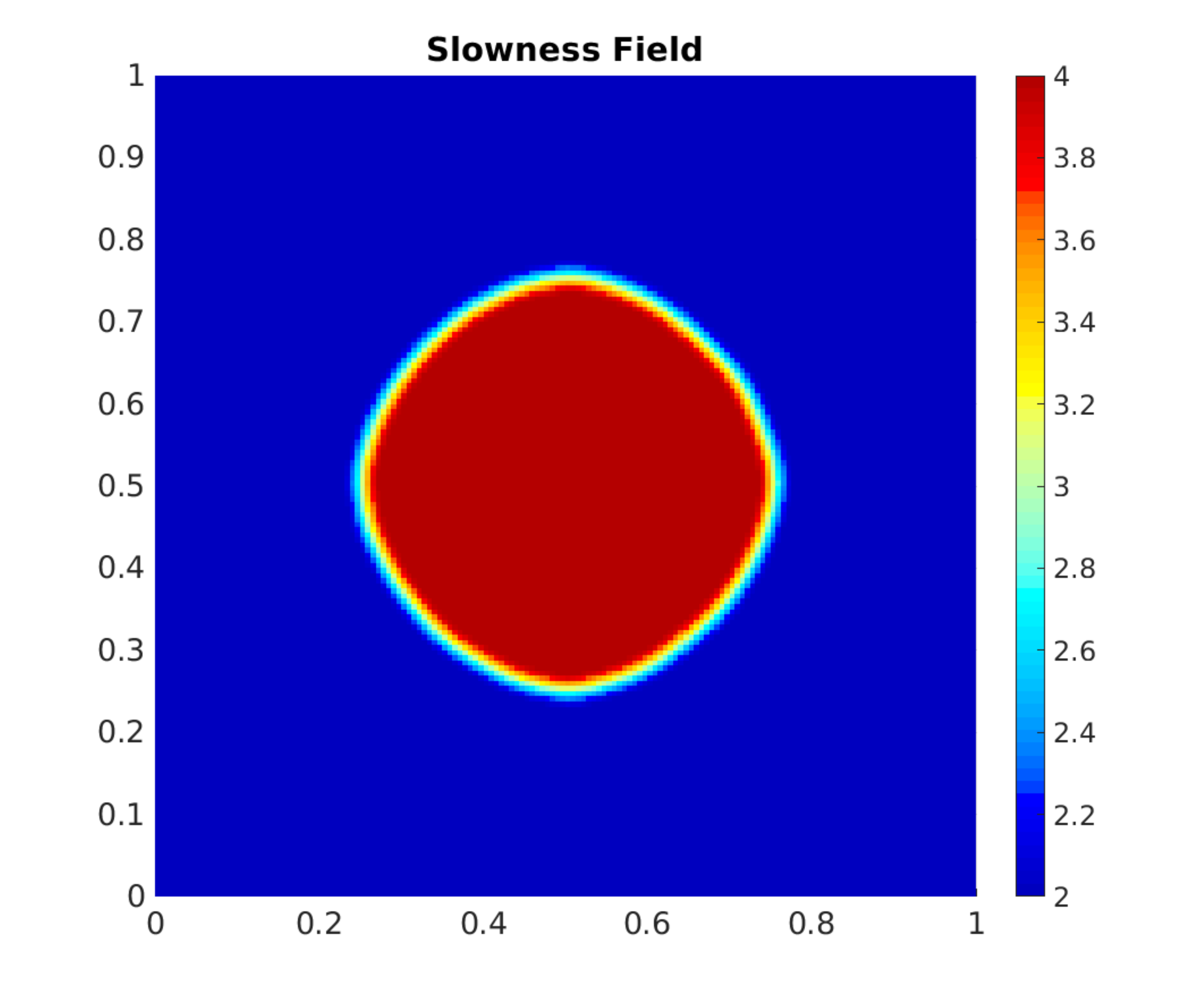}
          
          \caption{\label{fig:circ}(a) The true slowness function. (b) Position of source (cross) and observation points (dots). (c) Example recovery from data perturbed with $1\%$ noise at the observations, $\hbar=\frac{1}{160}$.}
\end{figure}

We first investigate the interfacial resolution. We do so by fixing $\gamma=10^{-4}$, and construct $\delta^\gamma$, by numerically solving \eref{eq:delta}. In view of the ansatz in Remark \ref{rem:gamma1}, we set $\e=K\hbar/(2\delta^\gamma)$, which produces an interface width of size $K\hbar$. The simulation run is displayed in \Tref{tab:eps}. We observe the $\mmm[I]^\gamma_\e$ decreases and appears to converge as interfacial width increases. We choose interfacial width at least $8\hbar$ for good accuracy. Results remained largely constant at different values of $\gamma$ - so the higher order contributions need little extra resolution.

\begin{table}[ht]
 
 \caption{\label{tab:eps}Table listing the interface widths from different $\e$ values, given a fixed value of $\gamma = 10^{-4}$. We display the size of both misfit and regularization terms, and the full functional $\mmm[I]_{\e_\hbar,h}^\gamma$.}
\begin{indented}
\item[]\begin{tabular}{@{}llll} 
\br
interface width & Misfit & $\sigma\mmm[J]^\gamma_\e(u_h,w_h)$ & $\mmm[I]_{\e_\hbar,h}^\gamma(u_\hbar,w_\hbar)$\cr
\mr
$4\hbar$ 	&	2.4442E-08&	1.6022E-04&	1.6024E-04\cr
$6\hbar$	&	2.3324E-08&	1.5754E-04&	1.5756E-04\cr
$8\hbar$	&	2.2235E-08&	1.5656E-04&	1.5658E-04\cr
$10\hbar$	&	2.2193E-08&	1.5609E-04&	1.5611E-04\cr
$12\hbar$	&	2.1949E-08&	1.5584E-04&	1.5586E-04\cr
$14\hbar$	&	2.1876E-08&	1.5568E-04&	1.5570E-04\cr
\br
\end{tabular}
\end{indented}
\end{table}

We now assess $\gamma$. For each $\gamma$, we choose $\e$ to ensure constant interface width of $8\hbar$. The results are listed in \Tref{tab:gamma}. We see that one obtains consistent values of the functional below a value $10^{-2}$, and there is significant difference for greater values. This indicates for consistency we should choose $\gamma\leq 10^{-2}$.

\begin{table}[ht]

 \caption{\label{tab:gamma}Table listing varying values of $\gamma$, we choose $\e$ appropriately in each simulation to yield interface width $8\hbar$. We display the size of both misfit and regularization terms, and the full functional $\mmm[I]_{\e_\hbar,h}^\gamma$.}
\begin{indented}
\item[]\begin{tabular}{@{}llll} 
\br
$\gamma$ & Misfit & $\sigma\mmm[J]^\gamma_\e(u_h,w_h)$ & $\mmm[I]_{\e_\hbar,h}^\gamma(u_h,w_h)$\cr
\mr
1.0000E+00&	2.3322E-08&	1.5839E-04&	1.5841E-04\cr
1.0000E-01&	2.2446E-08&	1.5726E-04&	1.5728E-04\cr
1.0000E-02&	2.1970E-08&	1.5673E-04&	1.5675E-04\cr
1.0000E-03&	2.2100E-08&	1.5659E-04&	1.5661E-04\cr
1.000E-04&      2.2235E-08&	1.5656E-04&	1.5658E-04\cr
\br
\end{tabular}
\end{indented}
\end{table}

We next investigate $\sigma$. We take $\gamma= 10^{-2}$, and interface width $8\hbar$. We validate the choice of $\sigma$ by using the approximation property of the regularization to the perimeter length of the interface. Theorem \ref{thm:gammaconv}, shows $\mmm[J]^\gamma_\e \stackrel[]{\Gamma}{\to } \mmm[J]^\gamma_0$ and this is proportional to the interfacial length by a factor $P^\gamma$ (see \eref{eq:J0}). The truth in \Fref{fig:circ}, has interfacial length $\pi/2$.  \Tref{tab:signonoise} lists the results, along with a difference of the approximate and true interface length $|1/P^\gamma \mmm[J]^\gamma_\e -\pi/2|$. We see the best choice of $\sigma$ is around $10^{-3}$. Moreover, accuracy is still good if $\sigma$ is taken too small, but quickly poor if $\sigma$ is taken too large.  

\begin{table}[ht]
 \centering
 \caption{\label{tab:signonoise}Table listing varying values of $\sigma$ - the regularization parameter, we fix $\gamma=10^{-2}$, and choose $\e$ to yield interface width $8\hbar$. We display the size of both misfit and regularization terms, and the difference of predicted and true interfacial length: $\pi/2$.}
\begin{indented}
\item[]\begin{tabular}{@{}llll} 
\br
$\sigma$ & Misfit & $\sigma\mmm[J]^\gamma_\e(u_h,w_h)$ & $|\frac{1}{P^\gamma}\mmm[J]^\gamma_\e-\frac{\pi}{2}|$\cr
\br
1.0000E-04&	2.2112E-08&	1.5748E-04&	4.0037E-03\cr
2.0000E-04&	6.5955E-08&	3.1489E-04&	3.6537E-03\cr
4.0000E-04&	2.3649E-07&	6.2956E-04&	3.1037E-03\cr
8.0000E-04&	9.0711E-07&	1.2582E-03&	1.9537E-03\cr
1.6000E-03&	3.8804E-06&	2.4975E-03&	9.8588E-03\cr
\br
\end{tabular}
\end{indented}
\end{table}
Finally we  verify the scaling of $\sigma$ when observations are subject to noise. This uncertainty is incorporated into the discrete misfit functional as in \Sref{sec:noise}. The results are displayed in \Tref{tab:signoise} for three levels of uncertainty with standard deviation $\nu=1/200,\ 1/100$, and $1/50$ corresponding to $0.5\%,\ 1\%,\ 2\%$ noise. We observe that for larger noise levels the accuracy decreases, as expected. We also observe that choosing the scaling of $\sigma$ with $1/\nu^2$ is sensible. If one takes $\bar\sigma$ too large ($4\times 10^{-3}$) we once again encounter a large loss of accuracy. 
\begin{table}[ht]

 \caption{\label{tab:signoise}Table listing results produced with detector noise defined by the standard deviation $\nu$ of a centred normal random variable, and regularization parameter scaled by $\frac{1}{\nu^2}$. We fix $\gamma=10^{-2}$, and choose $\e$ to yield interface width $8\hbar$. We display the size of both misfit and regularization terms, and the difference of predicted and true interfacial length: $\pi/2$.}
\begin{indented}
\item[]\begin{tabular}{@{}lllll} 
\br
 $\nu$, if $\Sigma\sim N(0, \nu^2)$ & $\bar \sigma$, ($\sigma \coloneqq \frac{\bar \sigma}{\nu^2})$  & Misfit & $\frac{\sigma}{\nu^2}\mmm[J]^\gamma_\e(u_h,w_h)$  & $|\frac{1}{P^\gamma}\mmm[J]^\gamma_\e-\frac{\pi}{2}|$\cr
 \mr
5.0000E-03&	1.0000E-03&	4.5291E-01&	6.2873E+01&	1.0387E-03\cr
1.0000E-02&	1.0000E-03&	3.7864E-01&	1.5734E+01&	2.5737E-03\cr
2.0000E-02&	1.0000E-03&	3.3873E-01&	3.9373E+00&	4.1237E-03\cr
\mr
5.0000E-03&	2.0000E-03&	5.8286E-01&	1.2561E+02&	6.8133E-04\cr
1.0000E-02&	2.0000E-03&	3.9781E-01&	3.1382E+01&	1.7163E-03\cr
2.0000E-02&	2.0000E-03&	3.4077E-01&	7.8455E+00&	1.6963E-03\cr
\mr
5.0000E-03&	4.0000E-03&	1.2109E+00&	2.5020E+02&	7.0463E-03\cr
1.0000E-02&	4.0000E-03&	5.8052E-01&	6.2531E+01&	7.5288E-03\cr
2.0000E-02&	4.0000E-03&	4.1042E-01&	1.5649E+01&	5.8963E-03\cr
\br
\end{tabular}
\end{indented}
\end{table}

We note that one must also balance the regularization and data misfit at different data contrasts, one can do this by rescaling $\sigma$ by $2/(s_{\max}-s_{\min})$.

\subsubsection{Discretization parameters}
We omit the tests of discretization parameters, but observed convergence in the case of fixed $\e$, and sending $\hbar=C h\to 0$ for $1\leq C \in \N$. The test scenario identical to the problem we created for the model parameters.

\subsection{Different geometries}

We have several different geometries that we wish to recover, which we shall use as ``truths'' for our inverse problem. We use the discretization parameters as in the previous section, and generate data on a mesh with $h_{\textrm{dat}}=h/8$. The choices of model parameter configuration are influenced by what we discovered in the parameter study. We take $\nu=10^{-2}$ for the noise, we also choose (unless otherwise stated) $s_{\min}=1$, $s_{\max}=1.1$ $\sigma=10^{-4}$, $\gamma=10^{-2}$, and $\e$ so that we have the interface width $8\hbar$. We use two source -- receiver configurations as demonstrated in \Fref{fig:srcdata}. 
\begin{figure}[ht]
   \centering
        \includegraphics[width=0.49\textwidth]{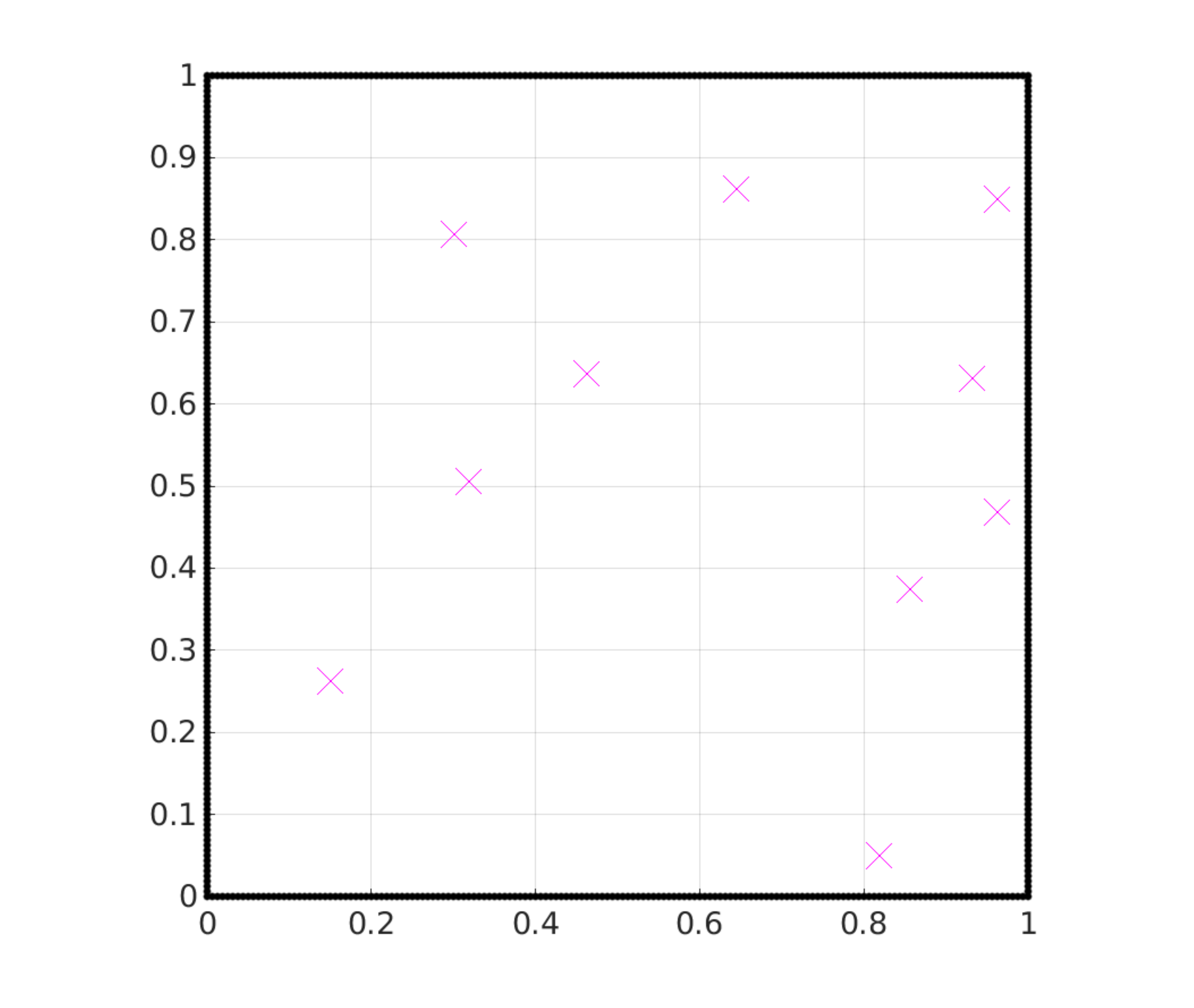}
        \includegraphics[width=0.49\textwidth]{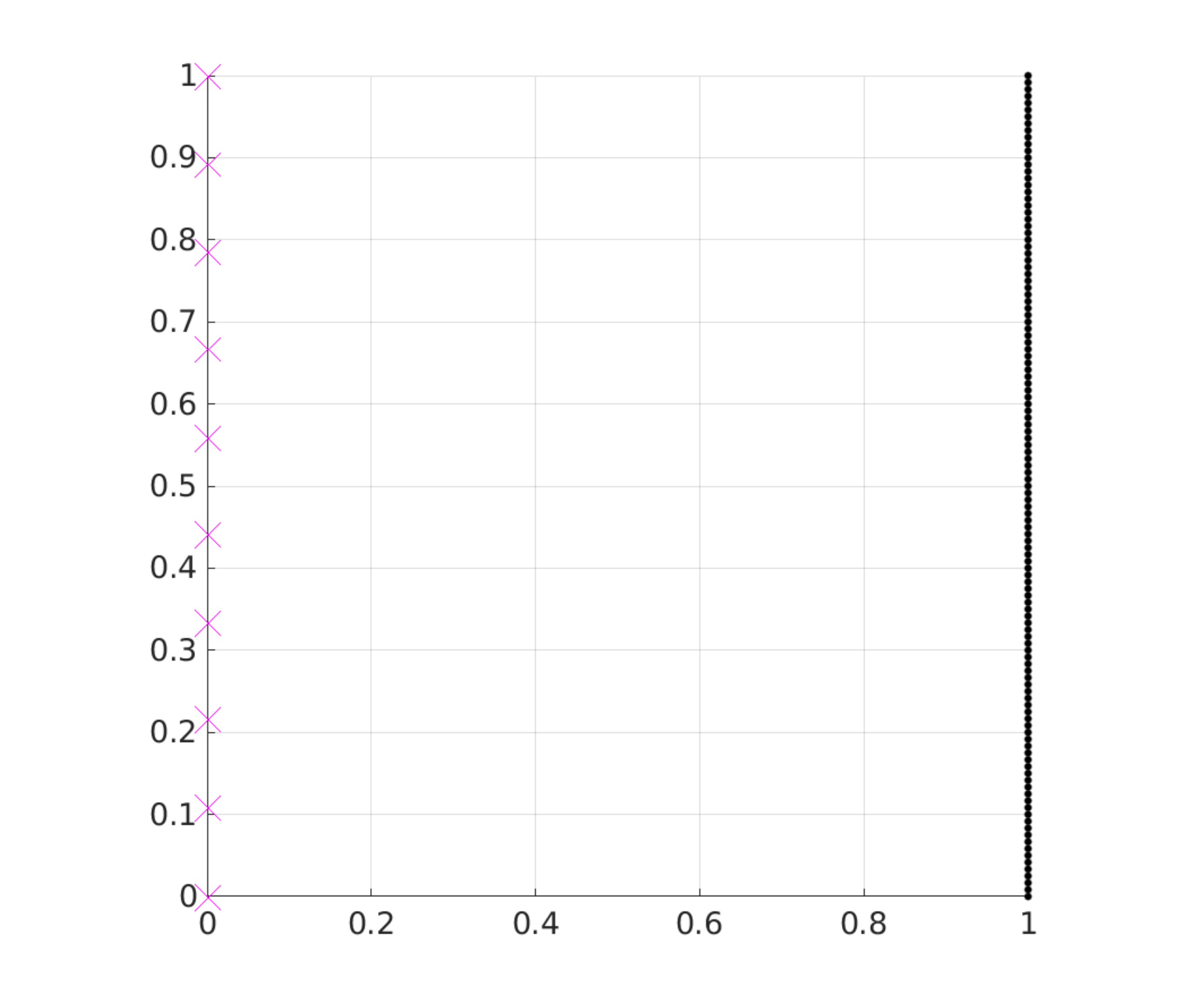}	
       
        \caption{ \label{fig:srcdata}(a) \underline{Random}:  dense boundary observations on $\p\Omega$; 10 source locations in $\Omega$ produced from MATLAB R2017b random number generator (seed=12131415, generator=twister). (b) \underline{Wells}: dense boundary observations on right hand wall of $\p\Omega$; 10 sources equidistributed on left hand wall of $\p\Omega$. }
\end{figure}

We define four more slowness function ``truths'': 
\begin{enumerate}
 \item Banded layers. We take this shape from \cite{LeuLi13}. The banded regions are sections of two annuli. That is, $\forall (x,y)\in [0,1]$
 \[\fl \indent
  \textrm{ if } \cases{
               \frac{1}{2}\Big(3.7-\sqrt{2.6^2-(2x-1)^2}\Big) \leq y \leq \frac{1}{2}\Big(4.1-\sqrt{2.6^2-(2x-1)^2}\Big),  \textrm{ or } \\
               \frac{1}{2}\Big(2.8-\sqrt{2.6^2-(2x-1)^2}\Big) \leq y \leq \frac{1}{2}\Big(3.2-\sqrt{2.6^2-(2x-1)^2}\Big),
              }
 \]
then $s(x,y)=s_{\max}$, otherwise $s(x,y)=s_{\min}$. The results are found in \Fref{fig:band}.
 
 \item Right angle. This example shape is similar to \cite{DecEllSty16}. $\forall (x,y)\in [0,1]$ 
 \[\fl \indent
 \textrm{ if } \cases{
              y\geq\frac{2}{3}x+0.4, \textrm{ or }\\
              y\geq -\frac{3}{2}x + 0.9, 
             }
 \]
then $s(x,y)=s_{\max}$, otherwise $s(x,y)=s_{\min}$. The results are found in \Fref{fig:jagged}.
 \item Arbitrary shape. $\forall (x,y)\in [0,1]$ 
 \[\fl \indent
   \textrm{ if } \cases{
                (x-\frac{2}{3})^2 + (y-\frac{1}{2})^2 \leq \frac{1}{5^2},&  \textrm{ or }\\
                (x-\frac{7}{15})^2 + (y-\frac{7}{10})^2 \leq \frac{1}{6^2},&  \textrm{ or }\\
                (x-\frac{7}{15})^2 + (y-\frac{3}{10})^2 \leq \frac{1}{8^2},&  
               }
 \]
then $s(x,y)=s_{\max}$, otherwise $s(x,y)=s_{\min}$. The results are found in \Fref{fig:splodge}.
 \item Shielded disk. $\forall (x,y)\in [0,1]$ 
 \[\fl \indent
  \textrm{ if }  (x-\frac{2}{3})^2 + (y-\frac{1}{2})^2 \leq \frac{1}{6^2},
  \textrm{ or if both }
  \cases{
  (x-\frac{2}{3})^2 + (y-\frac{1}{2})^2 \leq \Big(\frac{3}{8}\Big)^2,&  \\
               (x-\frac{4}{9})^2 + (y-\frac{1}{2})^2 \geq \frac{1}{4^2},&
  }
\]
then we set $s(x,y)=s_{\max}$, otherwise $s(x,y)=s_{\min}$. Here we choose values $s_{\min} = 1, \ s_{\max}$ to aid recovery in the wells configuration. The results are found in \Fref{fig:circban}. 
\end{enumerate}

\begin{figure}[ht]
        \centering   
        \includegraphics[width=0.32\textwidth]{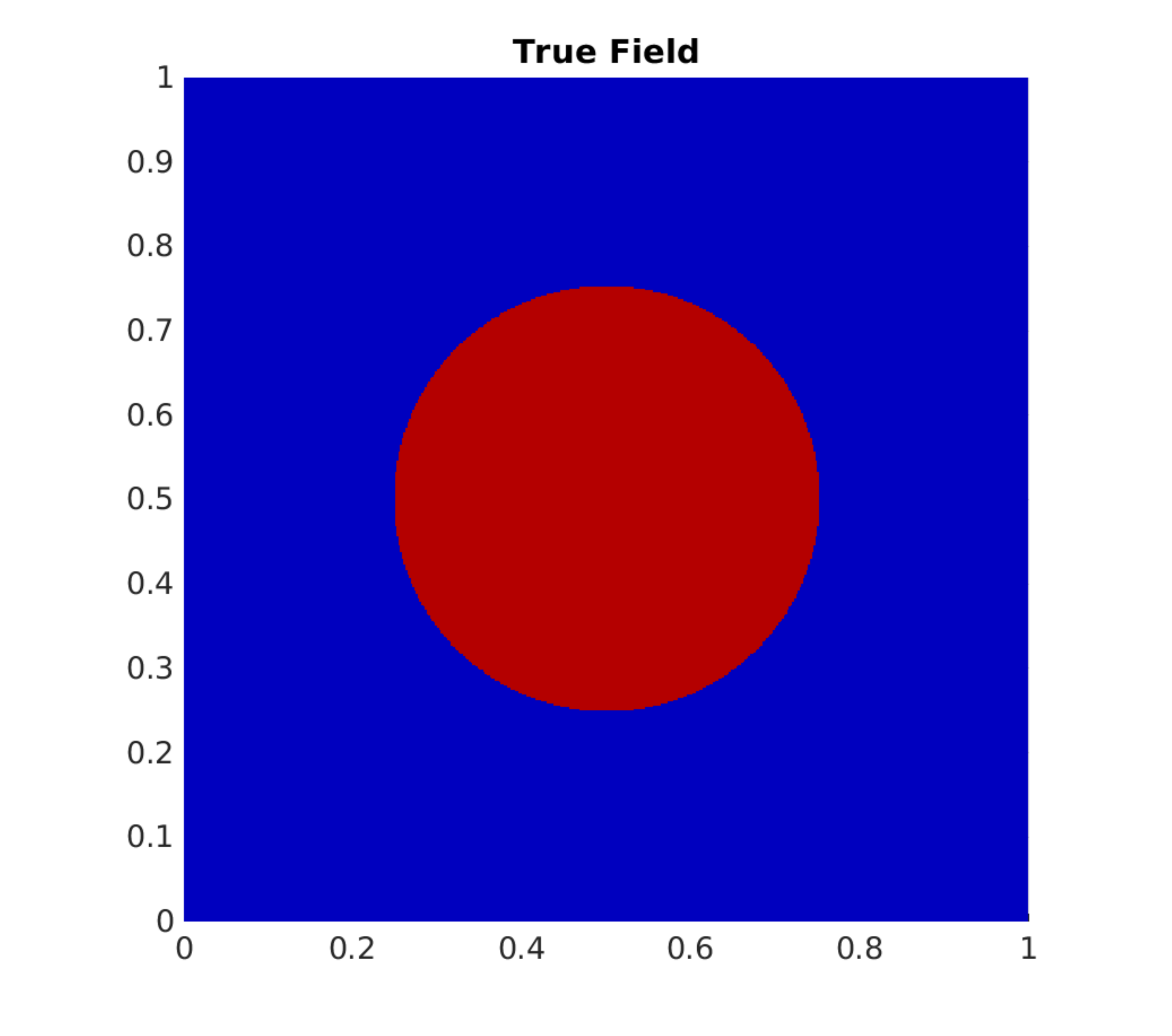}
        \includegraphics[width=0.32\textwidth]{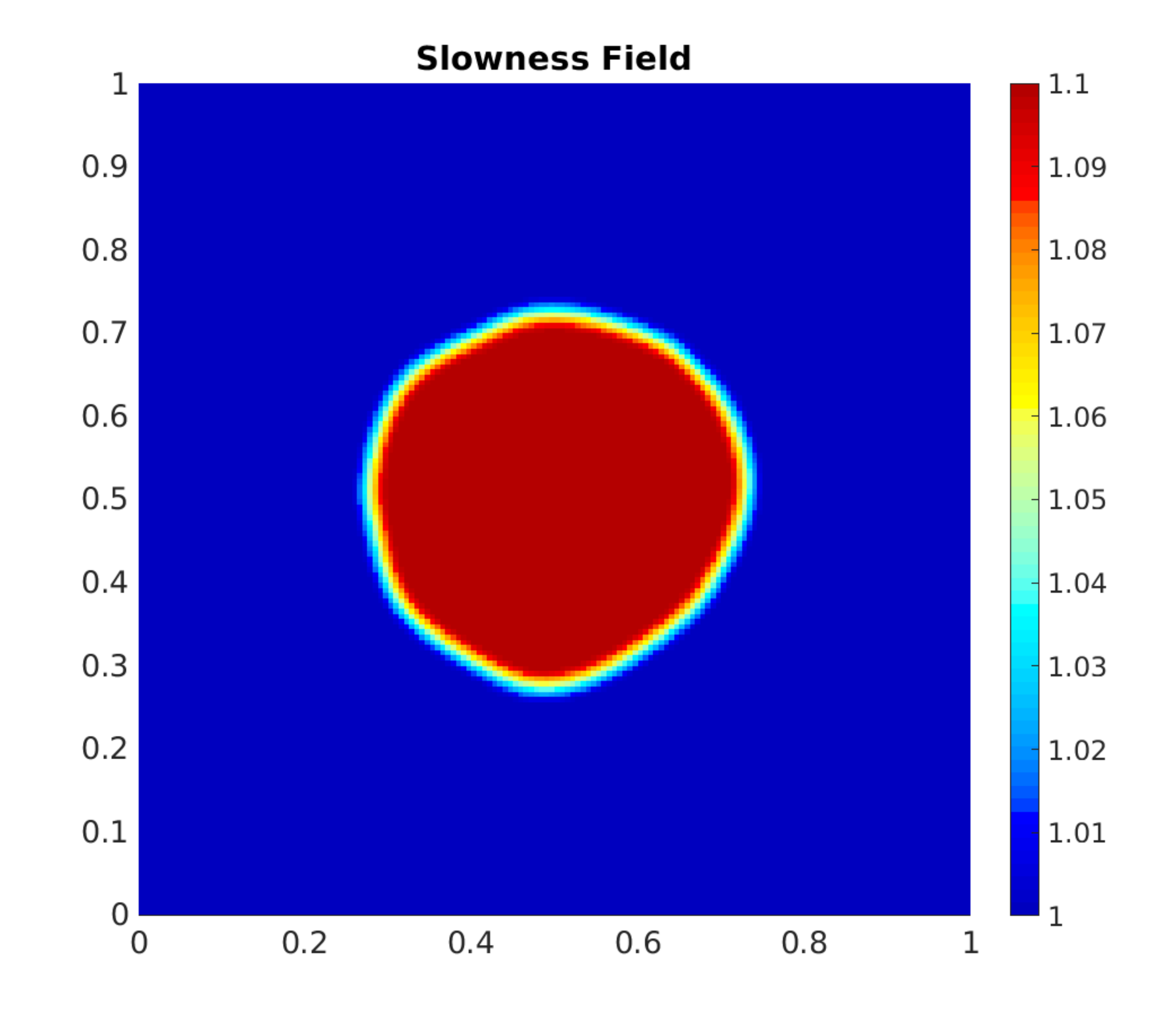}
         \includegraphics[width=0.32\textwidth]{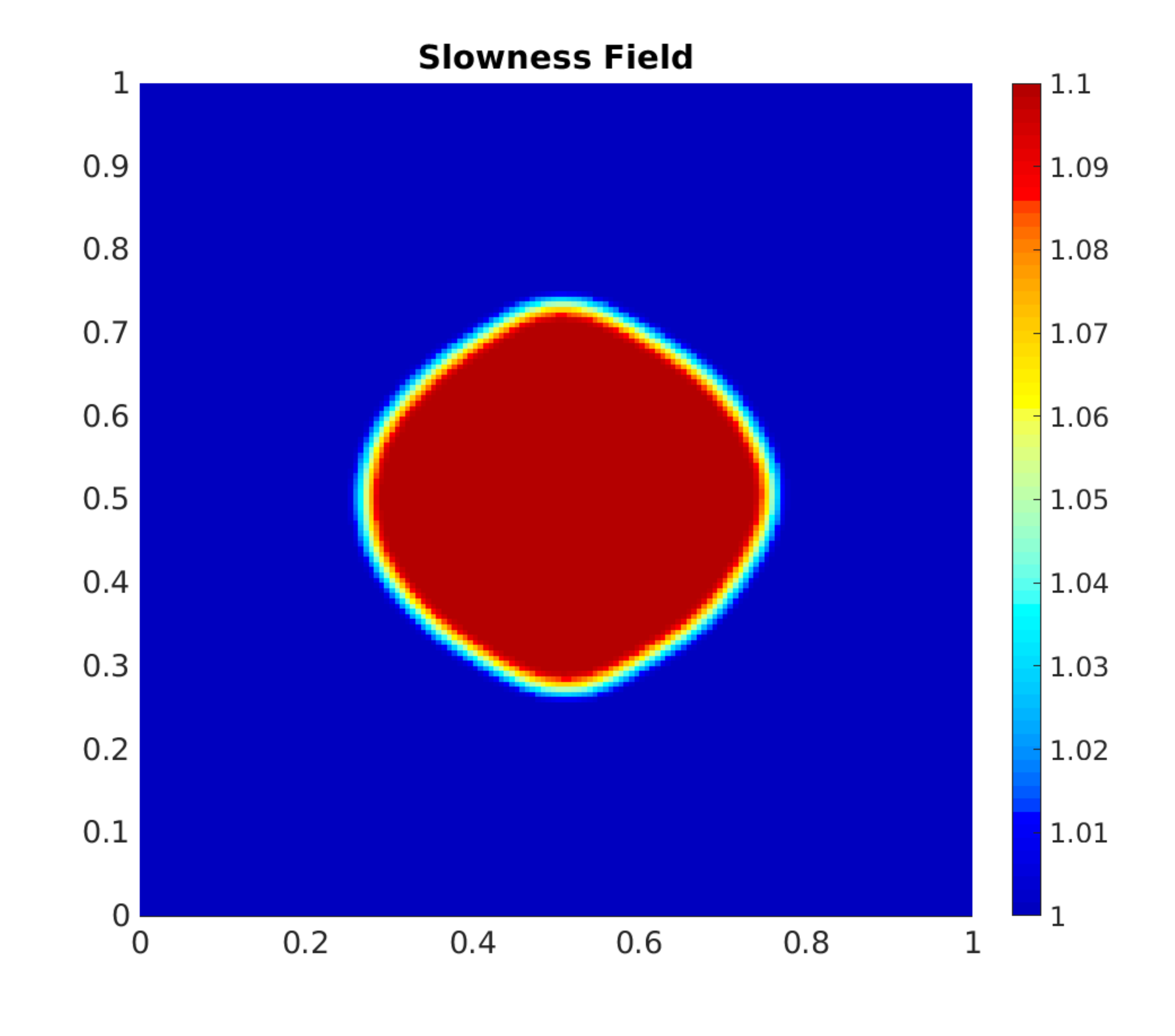}
        
        \caption{\label{fig:circ2}Circular disk recovery. (a) true slowness field, $s_{\min}=1,\ s_{\max}=1.1$. (b) slowness field from random configuration. (c) slowness field from wells configuration}
\end{figure}

\begin{figure}[ht]
        \centering   
        \includegraphics[width=0.32\textwidth]{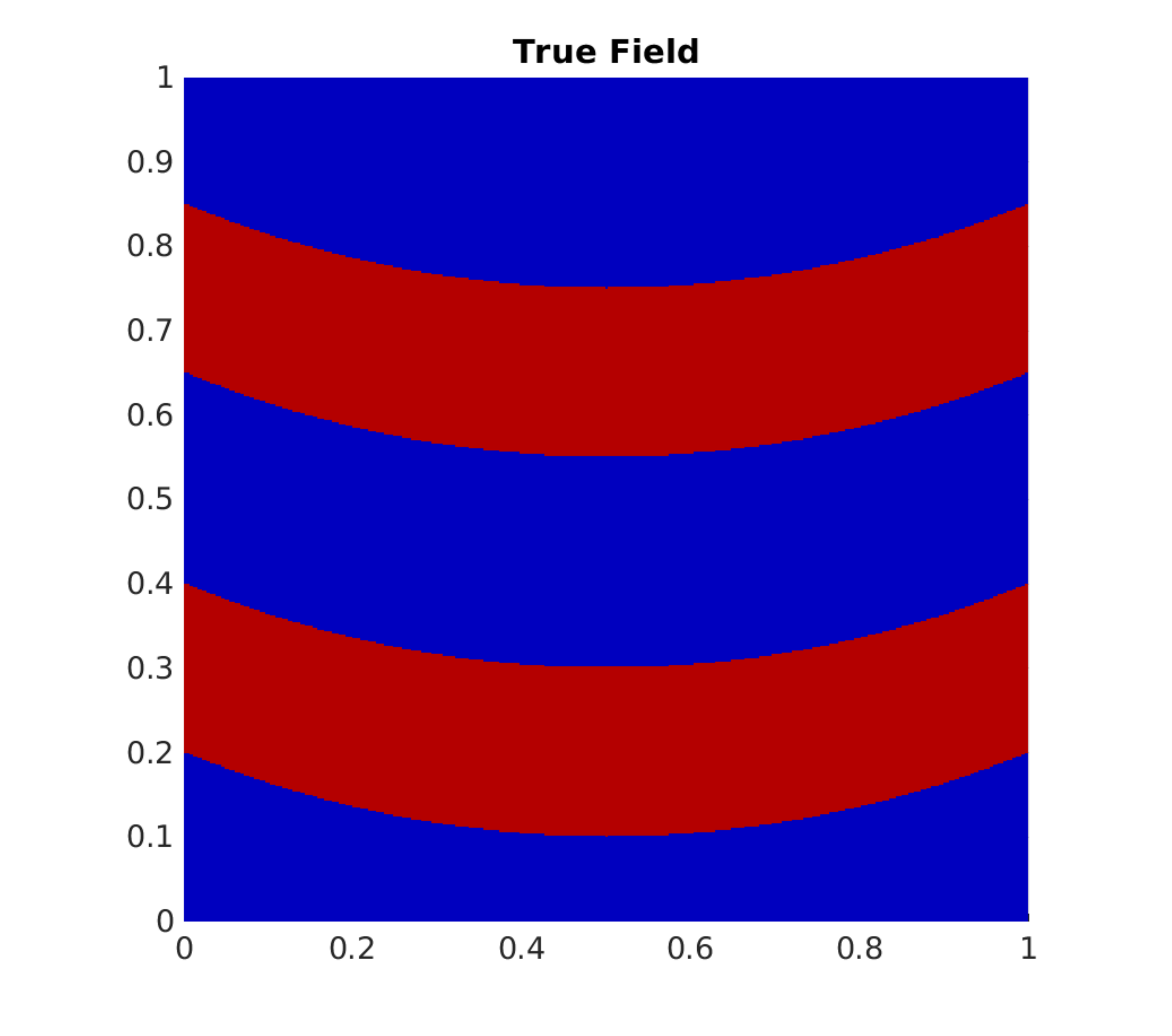}
        \includegraphics[width=0.32\textwidth]{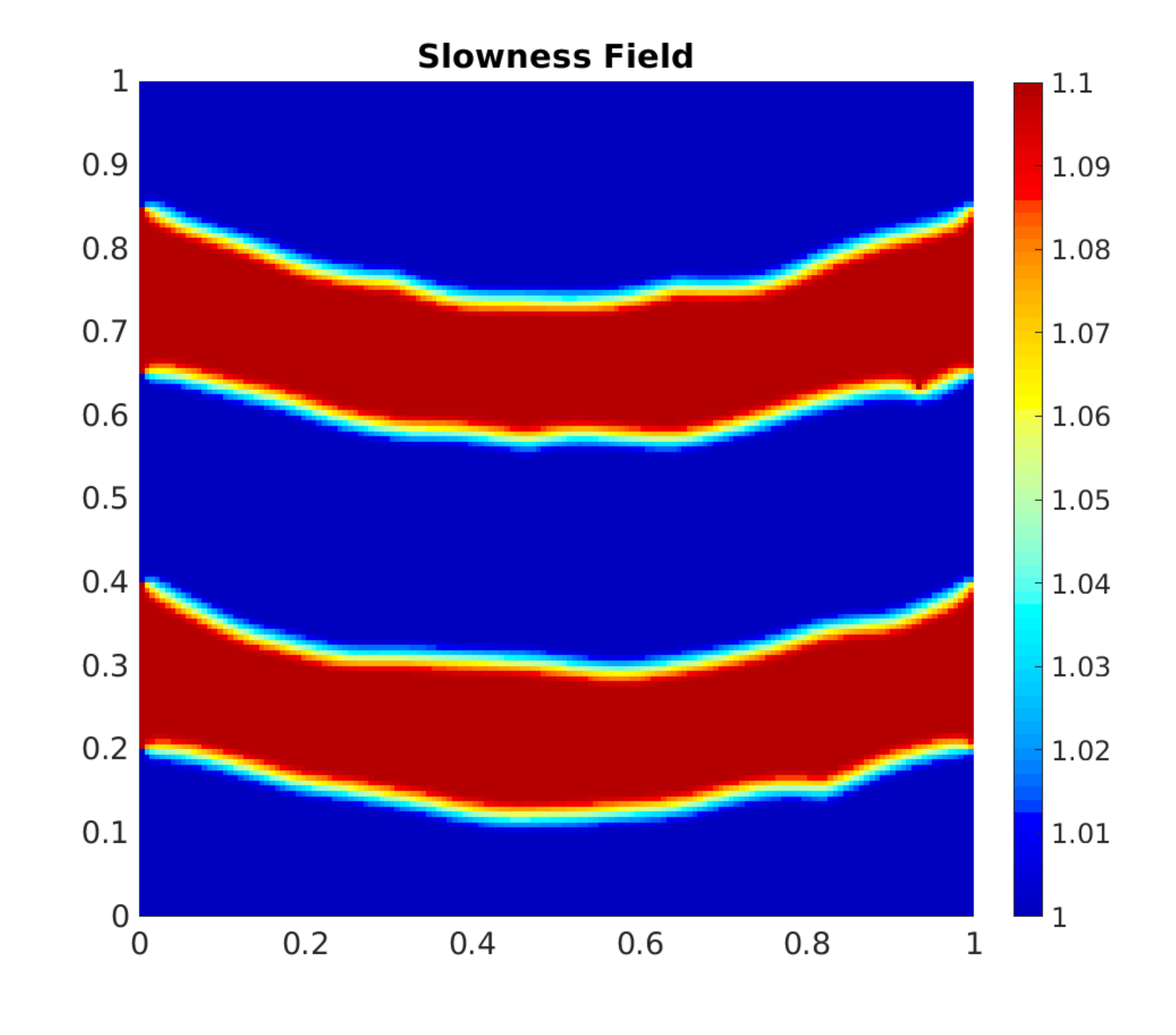}
         \includegraphics[width=0.32\textwidth]{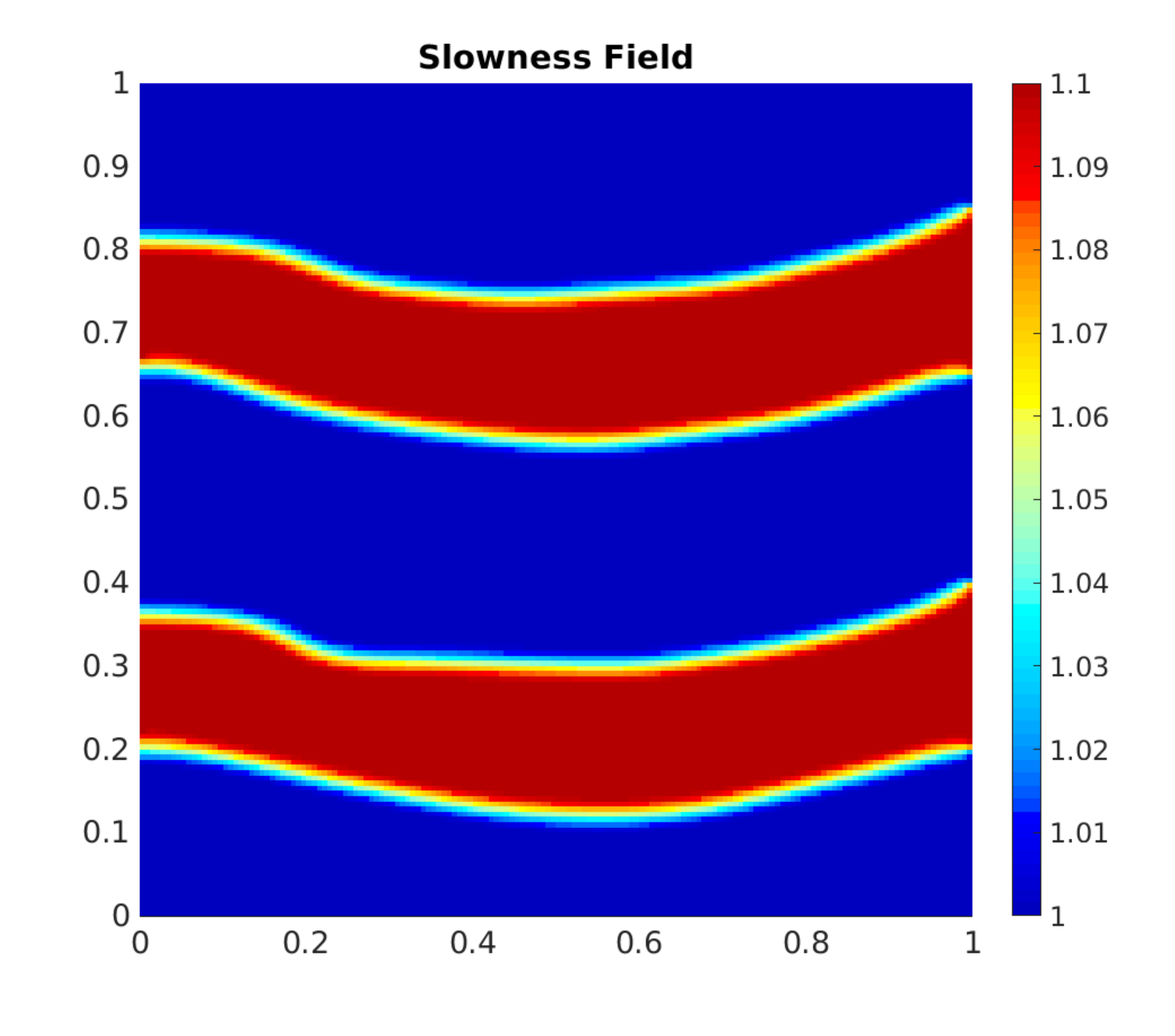}
        
        \caption{\label{fig:band}Right angle recovery. (a) true slowness field, $s_{\min}=1,\ s_{\max}=1.1$. (b) slowness field from random configuration. (c) slowness field from wells configuration}
\end{figure}

\begin{figure}[ht]
        \centering
        \includegraphics[width=0.32\textwidth]{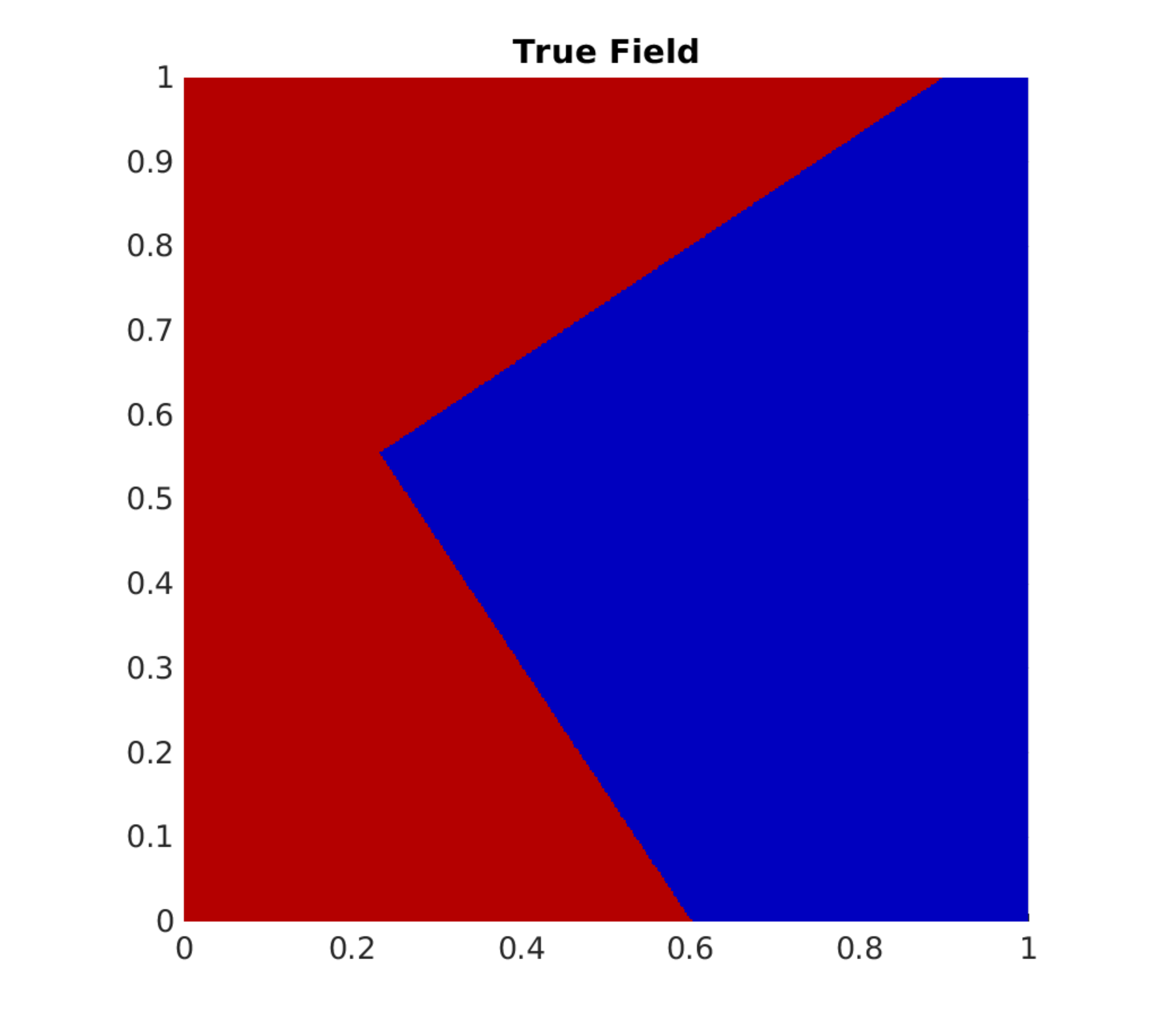}
        \includegraphics[width=0.32\textwidth]{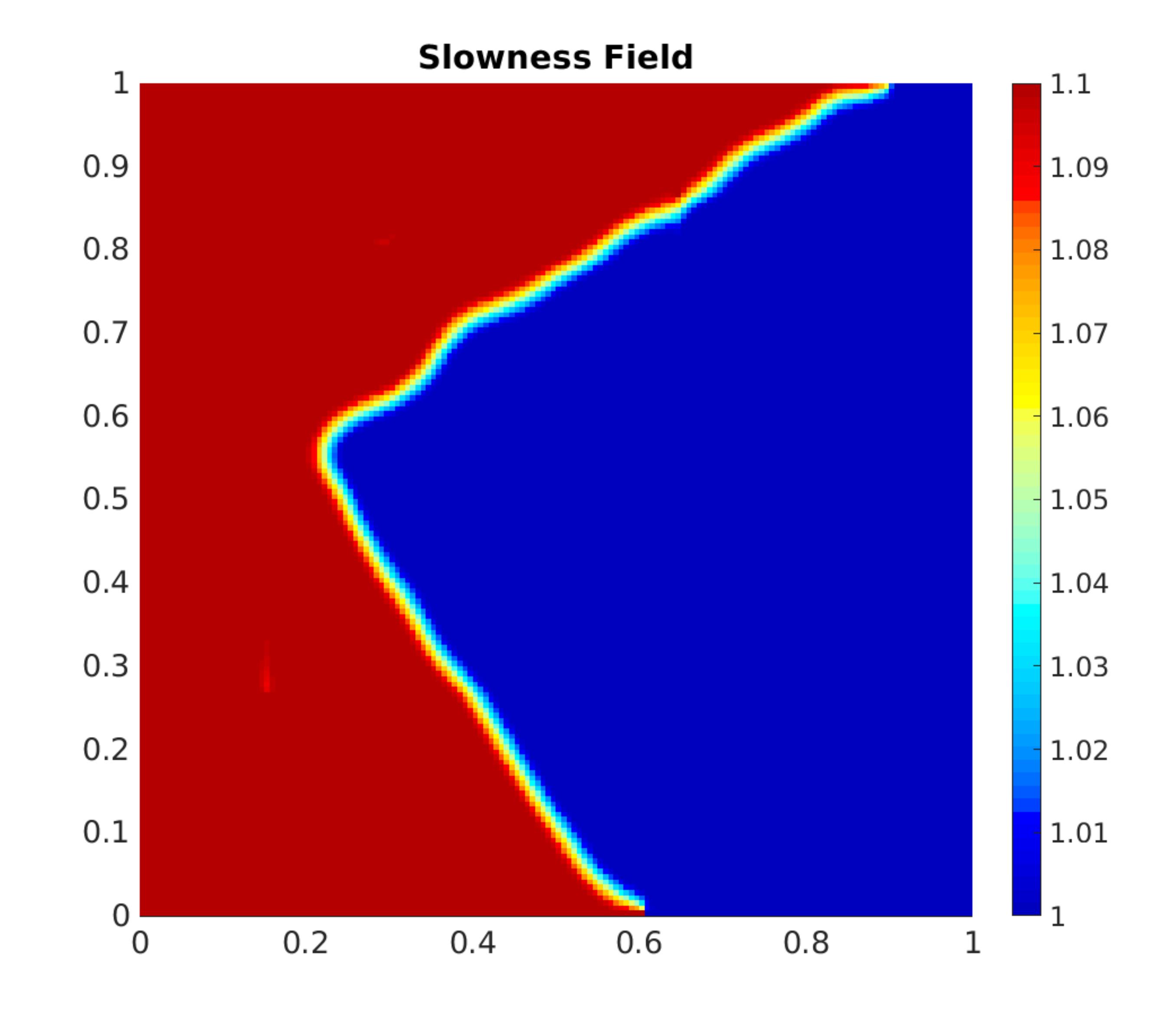}
        \includegraphics[width=0.32\textwidth]{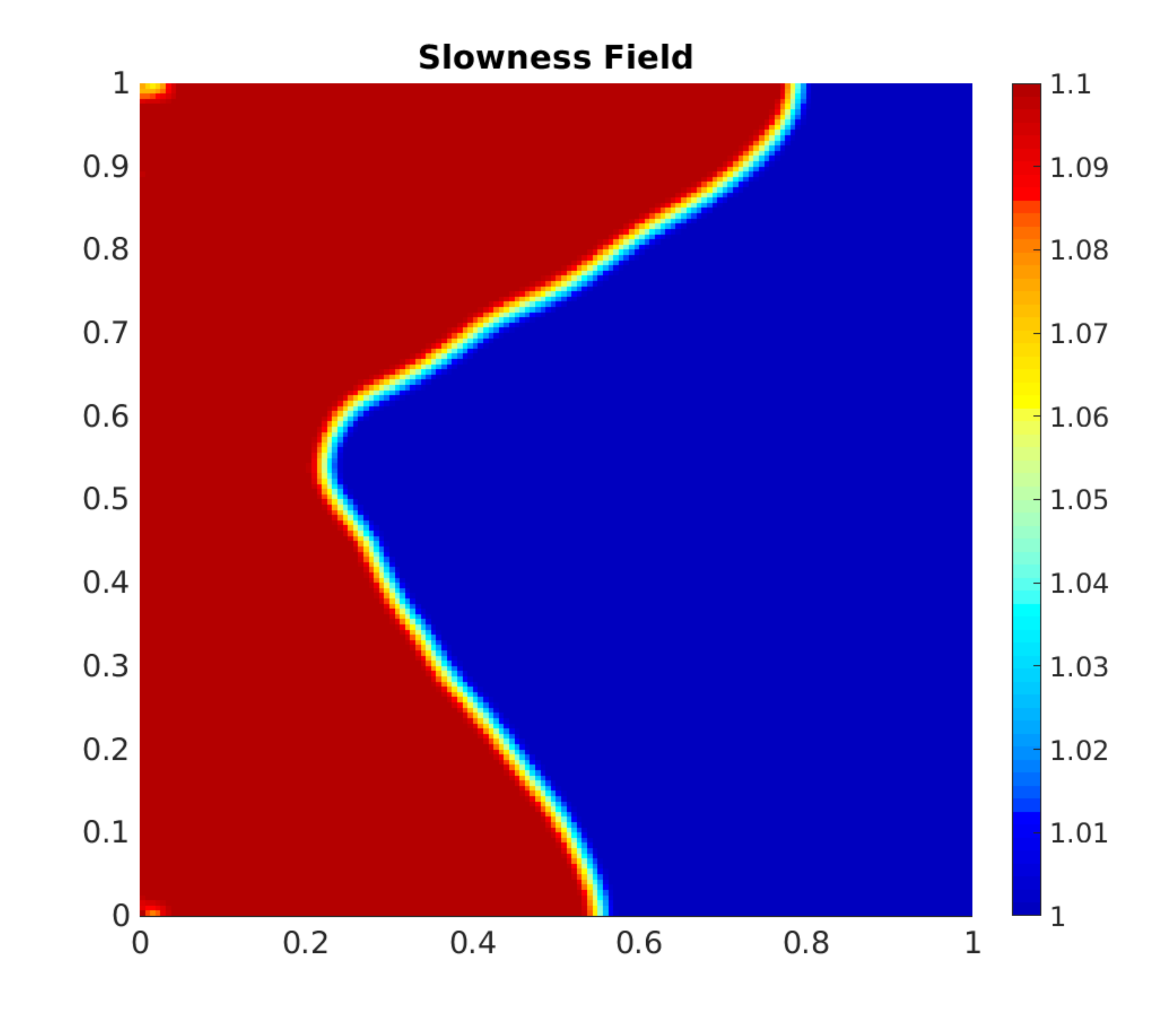}
        
        \caption{\label{fig:jagged}Angular boundary recovery. (a) true slowness field $s_{\min}=1, \ s_{\max}=1.1$. (b) slowness field from random configuration. (c) slowness field from wells configuration}
\end{figure}
        
\begin{figure}[ht]
        \centering
        \includegraphics[width=0.32\textwidth]{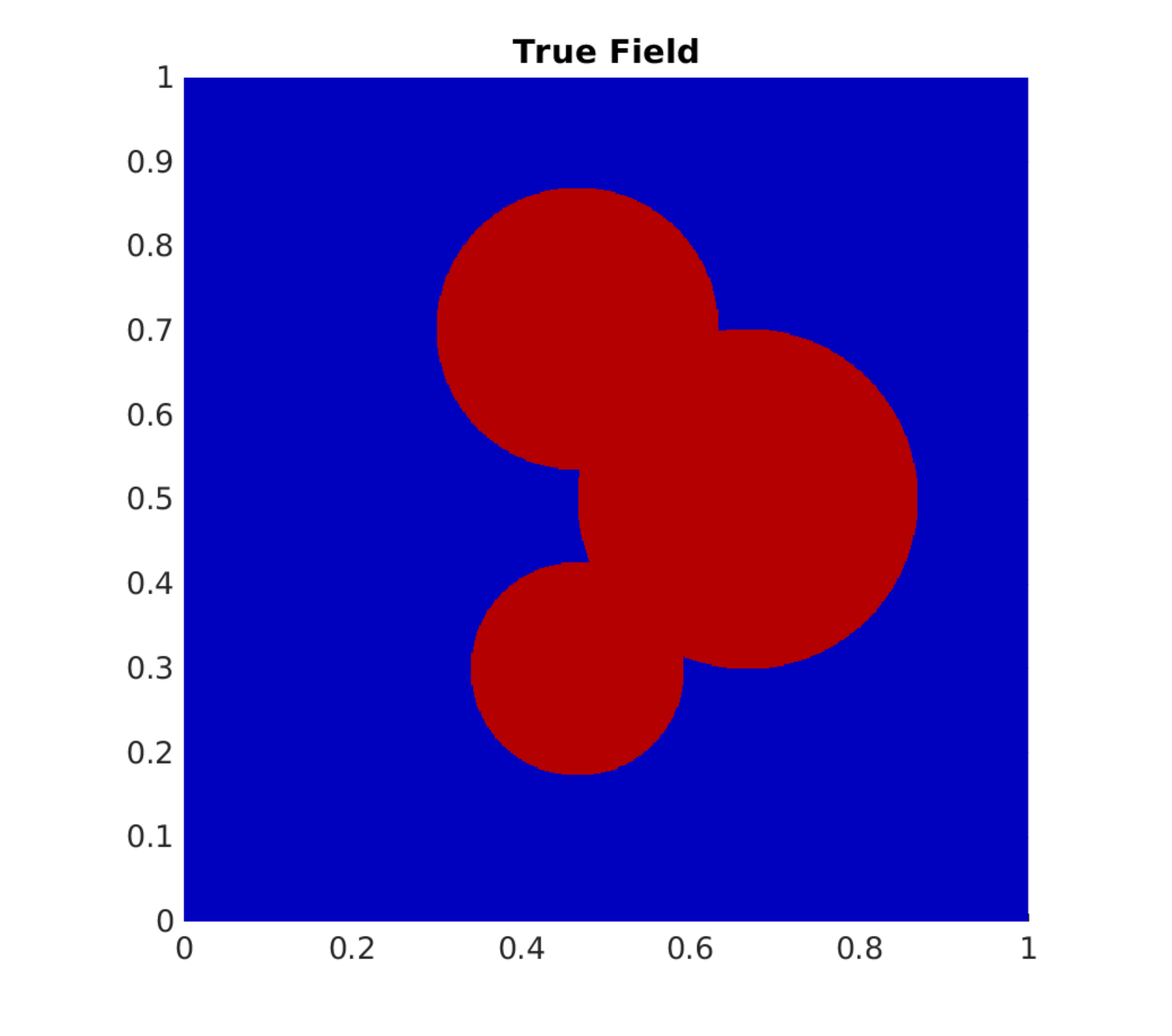}
        \includegraphics[width=0.32\textwidth]{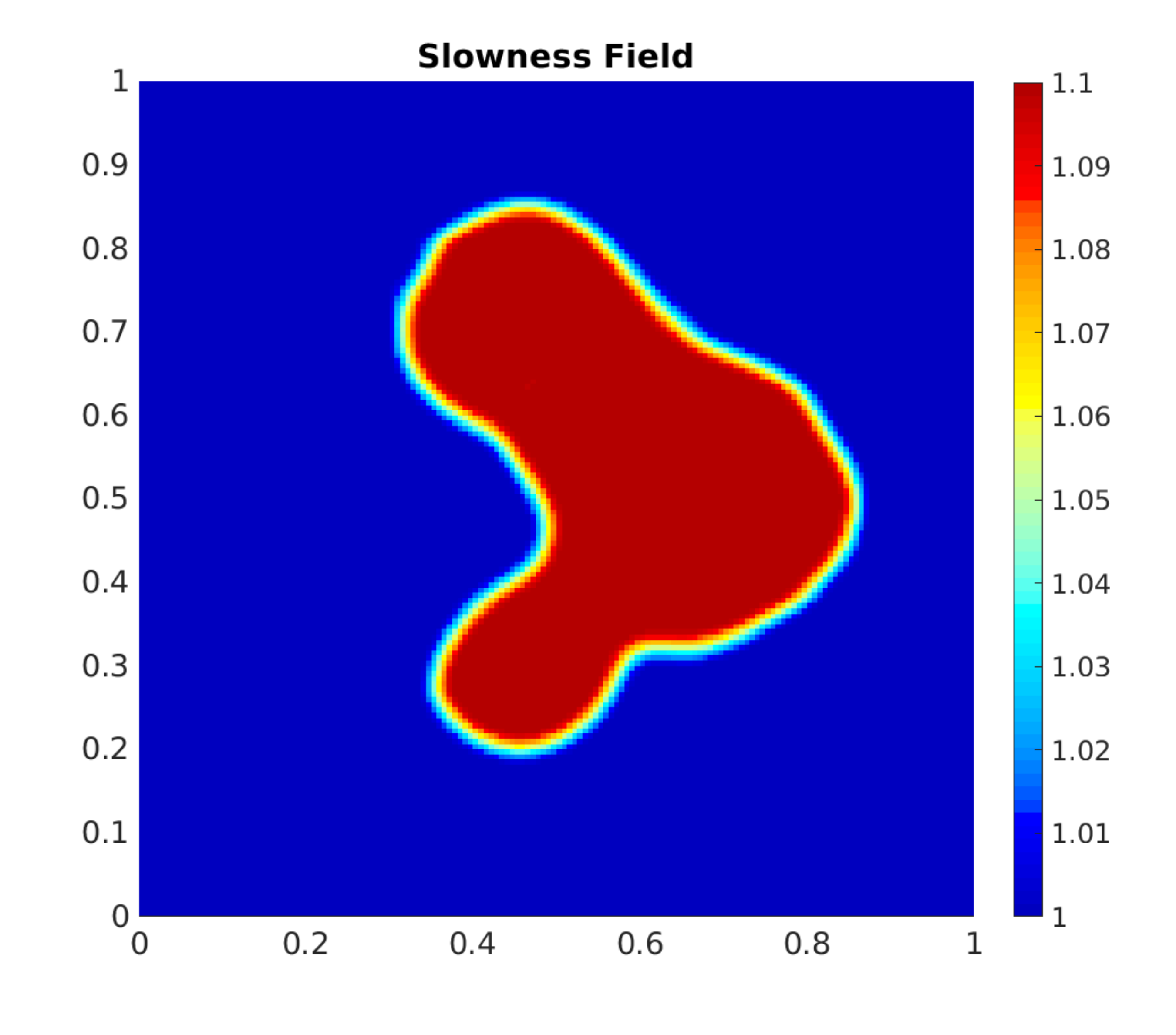}
        \includegraphics[width=0.32\textwidth]{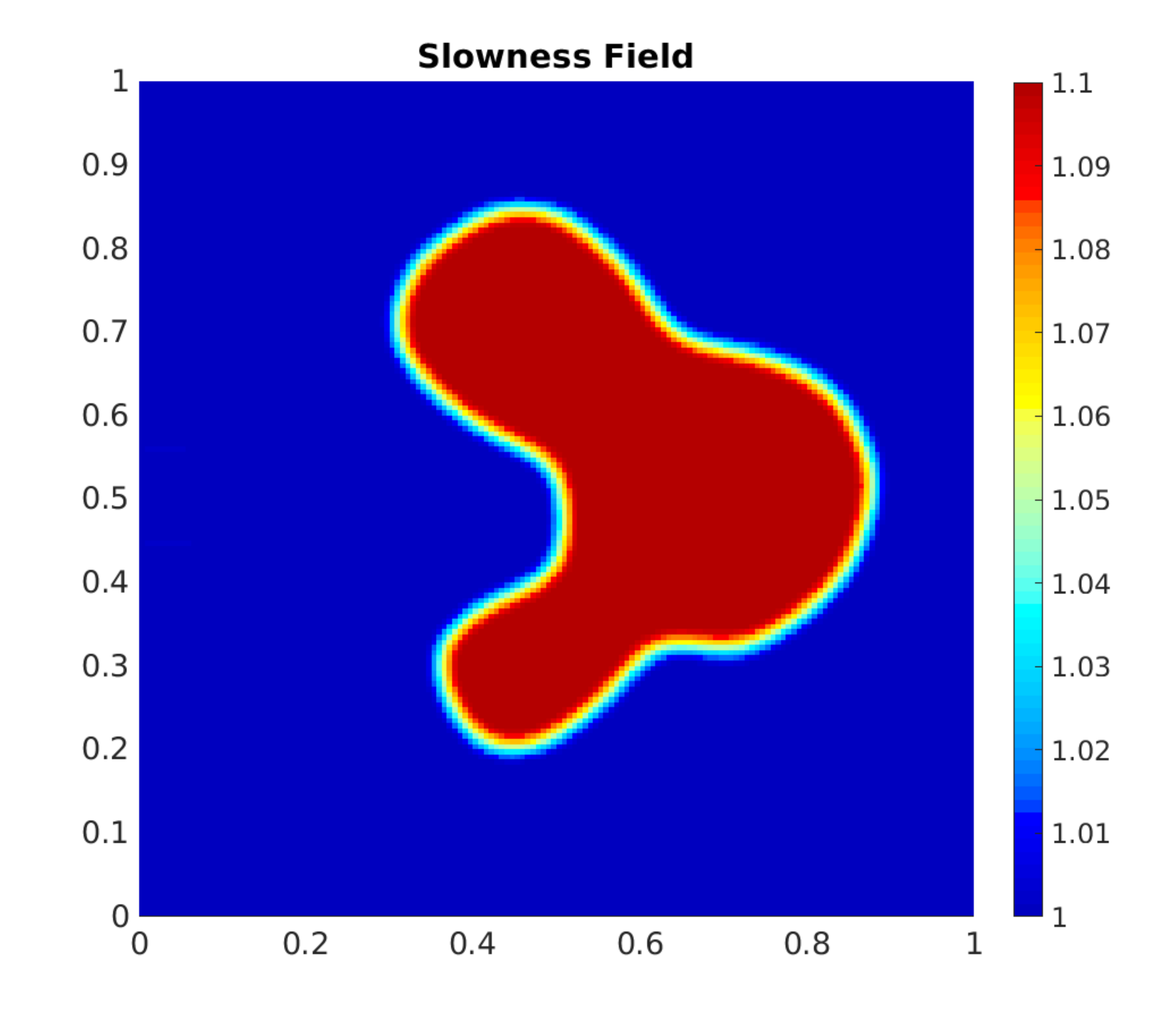}
        
        \caption{\label{fig:splodge}Arbitrary shape recovery. (a) true slowness field  $s_{\min}=1,\ s_{\max}=1.1$. (b) slowness field from random configuration. (c) slowness field from wells configuration }
\end{figure}
        
\begin{figure}[ht]
        \centering   
        \includegraphics[width=0.32\textwidth]{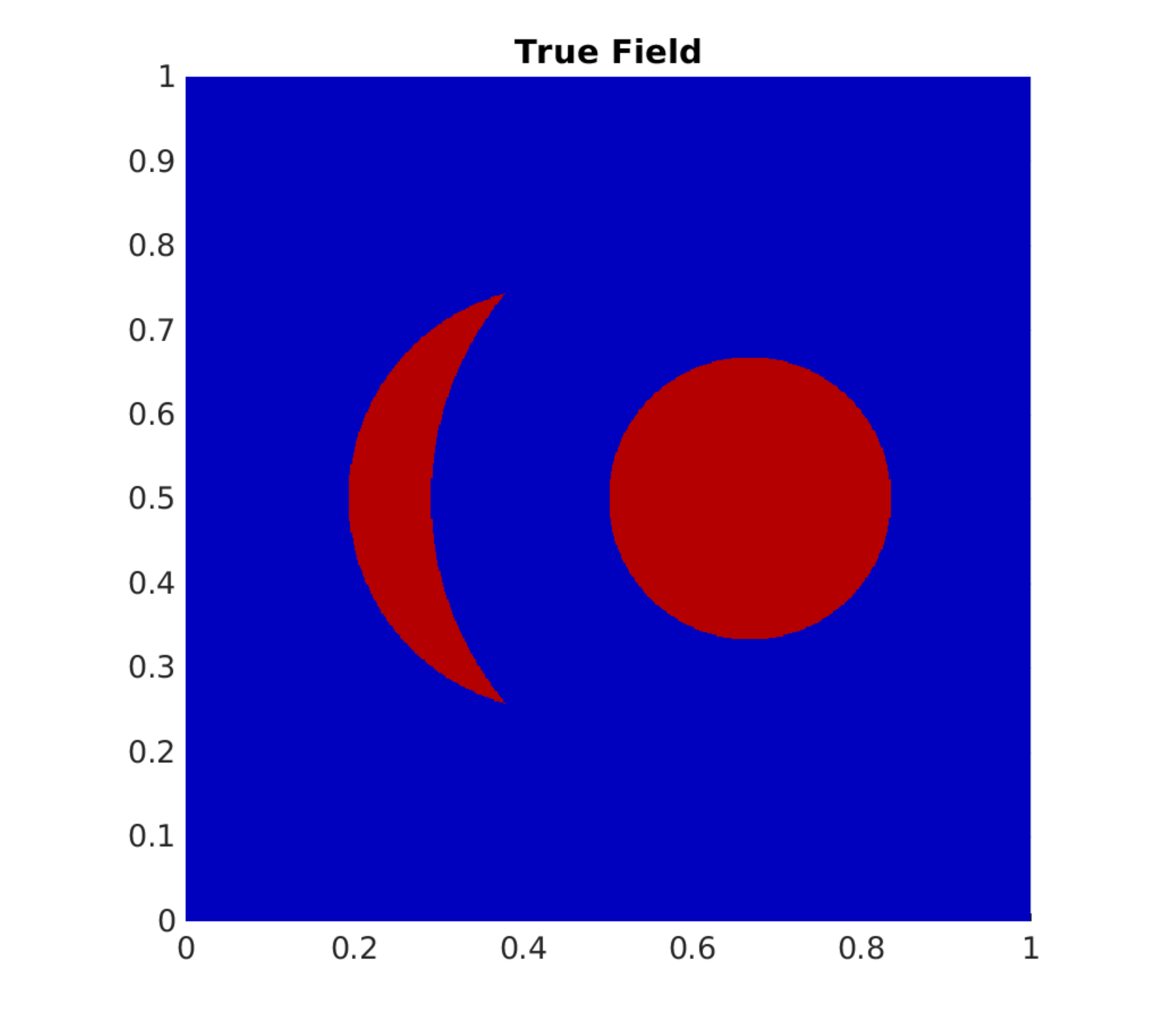}
        \includegraphics[width=0.32\textwidth]{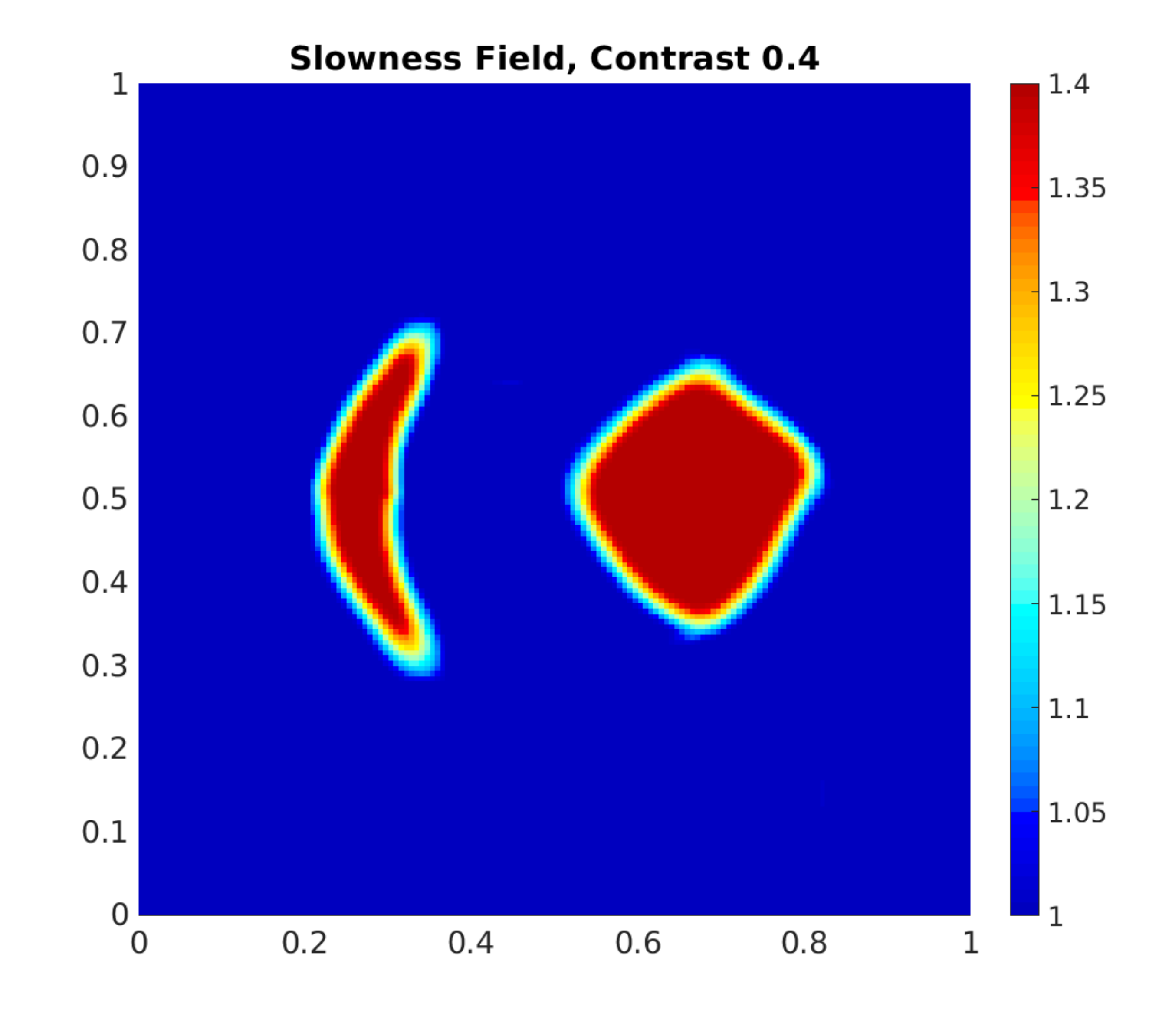}
        \includegraphics[width=0.32\textwidth]{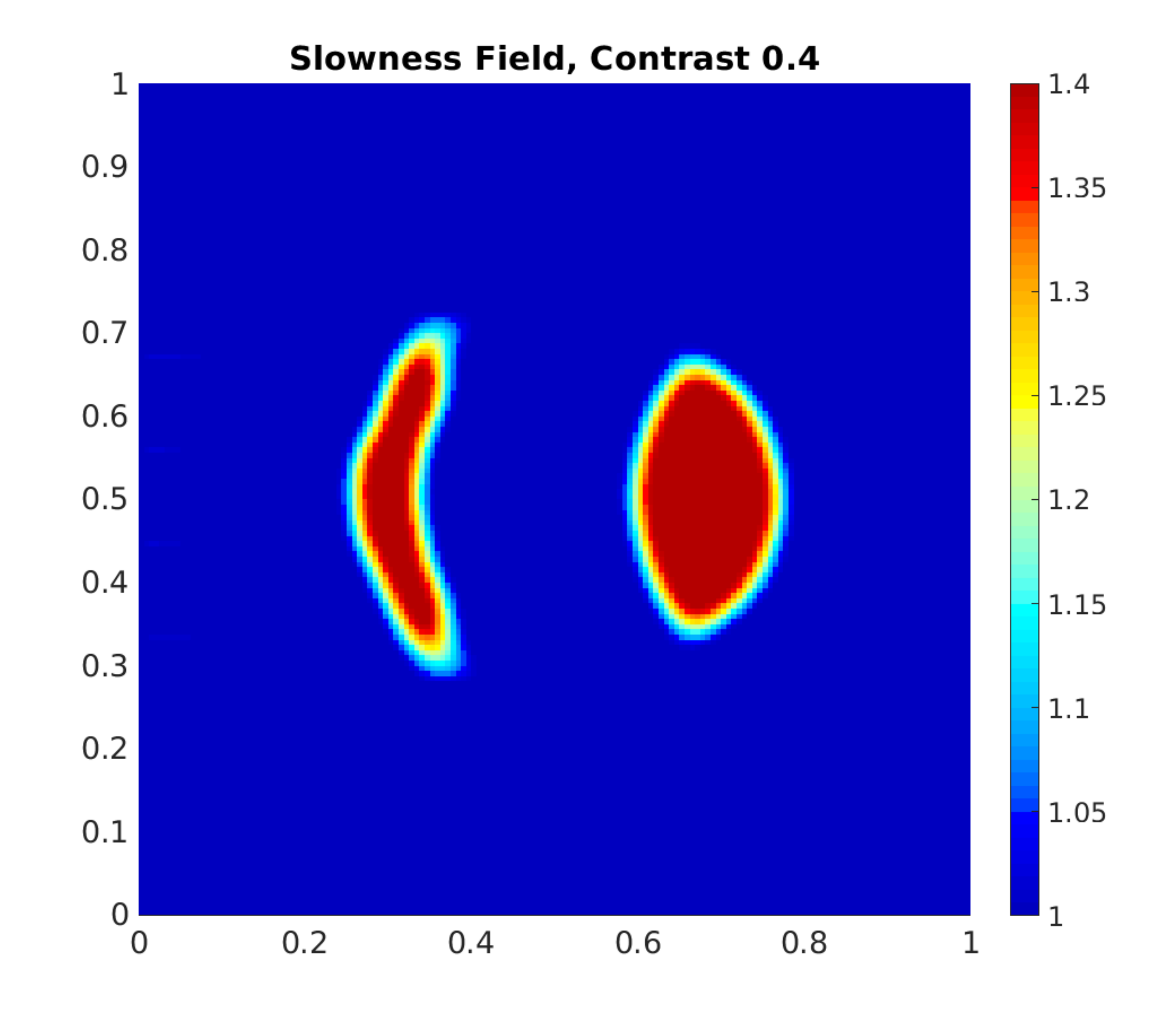}
        
        \caption{\label{fig:circban}Shielded disk recovery. (a) true slowness field $s_{\min}=1, \ s_{\max}=1.4$. (b) slowness field from random configuration. (c) slowness field from wells configuration}
\end{figure}

The different geometries provide a variety of challenges for the first traveltime binary recovery problem. The first recovery is of the circular disk, in \Fref{fig:circ2}. We see good recovery in both configurations for this simple geometry. For more complicated inclusions such as \Fref{fig:splodge} we similarly obtain reasonable recovery so long as the interfacial layer is thin relative to the lengthscale of geometric features. 

We observe two features of the underlying forward problem in \Fref{fig:band}. Firstly we see in locations where wave sources are near to interfaces we obtain geometric asymmetry. Secondly, the scheme is efficient if ray paths travel through relatively homogeneous structures and so we found the recovery in the wells configuration is quickly resolved.

The right angle of \Fref{fig:jagged} is well recovered, and we see the effects of the regularization on the boundary conditions. If one takes $\dD_\hbar = \p\Omega$ then interface positions match the truth (see \Fref{fig:jagged}(b)), but when the set $\{y = 0 \} \cup \{y=1\}$ has Neumann conditions (see \Fref{fig:jagged}(c)) we see interfaces meet domain boundaries at right angles. In this simulation it is vital for $\dD_\hbar \neq \emptyset$ to obtain a correct local minimizer.

The shielded disk of \Fref{fig:circban} performs well in the random configuration. In the wells configuration, the disk is shielded by the crescent inclusion and yet the inverse solver still distinguishes two objects and shapes are well recovered. We note this simulation performed better with contrast $s_{\min} =1$, $s_{\max}=1.4$. 

\begin{rem}  \label{rem:contrast}
There is a maximal contrast ratio of the binary slowness function for which first traveltime tomography is effective. Unless detectors are well placed (for example within every inclusion in the domain) then above this critical ratio, first hitting times will be (almost) independent of the contrast. We display traveltime fields in \Fref{fig:traveltimecontrast}, due to a source at $[0,1/2]$, and the (continuous approximation of the) shielded disk slowness field \Fref{fig:circban}(a), with a fixed binary value $s_{\min}=1$. For high contrasts (\Fref{fig:traveltimecontrast}(b) and (c)), the maximum traveltime occurs within the obstacle thus the wave's first hitting time at $[1,1/2]$ is interpreted as not having a ray path through the obstacle. This loss of information affects the recovery, see \Fref{fig:recovcontrast}. Impenetrable obstacles will be investigated in a forthcoming work.
\end{rem} 

\begin{figure}[ht]
 \centering   
         \includegraphics[width=0.32\textwidth]{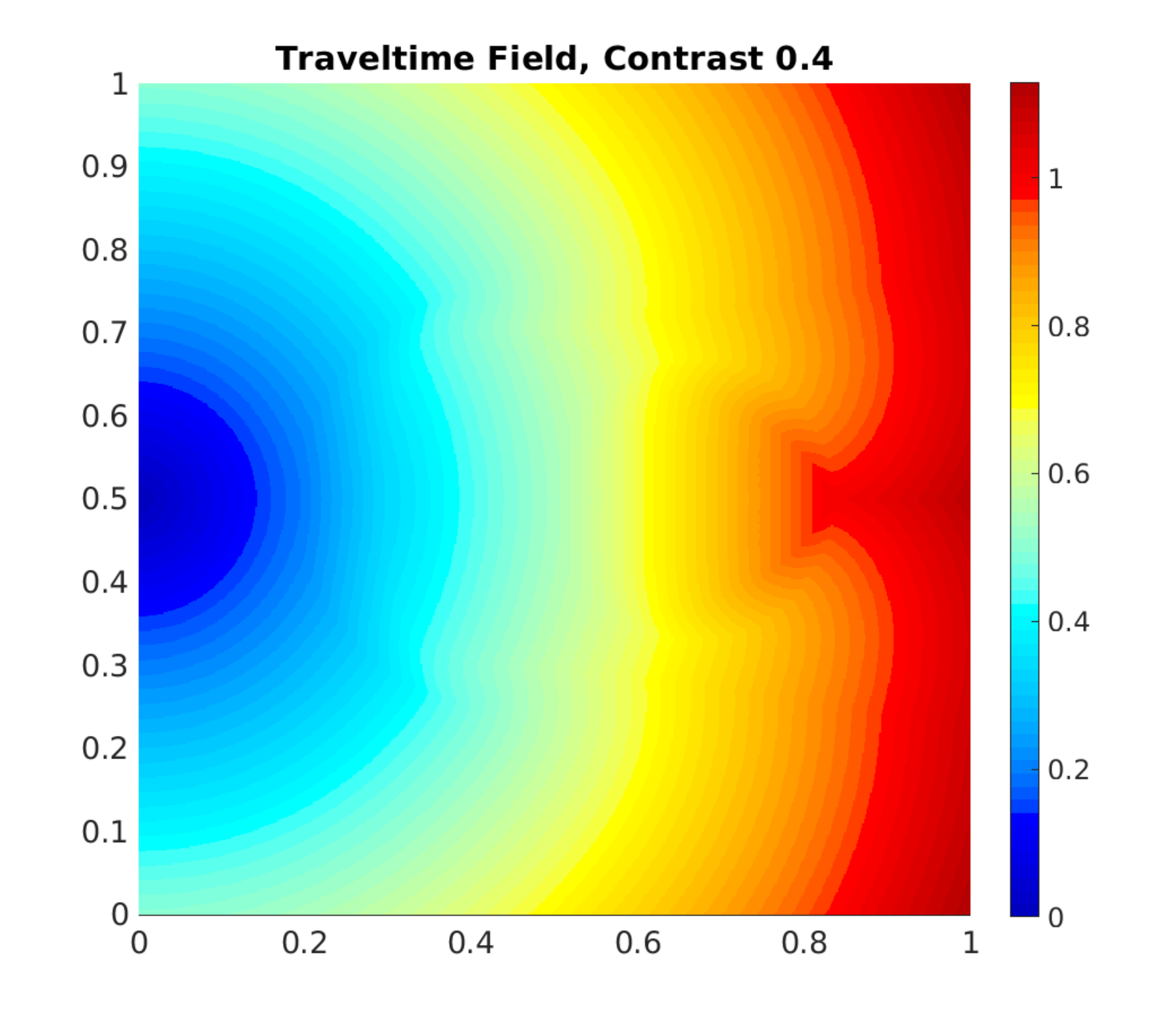}   
       \includegraphics[width=0.32\textwidth]{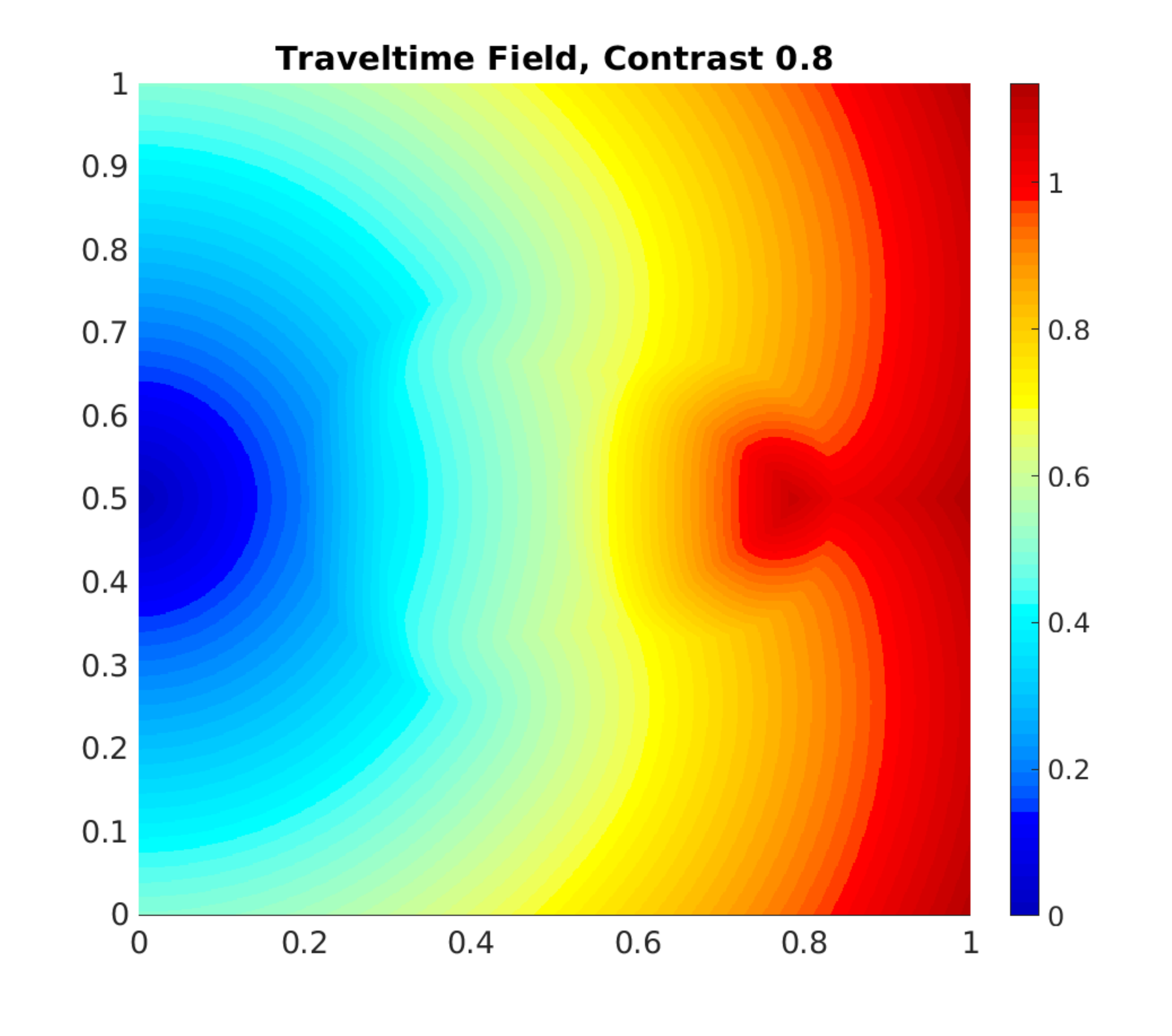}  
       \includegraphics[width=0.32\textwidth]{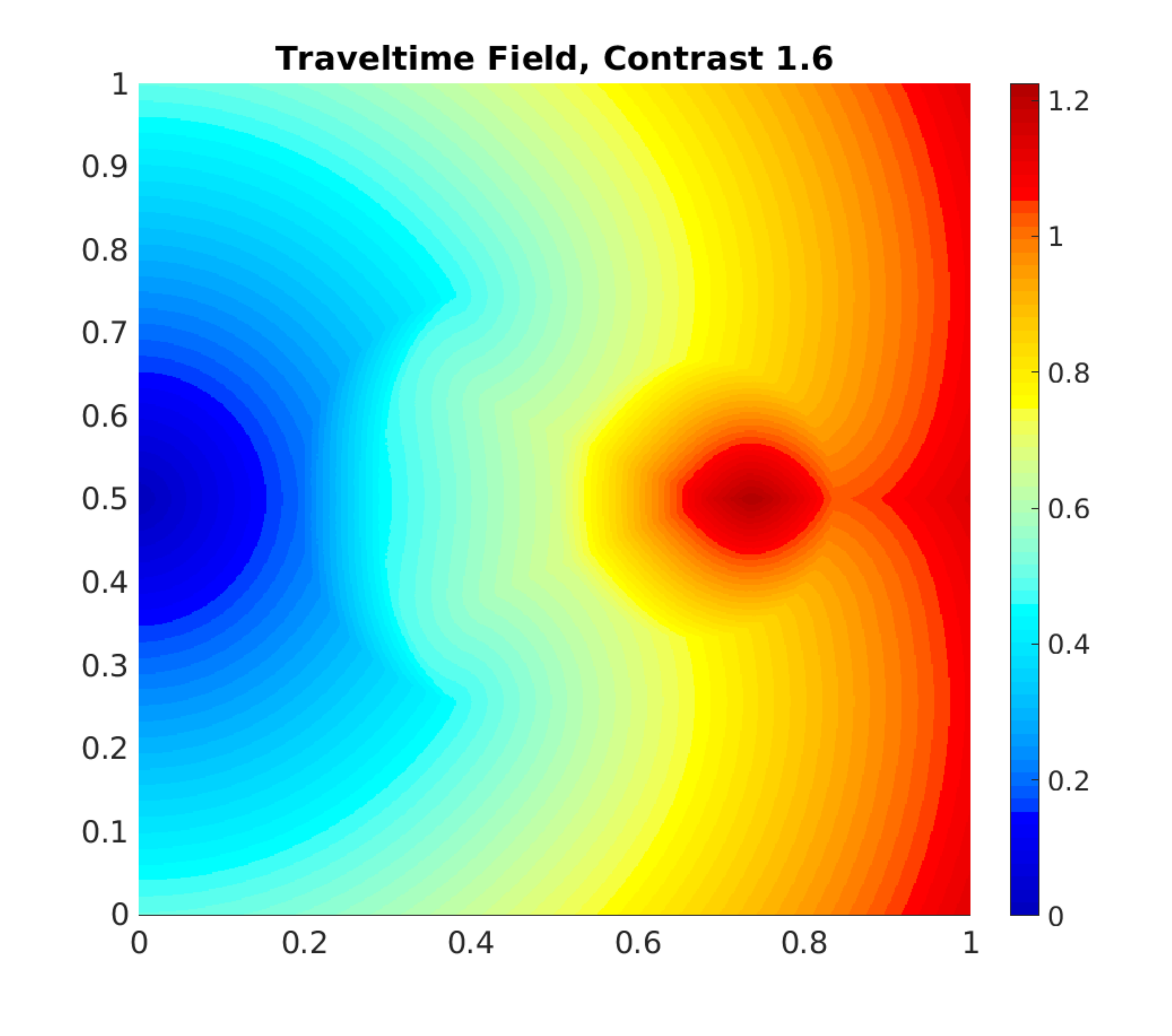}

        \caption{\label{fig:traveltimecontrast}Traveltimes for shielded disk. Source located at $[0,1/2]$ in the domain. (a) to (c): increasing contrast ($(s_{\max}-s_{\min})/s_{\min}$) values 0.4, 0.8, 1.6. Contrast $\geq 0.8$ has maximal traveltime in the domain.}
\end{figure}

\begin{figure}[ht]
 \centering 
      \includegraphics[width=0.32\textwidth]{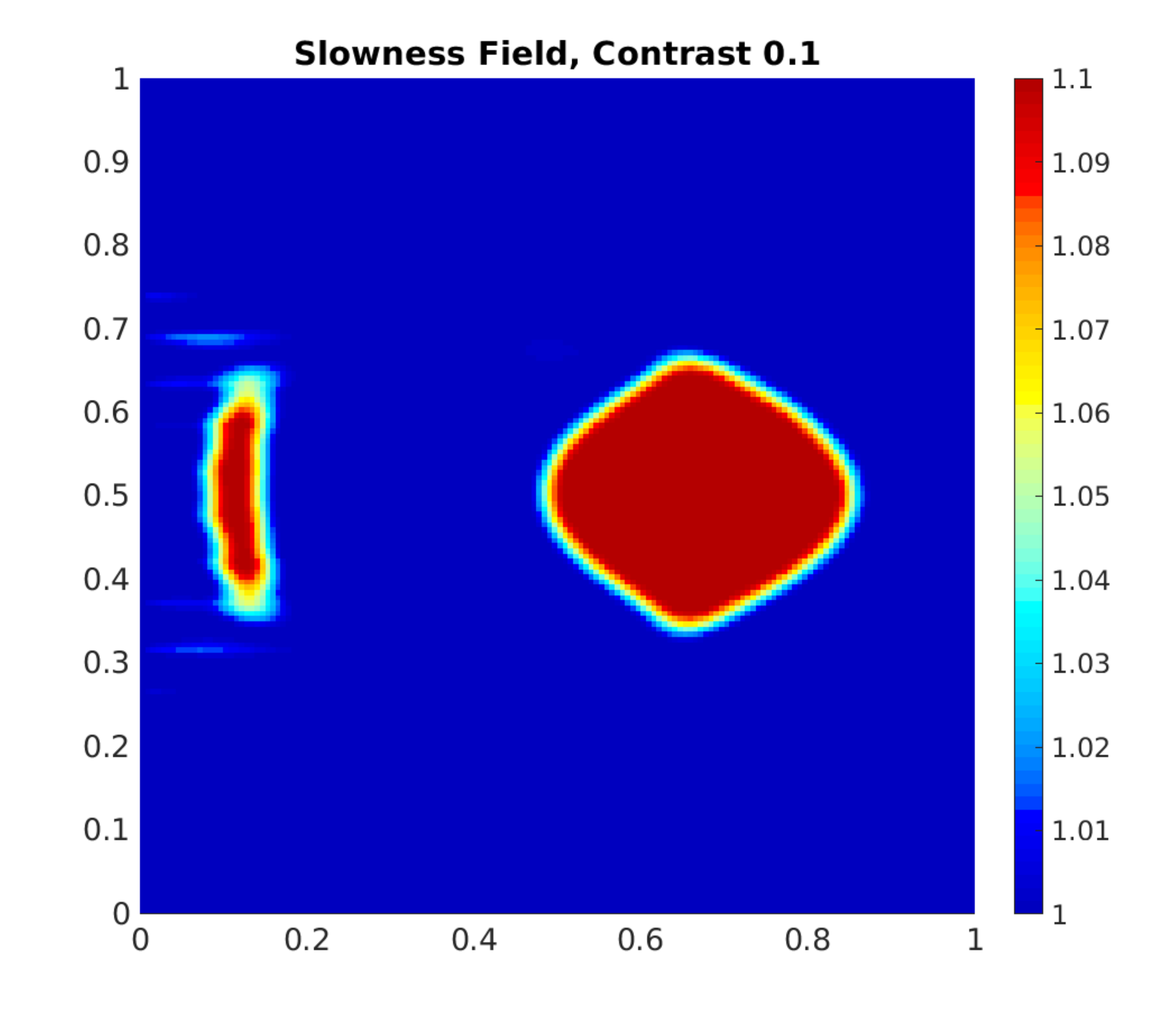}
 \includegraphics[width=0.32\textwidth]{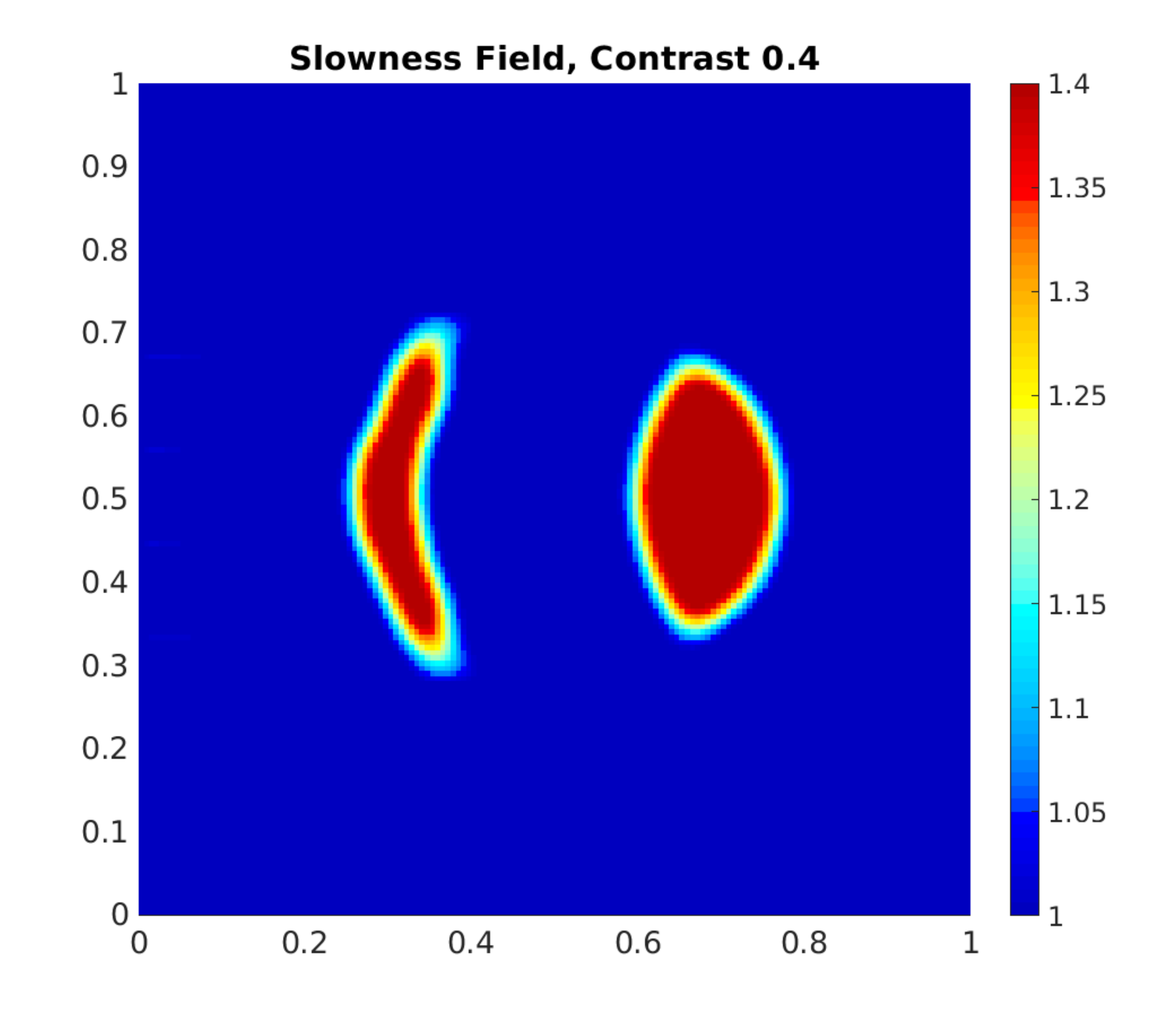}  
 \includegraphics[width=0.32\textwidth]{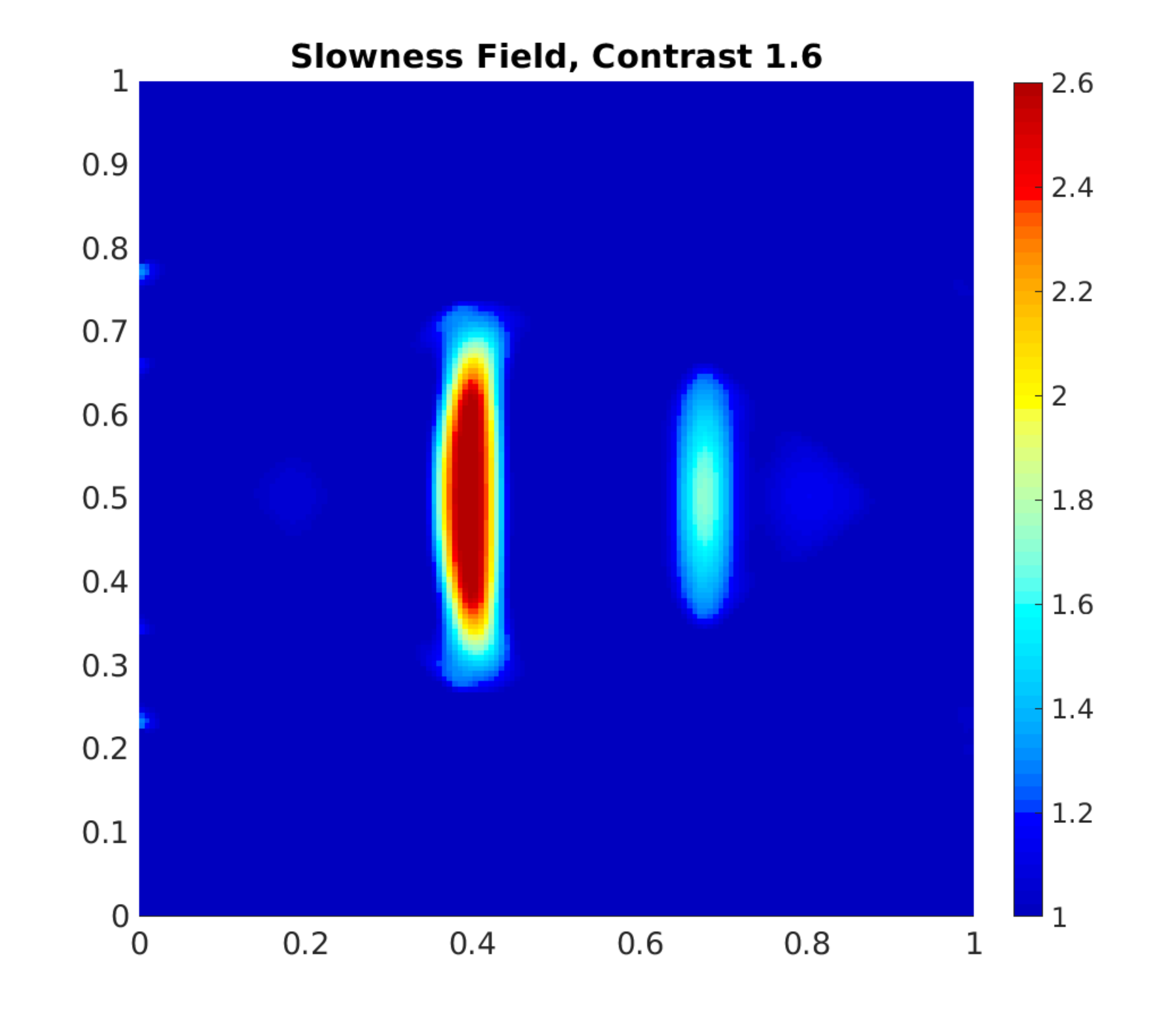}
 \caption{\label{fig:recovcontrast}Different contrast ratio for  shielded disk, with wells configuration and $s_{\min}=1$. (a) $s_{\max}= 1.1$, (b) $s_{\max}=1.4$, (c) $s_{\max} = 2.6$. Situations of too low, acceptable, and too high contrast respectively.  }
 
\end{figure}

\section{Conclusion}
\subsection{Closing remarks}
We have presented a technique for binary recovery based on the phase field methodology with an emphasis of the presence of an underlying mathematical theory.

We proved that the method ensures that the solutions exist, and that the forward problem remains well posed for them. We validated the technique with a $\Gamma-$convergence of the phase field regularization to a perimeter penalization technique. We have created a mixed formulation of the problem, and shown that minimizers of this solve the original problem. We have constructed a convergent discrete formulation relying on a monotone finite difference method for the forward problem and mixed finite element method for the inverse problem. Due to our careful treatment of the derivative we were able to use a descent algorithm and we demonstrated its effectiveness of recovery in many geometries and for different configurations of source -- receiver pairs, including one found in crosswell tomography.

\subsection{Outlook}
The framework we have set up suggests several developments worthy of future study. For example, investigation of a hierarchical model in which $s_{\max}$ is unknown, as  in \cite{LeuLi13}, or if $s_{\max}$ is known up to fluctuations, for example by setting  $s_{\max}(x) = \tilde s_{\max} + \delta f(x)$  for a continuous function $f$   to be determined with  $\delta \ll (\tilde s_{\max} - s_{\min})$. 

A natural extension is to investigate a piecewise constant slowness function formulated with a multi-obstacle potential. Modelling impenetrable obstacles ($s_{\max}\to \infty$) - as noted in Remark \ref{rem:contrast}, could also extend the applicability of the tomography. 

In practice, gravity acceleration data is also collected in geophysical surveys and one can perform a joint inversion for the slowness and subsurface density. It is a viable extension as the slowness and density are structurally related, and will share curves of discontinuity. Inversion is performed a weighted sum of the misfit functionals for slowness and gravity acceleration \cite{LiQia16}.

\section*{Acknowledgments}

The work of CME was partially supported by the Royal Society via a Wolfson Research Merit Award; the work of ORAD and CME by the EPSRC Programme Grant EQUIP.  
\section*{References}

\bibliographystyle{siam}


\appendix
\section{Proof of $\Gamma$ -- convergence}\label{sec:app}

We prove Theorem \ref{thm:gammaconv} by first proving a key result about the transition energy. 

\begin{applem}\label{lem:profsol}
There exists a solution to the minimization problem
 \[
 P^\gamma \coloneqq \stackrel[v \in {\mmm[V]}]{}{\inf} j^\gamma(z) =\stackrel[z \in {\mmm[V]}]{}{\inf} 
 \int_{\R}\frac{1}{2}\gamma (z'')^2 + \frac{1}{2}(z')^2+\Psi(z) \dd x.
\]
\end{applem}

\begin{proof}
For $\gamma>0$, we consider the Euler-Lagrange equations for $j^\gamma$. Due to the obstacle constraint, this can be written as a variational inequality. Find $z\in H^2(\R,[-1,1])$, such that
\[
 \gamma (z'',\eta'' - z'') + (z',\eta' - z') -(z,\eta-z)\geq 0, \qquad \forall\eta \in H^2(\R,[-1,1]).
\]
If we are considering solutions $z \in \mmm[V]\subset C^2(\R)$ then we have $z\in H^4(\R)$ and so by integrating by parts we obtain
\[
(\gamma z'''' - z'' -z,\eta-z)\geq 0,\qquad \forall \eta \in L^2(\R,[-1,1]).
\]
Hence for $z\in\mmm[V]$, we can relate $z$ satisfying the critical point, to solutions $Z$ of the following Fisher-Kolmogorov equation. Let $\delta^\gamma >0$, then find $Z \in C^2(\R)$ such that,  
\begin{equation}\label{eq:FKeqn}
 \gamma Z''''-Z''+\psi(Z)=0  \ \textrm{ on } (-\delta^\gamma,\delta^\gamma),\qquad \textrm{where }  \psi(s)\coloneqq -s,
\end{equation}
subject to the constraints and boundary conditions imposed by $\mmm[V]$:
\[\fl \indent
Z(x<-\delta^\gamma)=-1, \qquad Z(x>\delta^\gamma) = 1,   \qquad Z^{(k)}(-\delta^\gamma) =Z^{(k)}(\delta^\gamma)=0, \ \  \textrm{for } k=1,2.
\]
and $|Z|\leq 1$, and $Z'(x)\geq 0,\  \forall x \in (-\delta^\gamma,\delta^\gamma)$, $Z$ odd. To find this, we look at the characteristic equation: $\gamma \lambda^4 - \lambda^2 -1 = 0$. 
For $\gamma>0$ we have two real $\pm\lambda^\gamma_1$
and two imaginary roots  $\pm\lambda^\gamma_2i$.
\[
 \lambda^\gamma_1 = \sqrt{\frac{1}{2\gamma}\big(1+\sqrt{1+4\gamma}\big) },\qquad \lambda^\gamma_2 = \sqrt{\frac{1}{2\gamma}\big(\sqrt{1+4\gamma}-1\big) }.
\]
Thus, for all $\gamma>0$ we have a general form of solution
\[
 Z^\gamma(x) = C^\gamma_1e^{\lambda^\gamma_{1}x} + C^\gamma_2e^{-\lambda^\gamma_{1}x} + C^\gamma_3\cos(\lambda^\gamma_{2}x)+C^\gamma_4\sin(\lambda^\gamma_{2}x).
\]
The constants $C^{\gamma}_k$ depend on $\delta^\gamma$. A general odd solution must therefore be of the form $Z^\gamma_{odd}(x) = \tilde C^\gamma_1 \sinh(\lambda_1x) + \tilde C^\gamma_2 \sin(\lambda_2 x)$. Using the boundary conditions, one finds coefficients of the form:
\[
  \tilde C^\gamma_1 =\frac{(\lambda_2^\gamma)^2}{((\lambda_1^\gamma)^2+(\lambda_2^\gamma)^2)\sinh(\lambda^\gamma_1\delta^\gamma)}, \qquad \tilde C^\gamma_2 =\frac{(\lambda_1^\gamma)^2}{((\lambda_1^\gamma)^2+(\lambda_2^\gamma)^2)\sin(\lambda^\gamma_2\delta^\gamma)}.
\]
We require solutions to be $C^2(\R)$, so we seek $\delta^\gamma$ so that $(U^\gamma)''(\delta^\gamma) =0$. This holds where 
\begin{equation}\label{eq:delta}
  \lambda^\gamma_2 \tan(\lambda^\gamma_2 \delta^\gamma)= -\lambda^\gamma_1\tanh(\lambda^\gamma_1\delta^\gamma).
\end{equation}
The first positive solution of \eref{eq:delta} is $\delta^\gamma = \frac{\pi}{2\lambda_2}+\alpha$, for some $\alpha \in (0,\frac{\pi}{2\lambda_2})$, and it is unique over the interval range.

To summarize we have a found a unique strictly monotonic solution for the 1D extended obstacle Fisher-Kolmogorov equation. It is bounded on $[-1,1]$ and for $\gamma>0$, can be extended to a function $z^\gamma\in C^2(\R; [-1,1])$ by 
\[
 z^\gamma(t) \coloneqq \cases{
                      -1,&  if $ t<- \delta^\gamma,$\\
                      \tilde C^\gamma_1 \sinh(\lambda^\gamma_1 t ) + \tilde C^\gamma_2 \sin(\lambda^\gamma_2 t), &  if $ -\delta^\gamma \leq t \leq \delta^\gamma,$\\
                      1, &  if  $t>\delta^\gamma$.
                     }
\]
and so $z^\gamma \in \mmm[V]$. 
\end{proof}

\begin{apprem}
 As $\gamma \to 0$, we have $\delta^\gamma \searrow \frac{\pi}{2}$ and $\tilde C^\gamma_2 \to 1$ and $\tilde C^\gamma_1 \to 0$ and therefore $z^\gamma(t)$ converges to $\sin(t)$ on $(-\frac{\pi}{2},\frac{\pi}{2})$. This is as expected as the fourth order problem reduces to the second order problem seen in \cite{BloEll91}.
\end{apprem}

To prove Theorem \ref{thm:gammaconv}, we must prove two inequalities. In \cite{HilPelSch02}, the authors have a complete proof for the double well potential, so we provide detail of where the double obstacle theory differs.
\begin{applem}(liminf inequality) 
 Given a sequence $\{u_\e\}$ with $u_\e \to u$ as $\e \to 0$ strongly in $L^1(\Omega)$, then 
 \[
 \mmm[J]^\gamma_0\leq \stackrel[\e\to0]{}{\lim\inf} \mmm[J]^\gamma_\e(u_\e).
  \]
\end{applem}
\begin{proof}
This follows immediately from the proof of \cite[Proposition 3.2]{HilPelSch02}, which does not explicitly rely upon the double well or obstacle, we only require the profile solution ($z^\gamma$) is odd, as shown in Lemma \ref{lem:profsol}.
\end{proof}

\begin{applem}(limsup inequality) 
 For any $u \in L^1(\Omega)$, there exists a sequence $\{u_\e\}$ such that:
 \begin{enumerate}
  \item $u_\e \to u$ as $\e \to 0$ strongly in $L^1(\Omega)$.
  \item $\mmm[J]^\gamma_0(u) \geq \stackrel[\e\to0]{}{\lim\sup}\mmm[J]^\gamma_\e(u_\e)$.
 \end{enumerate}
\end{applem}
\begin{proof}

 We follow the proof of \cite{HilPelSch02}: Let $u \in L^1(\Omega)$, we must construct a sequence $\{u_\e\}$ such that $\lim_{\e\to0} u_\e = u$ in $L^1(\Omega)$ and 
 \[
  \stackrel[\e\to0]{}{\lim\sup} \mmm[J]^\gamma_\e ( u_\e ) \leq \mmm[J]^\gamma_0(u).
 \]
Due to constructions in \cite{Mod87} there exists a set $D \subset\R^d$ open and bounded with $\p D \in C^\infty$ and $\mmm[H]^{d-1} ( \p D \cap \p\Omega) = 0$, such that 
\[
 u = \chi_D - \chi_{\R^d \setminus D}.
\] 
Now, Let $U\in \mmm[V]$ minimize the functional $P^\gamma$, such a function exists due to Lemma \ref{lem:profsol}. and so $ P^\gamma = j^\gamma(U)$. With the specific form for $u$, (and $|\nabla\cdot|$ in the sense of total variation), we rewrite the limit $\mmm[J]^\gamma_0(u)$ 
\begin{equation}\label{eq:J0}
 \mmm[J]^\gamma_0(u) = \frac{1}{2}P^\gamma \int_\Omega |\nabla u| = P^\gamma \int_\Omega |\nabla \chi_D| = j^\gamma(U) \mmm[H]^{d-1}(\p D \cap \Omega).
\end{equation}
Let $d$ be the signed distance function to $\p D$, 
\[
 d(x) = \cases{
         \inf_{y \in \p D} |x-y|, &  if $ x \in D$,\\
         -\inf_{y \in \p D} |x-y|, &  if $ x \not\in D$.
        }
\]
There exists a neighbourhood $N_h$ of width $h$ to $\p D$ where $d$ is $C^2(N_h)$, and we define a function $\eta\colon \bar \Omega \to \R$
\[
\eta(x) =   d(x), \ \textrm{ if }  x \in N_h, \qquad  |\eta(x)|\geq h  \ \textrm{ if } x \not \in N_h.\\ 
\]
We may extend outside of $N_h$ to ensure $\eta \in C^2(\bar \Omega)$.
As mentioned in Remark  \ref{rem:gamma1}, a useful rescaling is defined through the following sequence:
\[
 u_\e(x) = U\left(\frac{\eta(x)}{\e}\right), \qquad x \in \Omega.
\]
We see that $u_\e\in C^2(\bar \Omega)$ and $\lim_{\e \to 0} u_\e (x)= u(x)$ $\forall x \in \Omega$ in $L^1(\Omega)$ by construction of $U \in \mmm[V]$. By dominated convergence theorem this converges in $L^1(\Omega)$. We must now show that this sequence provides the limsup inequality of the Lemma. The chain rule leads to 
\begin{eqnarray}\fl \indent
 \mmm[J]^\gamma_\e(u_\e) = \frac{1}{\e} \left(\int_ {\Omega \cap N_h} + \int_{\Omega\setminus N_h} \right)& \frac{1}{2}\gamma \Big|U''\left(\frac{\eta}{\e}\right)|\nabla \eta|^2 + \e U'\left(\frac{\eta}{\e}\right)\lap \eta\Big|^2 \nonumber\\
& +\frac{1}{2} \Big|U'\left(\frac{\eta}{\e}\right)\Big|^2|\nabla \eta|^2 + \Psi \Big(U \left(\frac{\eta}{\e}\right)\Big) \dd x. \label{eq:Jgamma}
\end{eqnarray}
By construction, $U(z) = 1\textrm{ or } -1 $ for $|z| > \delta^\gamma$, and so for $\e$ small enough, the integral in \eref{eq:Jgamma} over $\Omega \setminus N_h$ is $0$ and so this term is done. For the other integral, by construction $\eta = d$ here and so $|\nabla \eta|=1$, thus $|\lap \eta|\leq C_\eta$. 
\begin{eqnarray*}
|U''(z) + \e C_\eta U'(z)|^2 &= |U''(z)|^2 + 2\e C_\eta |U'(z) U''(z)|+  C^2_\eta\e^2|U''(z)|^2, 
\end{eqnarray*}
and we apply Young's inequality to the second term with weight $\frac{\nu}{2\gamma}$, where $\nu>0$: 
\begin{eqnarray*}\fl \indent
|U''(z) + \e C_\eta U'(z)|^2&\leq |U''(z)|^2+\frac{\nu}{\gamma}|U'(z)|^2+ \Big(\frac{\gamma(2\e C_\eta)^2}{\nu}+C^2_\eta\e^2 \Big)|U''(z)|^2 \\
&\leq (1+\nu)|U''(z)|^2+\frac{\nu}{\gamma}|U'(z)|^2 + C(\nu)\e^2,
\end{eqnarray*}
for constant $C(\nu)$. Notice that after passing to the limit $\e \to 0$, we could take $\nu>0$ as arbitrarily small without blowup. We may now bound the integral in \eref{eq:Jgamma} by 
\begin{eqnarray*}\fl 
 \mmm[J]^\gamma_\e \leq\ \frac{(1+\nu)}{\e} \int_{\Omega \cap N_h} \Big(\frac{1}{2} \gamma |U''\left(\frac{\eta}{\e}\right)|^2 + \frac{1}{2}|U'\left(\frac{\eta}{\e}\right)|^2 + \Psi(U(\omega(x))) \Big) \dd x+ C(\nu)\e\\ \fl 
 \phantom{\mmm[J]^\gamma_\e}=\  (1+\nu) \int_{\Omega \cap N_h} \Big(\frac{1}{2} \gamma |U''(\omega(x))|^2 + \frac{1}{2}|U'(\omega(x))|^2 + \Psi(U(\omega(x))) \Big) |\nabla \omega(x)| \dd x + C(\nu)\e,\\
 \end{eqnarray*}
 where $\omega(x) = \frac{d(x)}{\e}$. Note we have used $|\nabla \omega(x)| = \frac{1}{\e}$, then using the co-area formula \cite[Theorem 3.2.12]{book:Fed96} on $t = \omega(x)$ we obtain
 \begin{eqnarray*}\fl 
 \mmm[J]^\gamma_\e \ \leq& (1+\nu) \int_\R\int_{\omega^{-1}(t)\cap\Omega \cap N_h} \Big(\frac{1}{2} \gamma |U''(\omega(x))|^2 + \frac{1}{2}|U'(\omega(x))|^2 + \Psi(U(\omega(x))) \Big) |\nabla \omega(x)| \dd x\\ \fl
 &\ + C(\nu)\e.
\end{eqnarray*}
 Now we wish to rewrite these integrals. Firstly $\omega^{-1}(t) = \{ x \ | \ \frac{d(x)}{\e} = t \}$, so
\[
 \omega^{-1}(t) \cap N_h = \cases{
  \{ x \ | \ \frac{d(x)}{\e} = t \}, & if $t\leq\frac{h}{\e},$ \\
  \emptyset, & if  $t>\frac{h}{\e}$. 
 }
\]
We may now restate the limits.
\begin{eqnarray*}\fl 
 \mmm[J]^\gamma_\e  \leq\ (1+\nu) \int^{\frac{h}{\e}}_{- \frac{h}{\e}}\int_{d(x) = \e t} \Big(\frac{1}{2} \gamma |U''(t)|^2 + \frac{1}{2}|U'(t)|^2 + \Psi(U(t)) \Big) \dd \mmm[H]^{d-1}(x) \dd t \\ \fl 
 \phantom{\mmm[J]^\gamma_\e \leq} \ + C(\nu)\e\\ \fl 
 \phantom{\mmm[J]^\gamma_\e }=\ (1+\nu) \int^{\frac{h}{\e}}_{- \frac{h}{\e}}\Big(\frac{1}{2} \gamma |U''(t)|^2 + \frac{1}{2}|U'(t)|^2 + \Psi(U(t)) \Big) \dd \mmm[H]^{d-1}\{x\ |\ d(x) = \e t\} \dd t \\ \fl
 \phantom{\mmm[J]^\gamma_\e \leq}\ + C(\nu)\e.
\end{eqnarray*}
As $\e \to 0$, this converges to 
\[
 (1+\nu) j^\gamma(U) \dd \mmm[H]^{d-1}\{x\ |\ d(x) = 0\} = (1+\nu) P^\gamma \mmm[H]^{d-1}(\p D \cap \Omega) .
\]
Therefore we have shown that, in view of \eref{eq:J0},
\[
 \stackrel[\e\to0]{}{\lim\sup}(\mmm[J]^\gamma_\e(u_\e)) \leq (1+\nu)P^\gamma \mmm[H]^{d-1}(\p D \cap \Omega) = (1+\nu) \mmm[J]^\gamma_0(u),
\]
where the choice of $\nu$ may be arbitrarily small. Hence the limsup inequality is satisfied
\end{proof}

With both inequalities established, the proof of Theorem \ref{thm:gammaconv} is complete.  

\qed
\end{document}